\documentclass[a4paper,12pt]{article}
\usepackage[letter paper, margin=1 in]{geometry}
\usepackage[round]{natbib}

\usepackage{authblk}
\usepackage[utf8]{inputenc} 
\usepackage[T1]{fontenc}    
\usepackage[hidelinks,colorlinks=true,linkcolor=blue,citecolor=blue]{hyperref}
\usepackage{url}            
\usepackage{booktabs}       
\usepackage{algorithm}
\usepackage{algorithmic}
\usepackage[algo2e,ruled,vlined,linesnumbered]{algorithm2e}

\usepackage{amsfonts,amsmath,amssymb,amsthm}      
\usepackage{bbm}
\usepackage{nicefrac}       
\usepackage{microtype}      
\usepackage[capitalize,noabbrev]{cleveref}
\usepackage{graphicx}
\usepackage{wrapfig}
\usepackage{subcaption}
\usepackage{xcolor}
\usepackage{comment}
\usepackage{empheq}
\usepackage[most]{tcolorbox}

\newcommand{\mbb}{\mathbb}
\newcommand{\mbf}{\mathbf}
\newcommand{\mcl}{\mathcal}
\newcommand{\bs}{\boldsymbol}

\newcommand{\f}{\frac}

\newcommand{\T}{\textnormal}
\newcommand{\x}{\mathbf{x}}
\newcommand{\X}{\bs{\mathcal{X}}}
\newcommand{\D}{\bs{\mathcal{D}}}

\newcommand{\qpots}{$q\texttt{POTS}$}
\DeclareMathOperator*{\argmax}{arg,max}
\newcommand{\mP}{\mbb{P}r}

\theoremstyle{plain}
\newtheorem{theorem}{Theorem}[section]

\newtheorem{lemma}[theorem]{Lemma}

\theoremstyle{definition}

\newtheorem{assumption}[theorem]{Assumption}
\theoremstyle{remark}

\newtcbox{\mymath}[1][]{%
    nobeforeafter, math upper, tcbox raise base,
    enhanced, colframe=gray!30!black,
    colback=gray!30, boxrule=1pt,
    #1}
\tcbset{highlight math style={boxsep=5mm,colback=gray!30!red!30!white}}

\title{$q\texttt{POTS}$: Efficient Batch Multiobjective Bayesian Optimization via Pareto Optimal Thompson Sampling~\footnote{Accepted to the 28th International Conference on Artificial Intelligence and Statistics (AISTATS) 2025}}
\author[1,2]{Ashwin Renganathan}
\author[1,2]{Kade Carlson}
\affil[1]{Aerospace Engineering, The Pennsylvania State University, University Park, PA, 16802}
\affil[2]{Penn State Institute of Computational and Data Sciences, University Park, PA, 16802}

\begin{document}

\date{}
\maketitle
%

%

\begin{abstract}
     Classical evolutionary approaches for multiobjective optimization are quite accurate but incur a lot of queries to the objectives; this can be prohibitive when objectives are expensive oracles. A sample-efficient approach to solving multiobjective optimization is via Gaussian process (GP) surrogates and Bayesian optimization (BO). Multiobjective Bayesian optimization (MOBO) involves the construction of an acquisition function which is optimized to acquire new observation candidates sequentially. This ``inner'' optimization can be hard due to various reasons: acquisition functions being nonconvex, nondifferentiable and/or unavailable in analytical form; batch sampling usually exacerbates these problems and the success of MOBO heavily relies on this inner optimization. This, ultimately, affects their sample efficiency. To overcome these challenges, we propose a Thompson sampling (TS) based approach ($q\texttt{POTS}$). Whereas TS chooses candidates according to the probability that they are optimal, $q\texttt{POTS}$ chooses candidates according to the probability that they are Pareto optimal. Instead of a hard acquisition function optimization, $q\texttt{POTS}~$ solves a cheap multiobjective optimization on the GP posteriors with evolutionary approaches. This way we get the best of both worlds: accuracy of evolutionary approaches and sample-efficiency of MOBO. New candidates are chosen on the posterior GP Pareto frontier according to a maximin distance criterion. $q\texttt{POTS}~$ is endowed with theoretical guarantees, a natural exploration-exploitation trade-off, and superior empirical performance.
\end{abstract}

\section{Introduction}
 Mathematically, we consider the simultaneous constrained optimization of $K$ objectives 
\begin{equation} 
    \begin{split}
    \max_{\x \in \X}~&\{f_1(\x), \ldots, f_K(\x) \} \\    
    \T{s.t.} ~& c_i(\x) = 0,~i\in \mcl{E} \\
     ~& c_i(\x) \geq 0,~ i \in \mcl{I}
    \end{split}
    \label{eqn:main_problem}
    \end{equation}
where $\x \in \X \subset \mbb{R}^d$ is the design variable, $f_k: \X \rightarrow \mbb{R},~\forall k=1,\ldots,K$, are expensive zeroth-order oracles (i.e., no derivative information). $c_i : \X \rightarrow \mbb{R}$ are nonlinear constraints, also expensive zeroth-order oracles, where $\mcl{E}$ and $\mcl{I}$ are the sets of indices of equality and inequality constraints, respectively. $\X$ is the domain of $f$ and $c$. We are interested in identifying a \emph{Pareto} optimal set of solutions, which ``Pareto dominates'' all other points. Let $\bs{f(\x)}=[f_1(\x),\ldots,f_K(\x)]^\top$; when a solution $\bs{f(\x)}$ Pareto dominates another solution $\bs{f(\x')}$, then $f_k(\x) \geq f_k(\x'),~\forall k=1,\ldots,K$ and $ \exists k \in [K]$ such that $f_k(\x) > f_k(\x')$. We write Pareto dominance as $\bs{f(\x)} \succ \bs{f(\x')}$. The set $\mcl{Y}^* = \{\bs{f(\x)} ~:~ \nexists \x' \in \X: \bs{f(\x') \succ \bs{f(\x)}} \}$ is called the \emph{Pareto frontier} and the set $\X^* = \{ \x \in \X ~:~ \bs{f(\x)} \in \mcl{Y}^*\}$ is called the \emph{Pareto set}.

 We are interested in sample-efficient (that is, minimal queries to the oracles) approaches to solving our problem. Evolutionary algorithms (EA) such as the nondominated sorting genetic algorithm (NSGA)~\cite{deb2000fast, deb2002fast} are a popular choice for solving multiobjective problems (see \cite{zitzler2000comparison} for a review). However, EA based approaches are known to be not sample efficient, which could be prohibitive in real-world applications with expensive oracles. A common approach to solving this problem under such circumstances is to use Bayesian optimization (BO) with Gaussian process (GP) surrogate models for each objective. The idea is to fit a GP model for the objectives using some initial observations $\D^i_n = \{(\x_j, y_j),~j=1,\ldots,n\}$ (where we denote $y_j = f_i(\x_j)$) and construct an appropriate \emph{acquisition} function that quantifies the utility of a candidate point $\x$. Then, an inner optimization problem is solved to choose new point(s) that maximize the acquisition function. The surrogate models are updated with observations at the new points, and this process is repeated until a suitable convergence criterion is met. We provide a formal introduction in \Cref{sec:bo}. BO is very popular in the single objective ($K=1$) case; for instance, the expected improvement (EI)~\cite{jones1998efficient} and probability of improvement (PI)~\cite{mockus1978application,jones2001taxonomy} quantify the probabilistic improvement offered by a candidate point over the ``best'' observed point so far. The GP upper confidence bound (UCB)~\cite{srinivas2009gaussian} makes a conservative estimate of the maximum using the current GP estimate of the function. Other acquisition functions, for $K=1$, include entropy-based~\cite{wang2017max,hernandez2014predictive}, knowledge gradient (KG)~\cite{frazier2008knowledge,wu2016parallel,wu2019practical}, and stepwise uncertainty reduction (SUR)~\cite{picheny2014stepwise,chevalier2014fast}. Multiobjective Bayesian optimization (MOBO) extends BO to the multiobjective setting; we review some past work in what follows. 

 \subsection{Related work}
\label{sec:related_work}
Naturally, one is tempted to extend the ideas from single objective BO to the multiobjective setting. 

One of the most common approaches in MOBO are the ones based on scalarizations; that is the $K$ objectives are scalarized into one enabling the use of single objective optimization methods. For instance, the multiobjective efficient global optimization (ParEGO)~\cite{knowles2006parego}
scalarizes the $K$ objectives using randomly drawn weights, which is then treated like a single-objective using EI.  Multiobjective evolutionary algorithm based on decomposition (MOEA/D-EGO) \cite{zhang2009expensive} extends ParEGO to the batch setting using multiple random scalarizations and a genetic algorithm \cite{zhou2012multiobjective} to optimize these scalarizations in parallel. Other extensions of ParEGO to the batch setting (qParEGO) and its noisy variant, qNParEGO, have also been proposed \cite{daultonParallelBayesianOptimization2021}.
\cite{paria2020flexible} do a similar ``Chebyshev'' scalarization on the objectives but use Thompson sampling (TS) to optimize the scalarized objective. \cite{zhang2020random} propose to use a hypervolume scalarization with the property that the expected value of the scalarization over a specific distribution of weights is equivalent to the hypervolume indicator. The authors propose a upper confidence bound algorithm using randomly sampled weights, but provide a very limited empirical evaluation. Scalarizations have the drawback that they are not sample efficient and often fail to capture disconnected and non-convex/concave Pareto frontiers.

Outside of scalarizations, the most common approach in MOBO uses the hypervolume indicator to construct an acquisition function that quantifies the improvement in hypervolume. This usually results in a nonconvex optimization problem that is not straightforward to solve.
The EI acquisition function can be extended to the multiobjective setting via the expected hypervolume improvement (EHVI)~\cite{couckuytFastCalculationMultiobjective2014, emmerichComputationExpectedImprovement2008}; gradient based extensions of the EHVI also exist~\cite{yang2019multi, daulton2020differentiable}. Similarly, the SUR criterion extended to multiobjective problems~\cite{picheny2015multiobjective} is another example. 
Recent efforts \cite{daultonDifferentiableExpectedHypervolume2020,daultonParallelBayesianOptimization2021} extend the EHVI idea to sample a batch of acquisitions $(q>1)$ for parallel evaluation and with gradient computation for efficient optimization. However, the stochasticity of the acquisition functions and the associated difficulty in their optimization, particularly as $q$ increases, remains.

Recognizing the non-convexity of HVI based acquisition functions, several methods emerged which make ad hoc changes that better balance exploration and exploitation.
Diversity-Guided Efficient Multiobjective Optimization (DGEMO) \cite{konakoviclukovicDiversityGuidedMultiObjectiveBayesian2020} is a recent method for batch MOBO that improves diversity of the chosen points by solving a local optimization problem to maximize hypervolume improvement. Although DGEMO scales well to large batch sizes, it does not account for noisy observations and its ability to account for constraints is unknown.
Pareto active learning (PAL) \cite{zuluaga2013active} seeks to classify each candidate point in terms of their probability of being Pareto optimal or not, depending upon the GP posterior uncertainty. This approach presents another difficult acquisition function optimization that is handled via a discretized $\X$ in \cite{zuluaga2013active}.

Instead of computing hypervolume improvement, the entropy-based approaches compute the uncertainty reduction in the hypervolume due to new candidates. At the time of writing this manuscript, these approaches are considered state of the art, and primarily include predictive entropy search (PESMO)~\cite{hernandez-lobatoPredictiveEntropySearch} and max-value entropy search (MESMO)~\cite{belakariaMaxvalueEntropySearch2019}. MESMO improves upon PESMO in terms of scalability and computational cost however, the acquisition function is still not available in closed form and requires sample approximations. Pareto frontier entropy search (PFES)~\cite{suzuki2020multi} propose a closed-form entropy criterion for MOBO, but they are restricted to stationary GP kernels, via random Fourier features (RFFs). Additionally, joint entropy search multiobjective optimization (JESMO)~\cite{hvarfner2022joint, tu2022joint} which considers the joint information gain for
the optimal set of inputs and outputs, still depends on approximations to the acquisition function. 

The use of Thompson sampling (TS) in MOBO is relatively underexplored to the best of our knowledge. Existing methods still depend on the classical approaches---scalarization and hypervolume improvement---and, thus, inherit their limitations. For instance, the
Thompson sampling efficient multiobjective
optimization (TSEMO) \cite{bradford2018efficient} considers the Pareto set identified from posterior GP samples, which are then evaluated for maximum hypervolume improvement. The Thompson sampling with Tcheybchev scalarizations (TS-TCH), likewise, depends on scalarizing objective functions. Multiobjective regionalized Bayesian optimization (MORBO)~\cite{daulton2022multi} combines a trust-region strategy with MOBO to achieve scalability in higher dimensions, where they use TS to identify local candidates; however they still use the hypervolume improvement to identify the candidates and their ability to handle constraints is unknown.

\paragraph{Existing limitations and our contributions.} In summary, all of the existing approaches suffer from one or more of the following limitations. (i) Leading to a difficult inner optimization problem that can rarely be solved efficiently and accurately with off-the-shelf optimizers, which negatively impacts their sample efficiency (ii) inability to handle constraints, (iii) inability or difficulty handling batch acquisitions, and (iv) inability to handle noisy objectives .

Our contributions are focused on the following key parameters (i) sample efficiency in both the sequential and batch setting, (ii) ability to handle constraints, (iii) ability to handle noisy objectives, and (iv) ease of implementation. We show that our proposed method addresses all of these parameters while convincingly outperforming the state of the art.
Our central idea is inspired from TS in single-objective optimization; whereas TS selects candidate points according to the probability that they are optimal, our approach selects points according to the probability that they are Pareto optimal. Therefore, we call our method Pareto optimal Thompson sampling or $q\texttt{POTS}$, where the $q$ stands for ``batch'' acquisition. \qpots~involves solving a 
cheap multiobjective optimization problem on the GP posterior sample paths. Our acquisitions are sampled from the computed Pareto sets of the continuous GP posterior sample paths such that they maximize the minimum distance to previously observed points---thereby enabling diversity of candidates and well balancing exploration with exploitation. A very useful consequence of our approach is that selecting a batch of $q$ acquisitions costs the same as the sequential acquisitions. Overall, simplicity of our methodology is our biggest strength.

\begin{figure*}[ht!]
    \centering
    \begin{subfigure}{0.33\textwidth}
        \includegraphics[width=1\linewidth]{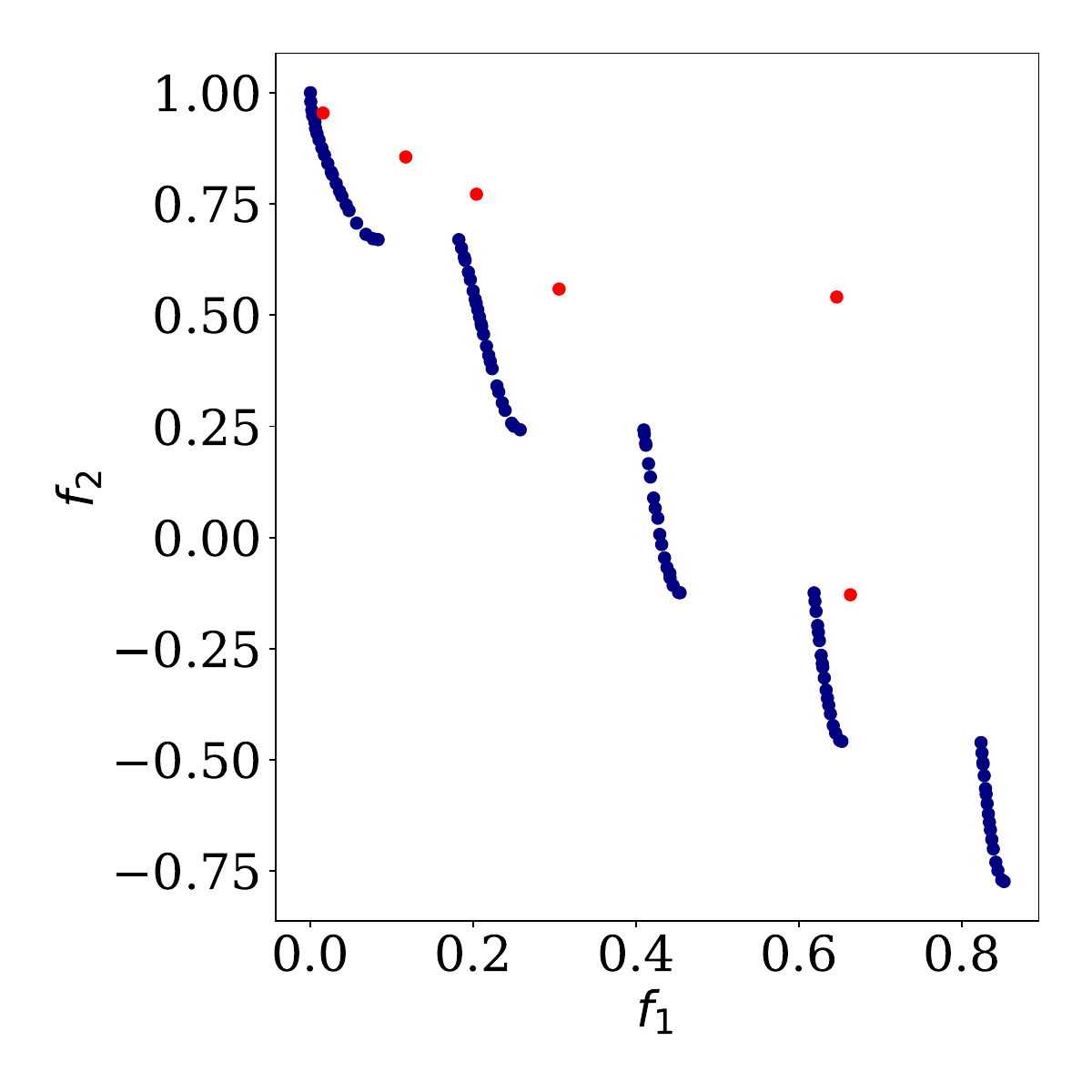}
        \caption{qNEHVI}
    \end{subfigure}%
    \begin{subfigure}{0.33\textwidth}
        \includegraphics[width=1\linewidth]{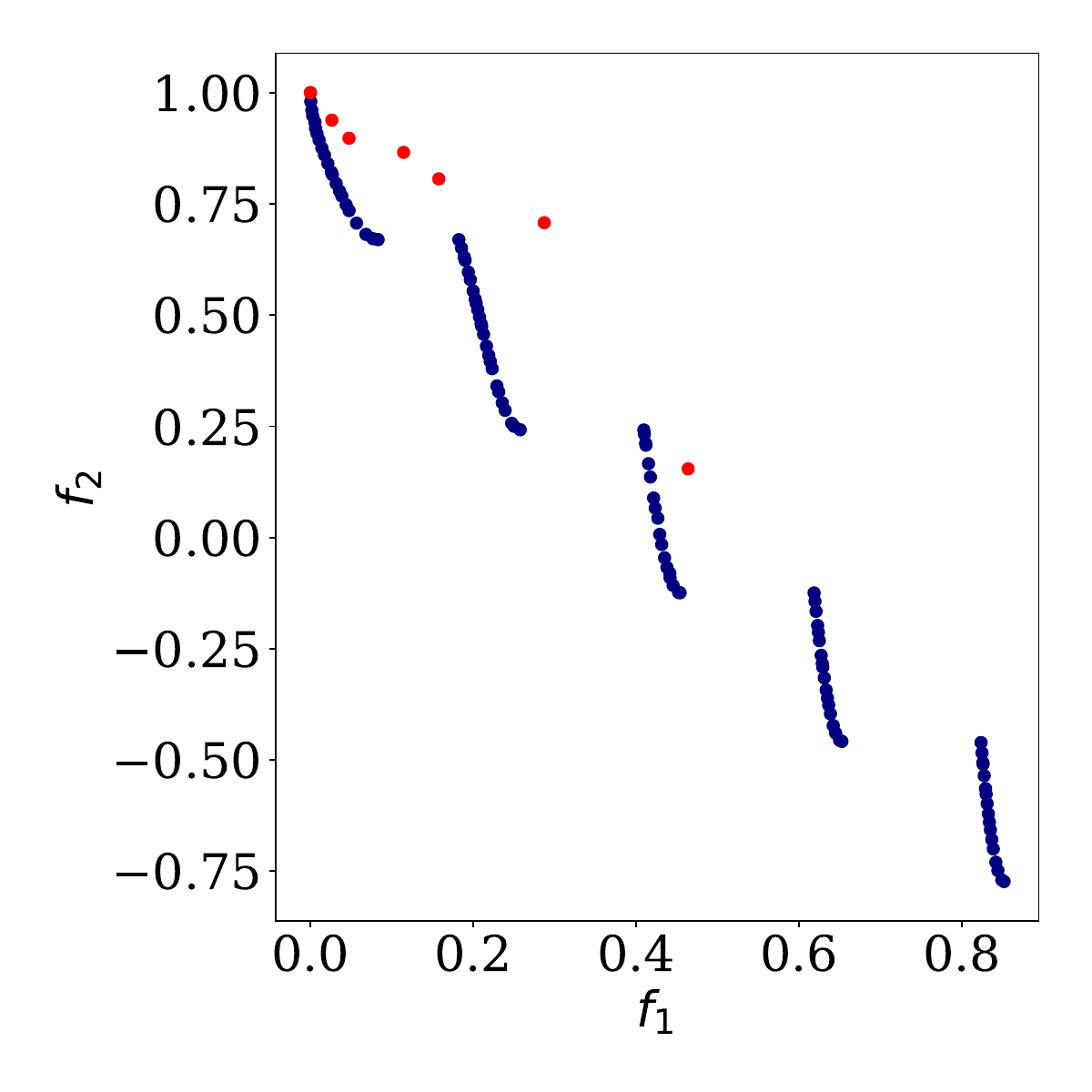}
        \caption{qNPAREGO}
    \end{subfigure}%
    \begin{subfigure}{0.33\textwidth}
        \includegraphics[width=1\linewidth]{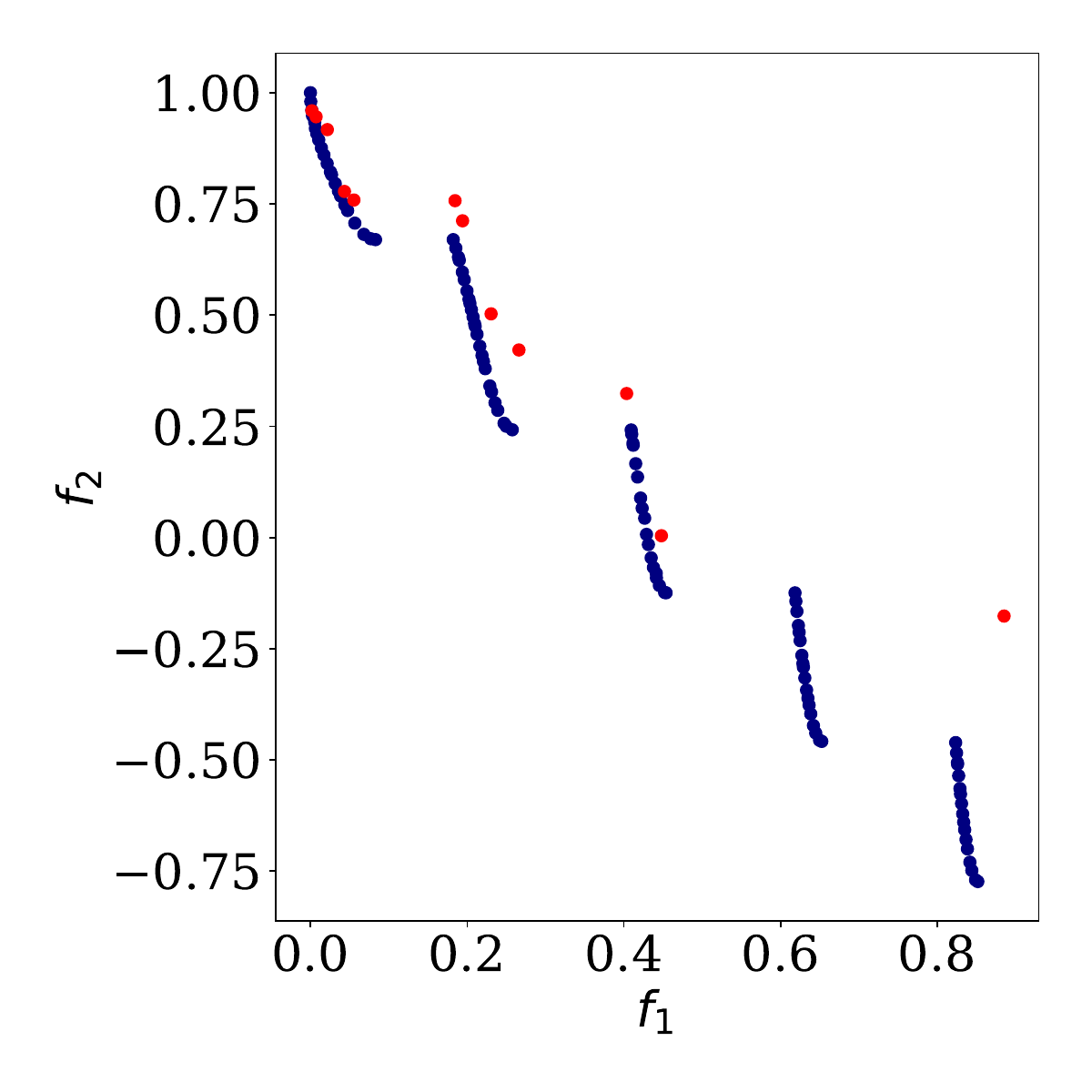}
        \caption{Sobol}
    \end{subfigure}\\
 \begin{subfigure}{0.33\textwidth}
        \includegraphics[width=1\linewidth]{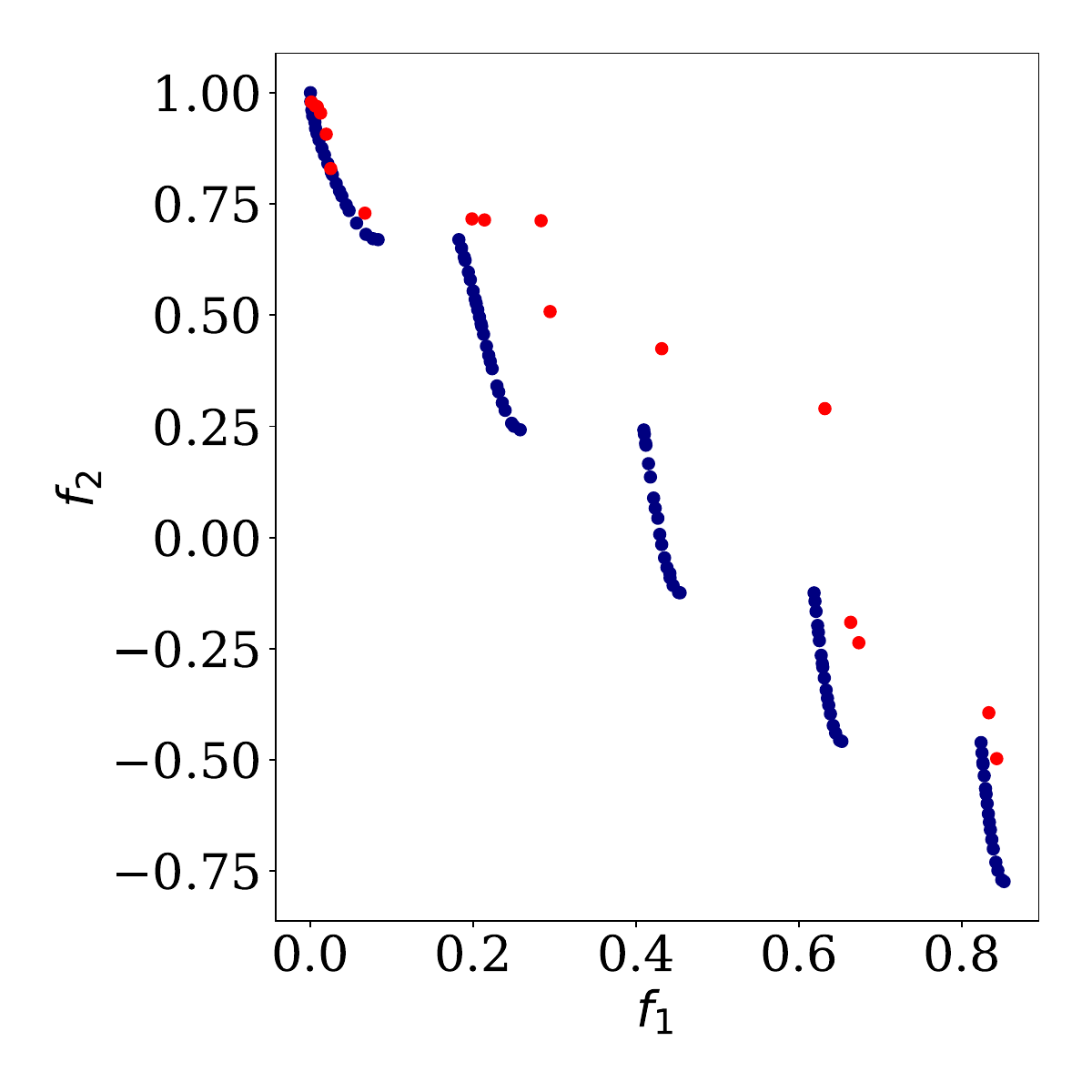}
        \caption{PESMO}
    \end{subfigure}%
    \begin{subfigure}{0.33\textwidth}
        \includegraphics[width=1\linewidth]{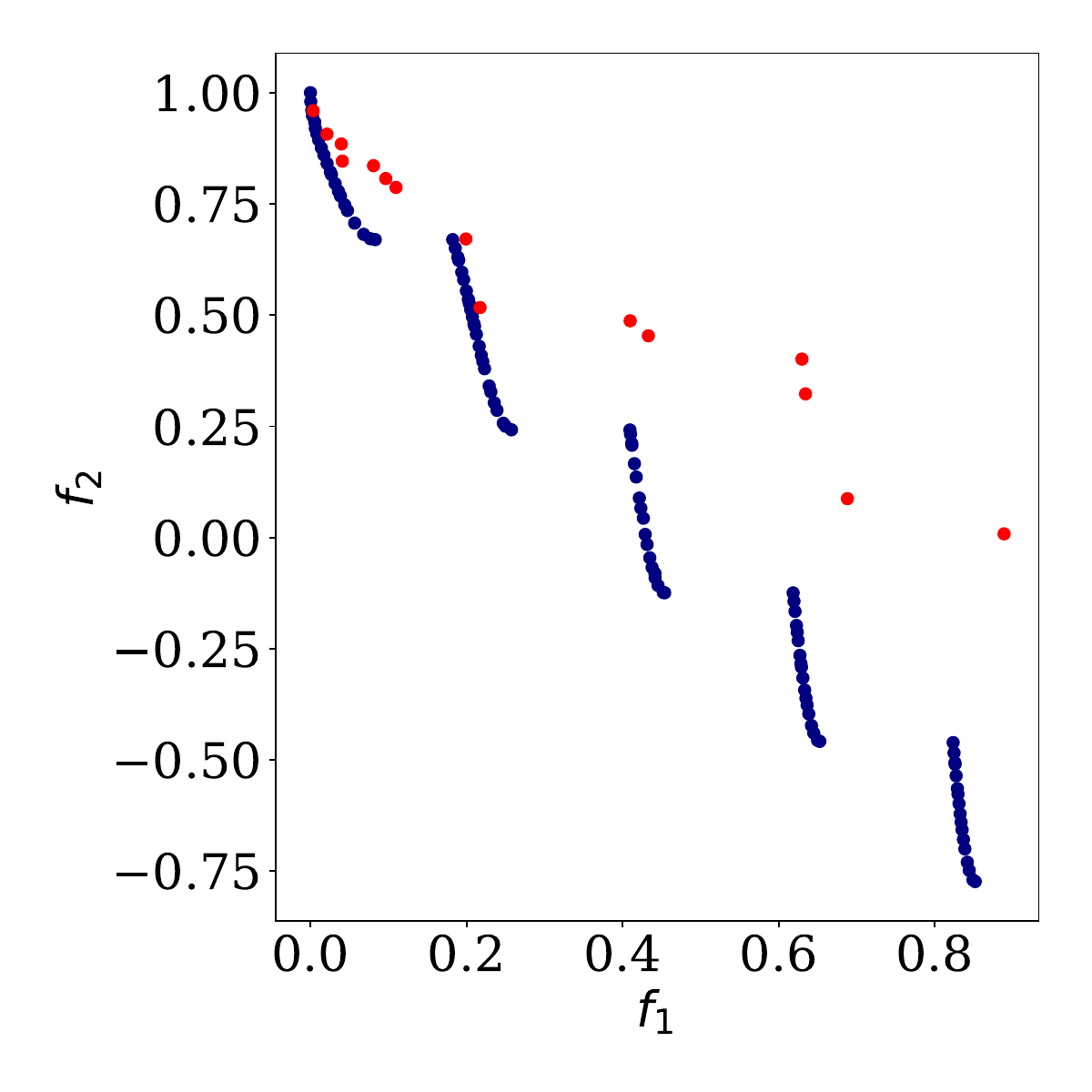}
        \caption{MESMO}
    \end{subfigure}%
    \begin{subfigure}{0.33\textwidth}
        \includegraphics[width=1\linewidth]{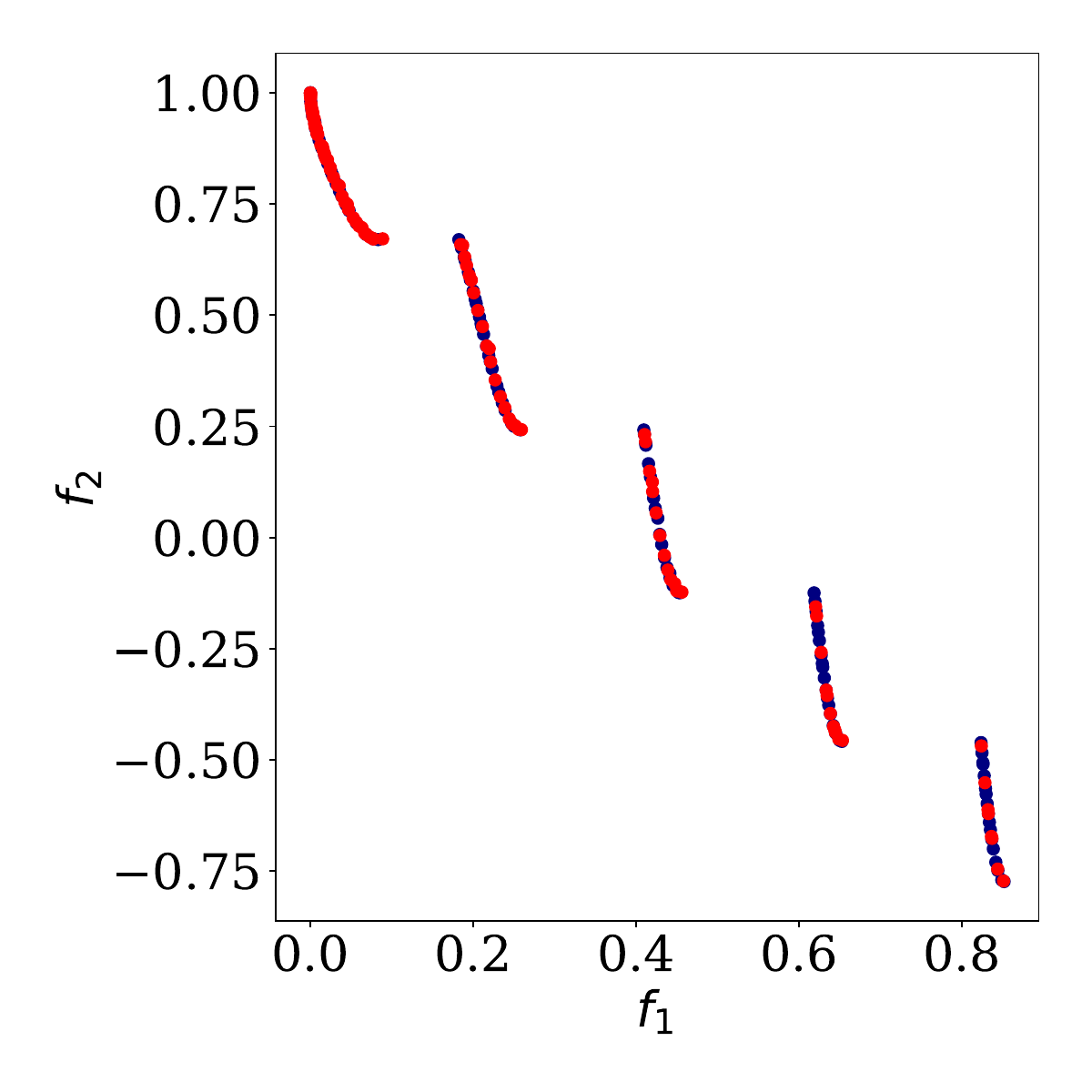}
        \caption{\qpots~(Ours)}
    \end{subfigure}\\    
    \caption{Demonstration on the ZDT3 ($K=2$) test function that has disjoint Pareto frontiers. Navy are the true Pareto frontier and red are the predicted Pareto frontier.
    With $20$ seed points and only an additional $204$ acquisitions, in batches of $q=4$, notice that \qpots~is quickly able to resolve the true Pareto frontier while other methods in the state of the art struggle to come any close.}
    \label{fig:zdt3}
\end{figure*}

The rationale behind our approach is as follows. Under low posterior uncertainty, the random Pareto frontier is likely a good approximation for $\mcl{Y}^*$, and thus we can exploit. On the other hand, under high posterior uncertainty, the random Pareto frontier is likely far from $\mcl{Y}^*$ and thus we can explore. Therefore our approach is endowed with a natural balance between exploration and exploitation; see \Cref{fig:bc_demo}. The work in the literature that is somewhat similar to us is that of \cite{bradford2018efficient} (TSEMO) but is significantly different. First, our method circumvents the hypervolume computation part, which scales exponentially with the number of objectives. Second, TSEMO does not support handling constraints, but that is straightforward in \qpots. Third, the ability for TSEMO to handle noisy objectives is unknown, but \qpots~handles it in a straightforward manner. Fourth, batch acquisition in TSEMO becomes intractable with batch size; in contrast, batch acquisition in our approach comes at the \emph{same} cost as sequential sampling. Fifth, TSEMO is restricted to stationary GP kernels due to their dependence on random Fourier features~\cite{rahimi2007random}; however our approach applies more widely and seamlessly works with nonstationary GP kernels such as deep GPs~\cite{damianou2013deep, booth2023contour, sauer2023active, rajaram2020deep} as well as multitask GPs~\cite{bonilla2007multi,swersky2013multi}. Finally, we show how our approach is better scalable in high dimensions by leveraging a Nystr\"{o}m approximation on the posterior covariance matrix. We show several experiments where \qpots~convincingly outperforms TSEMO, in addition to other methods in the state of the art. \Cref{fig:zdt3} provides a demonstration of \qpots~on the ZDT3 test problem that has a disjoint Pareto frontier.

The rest of the article is organized as follows. Details of the method are outlined in \Cref{s:method}, theoretical properties are discussed in \Cref{sec:theory}. The numerical experiments are presented in \Cref{s:experiments}. We provide concluding remarks and an outlook for future work in \Cref{s:conclusions}. Our software implementation is available at \url{https://github.com/csdlpsu/qpots}.

\section{Methodology}
\label{s:method}

\subsection{Single objective Bayesian optimization}
\label{sec:bo}
We place a GP prior on the oracle $f(\x) \sim \mcl{GP}(0, k(\x, \cdot))$, where $k(\cdot, \cdot): \X \times \X \rightarrow \mbb{R}_+$ is a covariance function (a.k.a., \emph{kernel}).
We denote the observations from the oracle as $y_i = f(\x_i) + \epsilon_i,~i=1,\ldots,n$, where we assume that $\epsilon_i$ is a zero-mean Gaussian with unknown variance $\tau^2$. We begin by fitting a posterior GP (parametrized by hyperparameters $\bs{\Omega}$) for the observations $\D_n = \{\x_i, y_i, \tau^2\},~i=1,\ldots,n$, which gives the conditional aposteriori distribution~\cite{rasmussen:williams:2006}:
\begin{equation}
    \begin{split}
        Y(\x) |
      \D_n, \bs{\Omega} &\sim \mcl{GP}(\mu_n(\x), {\sigma}^2_n (\x)), \\
\mu_n(\x) &= \mbf{k}_n^\top [\mbf{K}_n + \tau^2\mbf{I} ]^{-1} \mbf{y}_n\\
\sigma^2_n(\x) &=  k(\x, \x) - \mbf{k}_n^\top [\mbf{K}_n + \tau^2\mbf{I}]^{-1} \mbf{k}_n,
    \end{split}
    \label{e:GP}
\end{equation}
where $\mbf{k}_n \equiv k(\x, X_n)$ is a vector of covariances between $\x$ and all observed points in $\D_n$, $\mbf{K}_n \equiv k(X_n, X_n)$ is a sample covariance matrix of observed points in $\D_n$, $\mbf{I}$ is the identity matrix, and $\mbf{y}_n$ is the vector of all observations in $\D_n$. We also denote by $X_n = [\x_1,\ldots,\x_n]^\top \in \mbb{R}^{n\times d}$ the observation sites.  BO seeks to make sequential decisions based on a probabilistic utility function constructed out of the posterior GP $u(\x) = u(Y(\x)|\D_n)$. 
Sequential decisions are the result of an ``inner'' optimization subproblem that optimizes an acquisition function, typically, of the form $\alpha(\x) = \mbb{E}_{Y|\D_n}[u(Y(\x))]$. 

An approach that circumvents the acquisition function construction is Thompson sampling (TS)~\cite{thompson1933likelihood}.
TS proposes making a decision according to the probability that it is optimal in some sense~\cite{russo2014learning}.  In (single objective) optimization, points are chosen according to the distribution $p_{\x^*}(\x)$, where $\x^*$ is the optimizer of the objective function.  
In the context of BO, this can be expressed as
\begin{equation}
\begin{split}
    p_{\x^*}(\x) =& \int p_{\x^*}(\x|Y)p(Y|\D_n)dY \\  
    =& \int \delta(\x - arg\max_{\x \in \X}Y(\x))p(Y|\D_n)dY, 
\end{split}
\label{eqn:ts}
 \end{equation}
where $\delta$ is the Dirac delta function. Sampling from $p_{\x^*}(\x)$ is equivalent to finding the maximizer of a sample path $Y(\cdot, \omega)$, where $\omega$ is the random parameter; therefore, in TS for BO, the acquisition function is simply the posterior sample path $\alpha \equiv Y$. Our proposed approach builds on TS which we present next.
\subsection{Multiobjective Bayesian optimization}
\label{ss:mobo}
We first fit $K$ independent posterior GP models: $\mcl{M}_1, \ldots, \mcl{M}_K$ for the $K$ objectives, with observations $\D^{1:K}_n \triangleq \{ \D_n^1,\ldots, \D_n^K \}$.
In the multiobjective setting, acquisition functions are typically defined in terms of a hypervolume (HV) indicator. HV of a Pareto frontier $\mcl{Y}^*$ is the $K-$dimensional Lebesgue measure $\lambda$ of the region dominated by $\mcl{Y}^*$ and bounded from below by a reference point $\mbf{r} \in \mbb{R}^K$. That is, $HV(\mcl{Y}^*|\mbf{r}) = \lambda_K (\bigcup _{\nu \in \mcl{Y}^*} [\mbf{r}, \nu])$, where $[\mbf{r}, \nu]$ denotes the hyper-rectangle bounded by vertices $\mbf{r}$ and $\nu$. EHVI~\cite{couckuytFastCalculationMultiobjective2014, emmerichComputationExpectedImprovement2008} computes the expectation (with respect to the posterior GP) of the improvement in HV due to a candidate point. Our approach, however, does not involve the computation of HV but is a direct extension of \cref{eqn:ts} to the multiobjective setting.

\subsection{$q$\texttt{POTS:} Batch Pareto optimal Thompson sampling for MOBO}
\label{s:qpots}


Inspired from TS, \qpots~chooses points according to the probability that they are Pareto optimal. That is, we choose $\x$ according to the probability
\begin{equation}
    \begin{split}
    p_{\X^*}(\x) = \int &\delta \left(\x - arg\max_{\x \in \X}~\{Y_1(\x),\ldots,Y_K(\x)\}\right) \\
     & p(Y_1|\D^1_n)\ldots p(Y_K|\D^K_n)dY_1\ldots dY_K,
    \end{split}
\end{equation}
where the integral is over the support of the joint distribution of all posterior GPs.  In practice, this is equivalent to drawing sample paths from the $K$ posterior GPs $Y_i(\cdot, \omega),~i=1,\ldots,K$, and choosing the next point(s) from its Pareto set $\x_{n+1} \in \argmax_{\x \in \X} \{Y_1(\x, \omega), \ldots, Y_K(\x, \omega)\}$. We use the ``reparametrization trick'' to sample from the posterior (again we exclude notations for individual objectives): $Y(\cdot, \omega) = \mu_n(\cdot) + {\mbf{\Sigma}_n}^{1/2}(\cdot) \times Z(\omega)$, where $Z$ is a standard multivariate normal random variable. Then, we solve the following cheap multiobjective optimization problem
\begin{equation}
  X^* = \argmax_{\x \in \X} \{Y_1(\x, \omega), \ldots, Y_K(\x, \omega)\},
  \label{eqn:moo_gp}
\end{equation}
that can be solved using any method; we choose evolutionary approaches, e.g., the NSGA-II~\cite{deb2002fast}, so that we keep the best of both worlds: accuracy of evolutionary approaches and sample efficiency of MOBO. 
The next batch of $q$ points is then chosen as $\x_{n+1:q} \in X^*.$
That is, we select points from the predicted Pareto set of the sample paths of each posterior GP. 
In principle, any acquisition chosen from $X^*$ is Pareto optimal per the drawn sample paths. However, to promote diversity of candidates, we ensure that the acquisition $\x_{n+1}$ has the largest \emph{maximin distance}~\cite{johnson1990minimax,sun2019synthesizing} to the previously observed points. This way, there is an automatic balance between exploration and exploitation--see \Cref{fig:bc_demo} for an illustration.
Let $\gamma(\cdot, \cdot): \mbb{R}^d \times \mbb{R}^d \rightarrow \mbb{R}_{+}$ denote the Euclidean distance between two points in $\X$. Then, we pick a batch of $q$ new acquisitions according to the following sequential maximin optimization problems
\begin{equation}
    \begin{split}
    \x_{n+1} =& \argmax_{{\x}^* \in X^*} \min_{\x_i\in X_n} \gamma({\x}^*, \x_i) \\
    \x_{n+2} =& \argmax_{{\x}^* \in X^*} \min_{\x_i\in X_n \bigcup \x_{n+1}} \gamma({\x}^*, \x_i) \\
    \vdots& \\
    \x_{n+q} =& \argmax_{{\x}^* \in X^*} \min_{\x_i\in X_n \bigcup \{ \x_{n+1}, \ldots, \x_{n+q-1} \} } \gamma({\x}^*, \x_i).
    \end{split}
    \label{eqn:maximin}
\end{equation}
\Cref{eqn:maximin} is a seemingly hard optimization problem, but can be solved quite easily and efficiently. This is because, in practice, the solution of \eqref{eqn:moo_gp} returned by an evolutionary algorithm based optimizer is a point set with finite cardinality. Therefore, what we end up with is the computation of an $N^* \times n$ pairwise distance matrix ($N^* = |X^*|$) and rank-ordering them to according to their column-wise minima. An ascending sort of the resulting vector allows for selecting the $q$ batch acquisition as the last $q$ elements of the sorted vector.

\subsection{Extension to constrained problems}
 Our method seamlessly extends to constrained multiobjective problems and the key step is explained as follows. Let there be a total of $C$ constraints $\{ c_1, \ldots, c_C\}$, and $\mcl{I} \cup \mcl{E} = [C]$. Further, let $Y_{k+i\in \mcl{I}}$ and $Y_{k+i\in \mcl{E}}$ denote the posterior GP of the $i$th inequality constraint and equality constraint functions, respectively. This way, we fit a total of $K+C$ posterior GPs $Y_1, \ldots, Y_K, Y_{K+1}, \ldots, Y_{K+C}$.
 For all $i \in [C]$, we define the constrained \qpots~ to choose points according to
\begin{equation}
\begin{split}
p_{\X^*}(\x) = \int \delta &\left(\x - \argmax_{\x \in \X}
\left[ \{Y_1(\x),\ldots,Y_K(\x)\} \times \right. \right. \\
& 
\left. \left. 
\mathbbm{1}_{\{\bigcap_{i} Y_{k+i \in \mcl{E}}(\x)= 0 \} }\times \mathbbm{1}_{\{\bigcap_i Y_{k+i \in \mcl{I}}(\x)\geq 0 \}}\right]\right) \\
& p(Y_1|\D_n)\ldots p(Y_K|\D_n)dY_1\ldots dY_K,
\end{split}
\end{equation}
where $\mathbbm{1}_{\bigcap_i Y_{k+i \in \mcl{I}}(\x)\geq 0}$ and $\mathbbm{1}_{\bigcap_i Y_{k+i \in \mcl{I}}(\x) = 0}$ are the indicator functions for joint feasibility with inequality and equality constraints, respectively, that take a value of $1$ when feasible and $0$ otherwise. This 
is equivalent to drawing samples from the GP posteriors, filtering them out by the indicator functions, and finding a Pareto optimal solution set per the posterior GP sample paths. 
\begin{equation}
    \begin{split}            
  X^* = \argmax_{\x \in \X} &\left[ \{Y_1(\x, \omega), \ldots, Y_K(\x, \omega)\}\times \right.\\ 
  & \left. \mathbbm{1}_{\{\bigcap_{i} Y_{k+i \in \mcl{E}}(\x)= 0 \} }\times \mathbbm{1}_{\{\bigcap_i Y_{k+i \in \mcl{I}}(\x)\geq 0 \}}\right],
  \label{eqn:moo_gp_cons}
  \end{split}
\end{equation}
Then, $q$ maximin candidates are chosen similar to what is done in the unconstrained case. This way, handling constraints in \qpots~fits seamlessly into the unconstrained framework with minimal changes.
 An illustration of the constrained Branin-Currin test case is shown in \Cref{fig:cons_bc_demo}; notice that \qpots~is able to choose an acquisition within the feasible region.
\begin{figure}
    \centering
\includegraphics[width=1\linewidth]{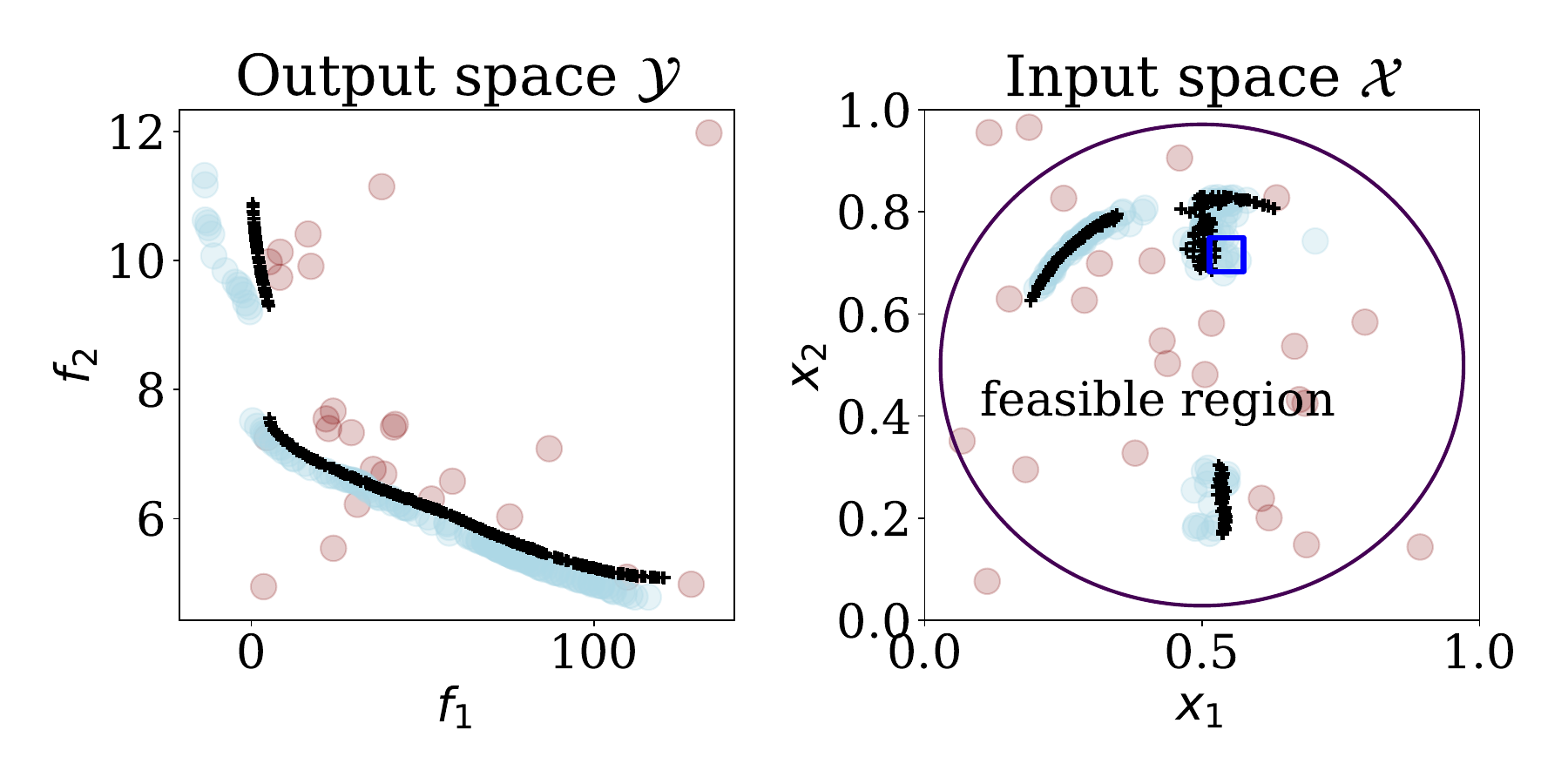}
    \caption{Illustration of \qpots~in the constrained setting. Brown circles are GP training points, blue circles are $X^*$, black crosses are the true Pareto frontier/set, black line is the constraint boundary and blue square is the \qpots~acquisition. Left: output space, right: input space. }
    \label{fig:cons_bc_demo}
\end{figure}

A natural question that could arise in the constrained case is \emph{what if $X^* = \empty$ (that is, an empty set) or $|X^*| < q$?}. This is particularly likely when the problem has one or more equality constraints. In such situations, \qpots~simply does not make any acquisition and, instead, solves \Cref{eqn:moo_gp_cons} and tries again. This process is repeated until $q \geq |X^*|$. As long as the population size in NSGA-II is large enough, in principle, this approach should eventually result in $q \geq |X^*|$. However, in the authors' experience so far, this issue has never occurred and $|X^*|$ is always non-empty in the first attempt.

 \subsection{Nystr\"{o}m approximation on active candidate set}
\label{ss:nystrom}
We illustrate how our approach can improve scalability in high dimensions with the Nystr\"{o}m approximation with GP posteriors. 
We begin by writing the reparametrization technique for a GP posterior:
\[
\begin{split}
Y(X) =& \mu(X) + \left[ k(X, X_n) k(X_n, X_n)^{-1} k(X_n, X) \right]^{1/2} Z \\
=& \mu(X) + \Sigma(X)^{1/2} Z,
\end{split}
\]
where $\Sigma(X) = k(X, X_n) k(X_n, X_n)^{-1} k(X_n, X)$ is an $N \times N$ posterior covariance matrix, that we also write as $\Sigma_{NN}$; the main computational bottleneck is in computing the square-root of $\Sigma(X)$ which costs $\mcl{O}(N^3)$. We now choose a subset $X_m \subset X$, $m \ll N$, to write the following, augmented $(N+m) \times (N+m)$ matrix
\[\Sigma_{(N+m)(N+m)} = \begin{bmatrix}
    \Sigma_{mm} & \Sigma_{mN} \\
    \Sigma_{Nm} & \Sigma_{NN}
\end{bmatrix},\]
where $\Sigma_{mm}$ is the intersection of the $m$ rows and $m$ columns of $\Sigma_{NN}$ corresponding to $X_m$, and $\Sigma_{mN} = \Sigma_{Nm}^\top$ are the $m$ rows of $\Sigma_{NN}$ corresponding to $X_m$. Then, we have that
\begin{equation}
    \begin{split}            
    \Sigma(X) =& \Sigma_{NN} \approx \Sigma_{mN} \Sigma_{mm}^{-1} \Sigma_{mN}^\top\\
    =& \left(\Sigma_{mN} \Sigma_{mm}^{-1/2} \right)\left(\Sigma_{mN} \Sigma_{mm}^{-1/2} \right)^\top.
    \end{split}
    \label{eqn:nystrom_squareroot}
\end{equation}
Therefore, we have the Nystr\"{o}m approximation of the posterior covariance matrix square root as $\Sigma(X)^{1/2} \approx  \left(\Sigma_{mN} \Sigma_{mm}^{-1/2} \right).$
The approximation in \Cref{eqn:nystrom_squareroot} requires an $m\times m$ matrix square root which costs $\mcl{O}(m^3)$ and an additional $\mcl{O}(Nm^2)$ for the matrix multiplication. Therefore, an overall cost of $\mcl{O}(m^3 + Nm^2)$ is much better than $\mcl{O}(N^3)$, provided $m \ll N$. The choice of $X_m$ is an open research question; in this work, we choose them to be the subset of nondominated points in $\{X_n, \mbf{y}_n\}$ which works uniformly well in practice.
We provide more illustration in the supplementary material.


\begin{algorithm}[tb]
 \caption{\qpots: Batch Pareto optimal Thompson sampling}
 \label{a:method}
\begin{algorithmic}
   \STATE {\bfseries Input:} With data sets $\D_n^{1:K+C}$, fit $K+C$ GP models $\mcl{M}_1$ through $\mcl{M}_{K+C}$
     and GP hyperparameters $\bs{\Omega}_1,\ldots,\bs{\Omega}_{K+C}$. \\
     Parameters $B$ (total budget) \\
   \STATE \textbf{Output:} {Pareto optimal solution $\mcl{X}^*,~\mcl{Y}^*$}.
  \STATE \FOR{$i=n+1, \ldots, B$, }
      \STATE {\bf 1.} Sample the posterior GP sample paths: $Y_i(\x, \omega),~\forall i=1,\ldots,K$.
      \STATE  {\bf 2.} {\bf (Optional)}\ Use Nystr\"{o}m approximation for scalability in higher dimensions.
      \STATE {\bf 3.} Solve the cheap multiobjective optimization problem  \Cref{eqn:moo_gp} (or \eqref{eqn:moo_gp_cons} for constrained problems) via evolutionary algorithms. 
      \STATE {\bf 4.} Choose $q$ candidate points from $X^*$ according to \eqref{eqn:maximin}.
      \STATE {\bf 5.} Observe oracles and constraints and append to data set $\D^{1:K+C}_i$.
      \STATE {\bf 6.} Update GP hyperparameters $\bs{\Omega}_1,\ldots,\bs{\Omega}_{K+C}$
    \ENDFOR
\end{algorithmic}
\end{algorithm}

\section{Theoretical properties}
\label{sec:theory}
We now provide theoretical guarantees on \qpots. We show that the absolute difference between $f_i$ and its GP posterior mean, $\forall \x\in\X$, is asymptotically consistent. Then, we show that the predicted Pareto frontier asymptotically converges to the true Pareto frontier. Our proofs depend on two key assumptions. First, we place regularity assumptions on $f_i,~\forall i$. Then, we assume that our inner (multiobjective) optimization solution, via evolutionary techniques, is solved exactly. We now proceed to present the key theoretical results.

\begin{assumption}[{\bf Exact knowledge of $X^*$}]
We assume that the inner multiobjective optimization problem $\argmax_{\x \in \X} \{Y_1(\x, \omega), \ldots, Y_K(\x, \omega)\},$ is solved exactly and the true $X^*$ is known.  
\label{ass:ass1}
\end{assumption}
Assumption \ref{ass:ass1} is quite reasonable in practice since evolutionary algorithms come with tuning parameters e.g., population size and number of generations, which can be made arbitrarily high to improve accuracy of $X^*$; we provide an empirical illustration in \Cref{fig:ngen} and \Cref{fig:npop}.

\begin{assumption}[{\bf Lipschitz}]
We assume that the sample paths drawn from the GP prior on $f$ are twice continuously differentiable. Such a GP is achieved by choosing the covariance kernel to be the squared-exponential or from the Matérn class with $\nu=5/2$
~\citep[Theorem 5]{ghosal2006posterior}. Then, the following holds (for $L \in \mbb{R}_+$): $\mP \left[ |f(\x) - f(\x')| > L \|\x - \x'\| \right] \leq e^{-L^2 / 2}.$
\label{ass:lipschitz}
\end{assumption}

\begin{lemma}
Let Assumption \ref{ass:lipschitz} hold. Let $\{\x_i\},~i=1,\ldots,n$ be a sequence of points selected via the \qpots~method according to \eqref{eqn:maximin}. Then, $\lim_{n \rightarrow \infty}~\sup_{\x \in \X} \sigma^2_n(\x) = 0$.
\label{lm:gpvariance}
\end{lemma}

Using the aforementioned assumptions and Lemma \ref{lm:gpvariance}, we first show that the vector of posterior GPs $\mbf{Y}=[Y_1,\ldots,Y_K]^\top$ converges to the vector of true objectives $\bs{f}=[f_1,\ldots,f_K]^\top$ with high probability.

\begin{theorem}[{\bf Convergence with high probability}]
\label{thm:ass_consistency}
Let $\delta \in (0, 1)$ and Assumption \ref{ass:lipschitz} hold with $L^2 = 2 \log \left(\f{2}{\delta}\right)$.  Let $\tilde{\X}$ denote a discretization of $\X$ with $|\tilde{\X}| < \infty$, where we evaluate the oracles $f_i$. Let $\epsilon \geq 0$, and $a_n = \sqrt{2\log(b_n |\tilde{\X}|)/\delta}$, where $\sum_{n=1}^\infty 1/b_n = 1$. Then the vector of posterior GP objectives converges to $\mbf{f}$ with high probability; that is, the following holds:

\[ \mP \left[\forall \x \in \X, ~ \lim_{n \rightarrow \infty}~ \|\mbf{f}(\x) - \mbf{Y}^{n-1}(\x)\|_\infty \leq \epsilon \right] > 1 - \delta. \]
\end{theorem}
Due to \Cref{thm:ass_consistency}, the Pareto frontier computed with the GP posteriors converges to true Pareto frontier with high probability. All proofs are presented in the supplementary material.

\section{Experiments}
\label{s:experiments}
We demonstrate our methodology on several synthetic and real-world experiments varying in input ($d$) and output ($K$) dimensionality. We compare \qpots~against TSEMO (our main competitor), PESMO, JESMO, MESMO (state of the art), qNEHVI, qPAREGO (classical methods), and random (Sobol) sampling. Note that, amongst all competitors, only qNPAREGO and qNEHVI support constraints.

\noindent {\bf Synthetic experiments.}
For the synthetic experiments, we entertain the Branin-Currin, ZDT3 which has disconnected Pareto frontiers, DTLZ3, and DTLZ7 test functions. Additionally, we test on the 
constrained OSY problem~\cite{osyczka1995new} with $C=6$ constraints. Additional experiments are included in the supplementary material.

\noindent {\bf Real-world experiments.}
 For the real-world experiments, we use the Vehicle Safety problem~\cite{tanabe2020easy} and the Penicillin production problem~\cite{liang2021scalable}, and the Car side-impact problem~\cite{jain2013evolutionary}. The Penicillin production problem ($d=7, K=3$) seeks to maximize Penicillin yield while minimizing the time for production as well as the carbon dioxide byproduct emission. The $d=7$ input variables include the culture medium volume, biomass concentration, temperature (in K), glucose substrate concentration, substrate feed rate, substrate feed concentration and the $H+$ concentration~\cite{liang2021scalable}. The Vehicle Safety problem ($d=5, K=3$)~\cite{tanabe2020easy} concerns the design of automobiles for enhanced safety. Specifically, we consider the minimization of vehicle mass, the vehicle acceleration during a full-frontal crash, and the toe-board intrusion---a measure of the mechanical damage to the vehicle during an off-frontal crash, with respect to $5$ design variables which represent the reinforced parts of the frontal frame; further detail are found in~\cite{liao2008multiobjective}. The Car side-impact problem is aimed at minimizing the weight of the car while minimizing the pubic force experienced by a passenger and the
average velocity of the V-pillar responsible for withstanding the impact load. Additionally, as the fourth objective, the maximum constraint violation of $10$ constraints which include limiting values of abdomen load, pubic
force, velocity of V-pillar, and rib deflection is also added~\cite{jain2013evolutionary}.

We provide each experiment a set of $n=10\times d$ seed samples and observations to start the algorithm, chosen uniformly at random from $\mcl{X}$; then they are repeated $10$ times. However, the seed is fixed for each repetition across all acquisition policies to keep comparison fair.  We add a Gaussian noise with variance $\tau^2=10^{-3}$ to all our observations.
Our metric for comparison is the hypervolume computed via box decomposition~\cite{yang2019efficient}.

Throughout the manuscript, we default to the anisotropic Matern class kernel with $\nu=5/2$.  Our implementation primarily builds on GPyTorch~\cite{gardner2018gpytorch} and BoTorch~\cite{balandat2019botorch}; we leverage 
MPI4Py~\cite{dalcin2021mpi4py} to parallelize our repetitions. We use PyMOO~\cite{blank2020pymoo} to leverage its NSGA-II solver; the population size for all problems are fixed as $100\times d$. The full source code is available at \url{https://github.com/csdlpsu/qpots}.

\Cref{fig:hv_seq} shows the performance comparison for sequential sampling ($q=1$).  Notice that \qpots~is consistently the best
performer, or is amongst the best. We show
\qpots-Nystrom-Pareto (that is $X_m$ chosen as the current Pareto set) for the $d=10$ experiments which shows no loss of accuracy for the DTLZ3 function and, surprisingly, a slight gain in accuracy for the ZDT3 function. In the bottom right, we show the performance on the constrained OSY~\cite{osyczka1995new} problem; qNEHVI failed to complete for this experiment due to memory overflow.
Additional experiments are provided in the supplementary material. Despite the success in the $q=1$ setting, our main benefit is seen in the batch sampling setting which we discuss next.


The demonstration for the batch setting is shown in \Cref{fig:hv_batch}; we present $q=2, 4$ with more in the supplementary material; like in the $q=1$ setting we include \qpots-Nystrom-Pareto for the $d=10$ experiments only.  While \qpots~comfortably wins in all experiments, notice that the difference with TSEMO is more pronounced in the batch setting. We demonstrate the computational efficiency gains of \qpots~in the supplementary material.

\section{Conclusion}
\label{s:conclusions}
\vspace{-3 mm}
We present \qpots,~a sample efficient method to solve expensive constrained multiobjective optimization problems. Classical evolutionary approaches are accurate but not sample efficient. Bayesian optimization is, in general, sample-efficient, but its success relies heavily on solving a difficult inner optimization problem. Our proposed \qpots~relies on sampling from the posterior GPs, and solving a cheap multiobjective optimization on the posterior sample paths with classical evolutionary algorithms such as NSGA-II. This way, we bridge the gap between classical EA's and MOBO and thus, in a sense, get the best of both worlds. \qpots~naturally balances exploration and exploitation, can handle constraints and noisy objectives, and can efficiently handle batch acquisitions; no other existing method has all these features and still outperform \qpots~to the best of our knowledge. Scaling up to higher dimensions is achieved through a Nystr\"{o}m approximation which significantly speeds up the method without any practical loss of accuracy. Crucially, our approach freely applies to any covariance kernel (unlike RFF for stationary kernels).

Looking into the future, one avenue of interest is  marrying \qpots~with nonstationary GPs, such as deep GPs~\cite{damianou2013deep} and input warping~\cite{snoek2014input}, which would open doors for more applications where nonstationarity in the output space is expected to be strong. Additionally, we also intend to explore the use of multitask GPs~\cite{swersky2013multi} to exploit potential underlying correlations between the objectives and constraints. We anticipate that \qpots~would require no structural changes to accommodate nonstationary and multitask GPs.
\begin{figure*}
    \centering
    \begin{subfigure}{.25\textwidth}
        \includegraphics[width=1\linewidth]{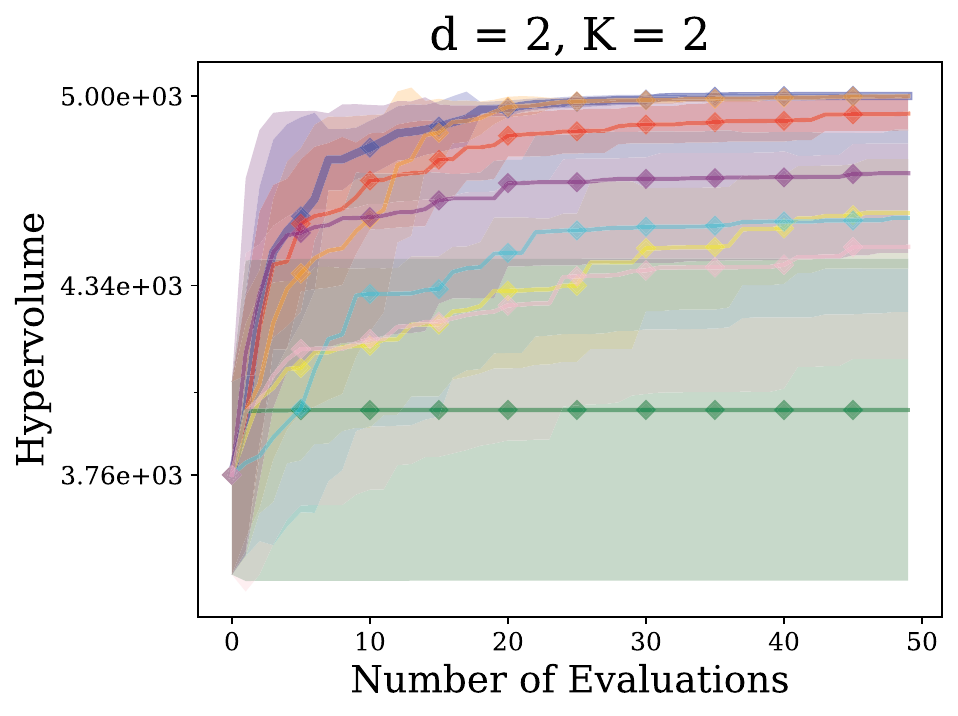}
        \caption{Branin-Currin}
    \end{subfigure}%
    \begin{subfigure}{.25\textwidth}
        \includegraphics[width=1\linewidth]{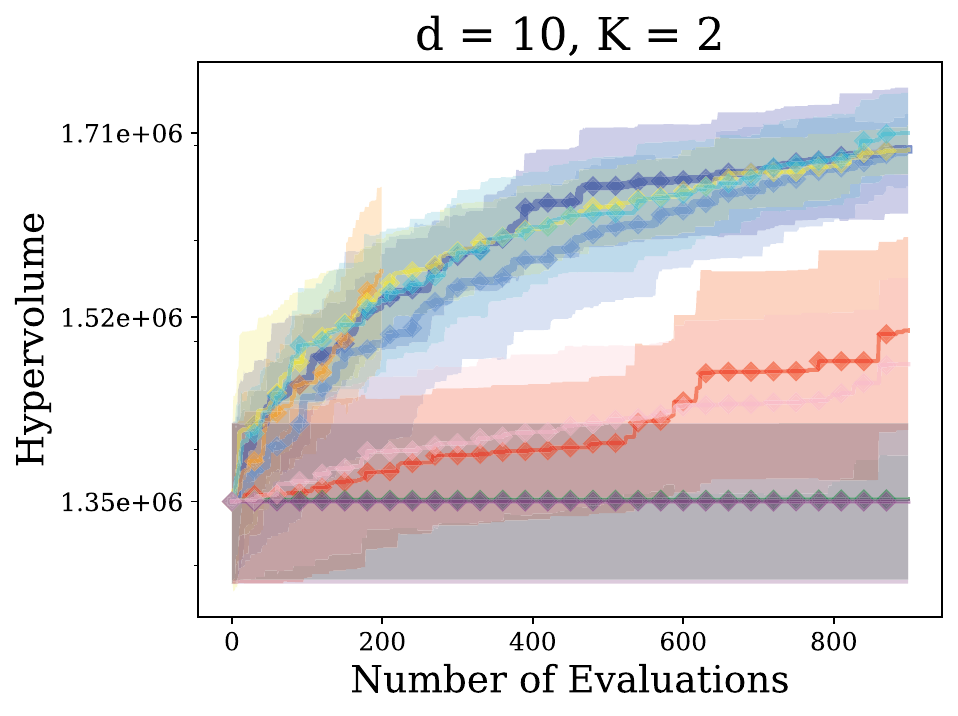}
        \caption{DTLZ3}
    \end{subfigure}%
    \begin{subfigure}{.25\textwidth}
        \includegraphics[width=1\linewidth]{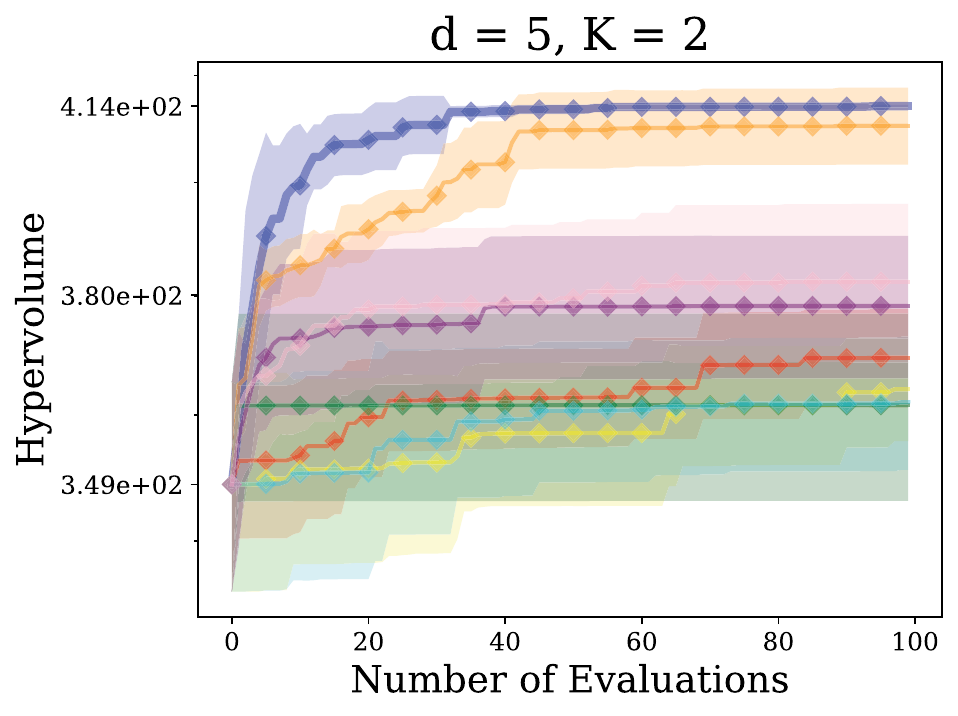}
        \caption{DTLZ7}
    \end{subfigure}%
    \begin{subfigure}{.25\textwidth}
        \includegraphics[width=1\linewidth]{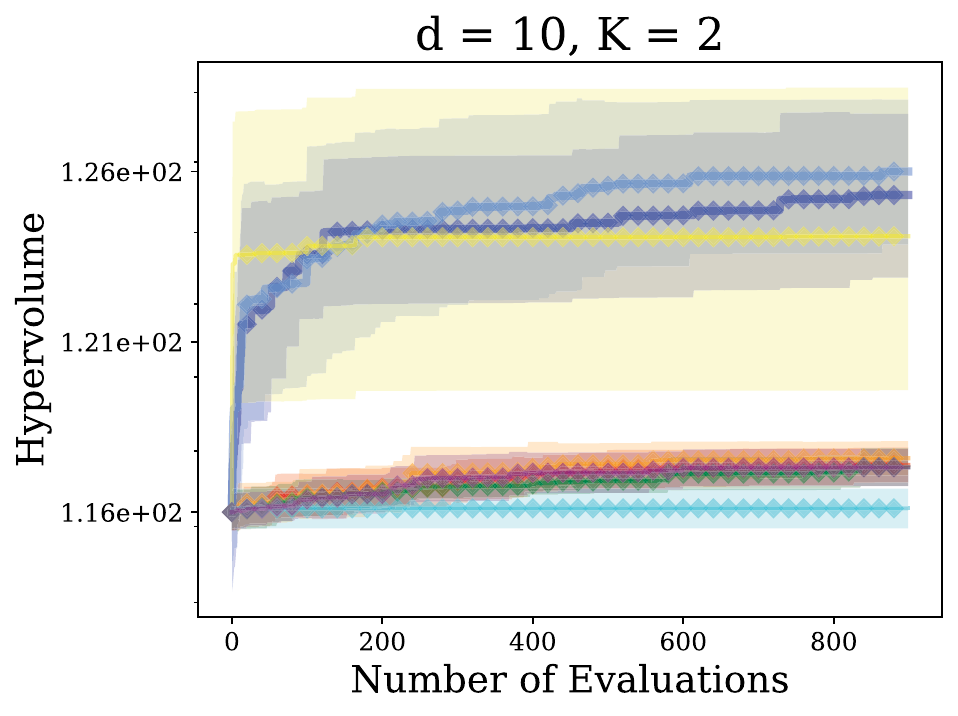}
        \caption{ZDT3}
    \end{subfigure}\\
    \begin{subfigure}{.25\textwidth}
        \includegraphics[width=1\linewidth]{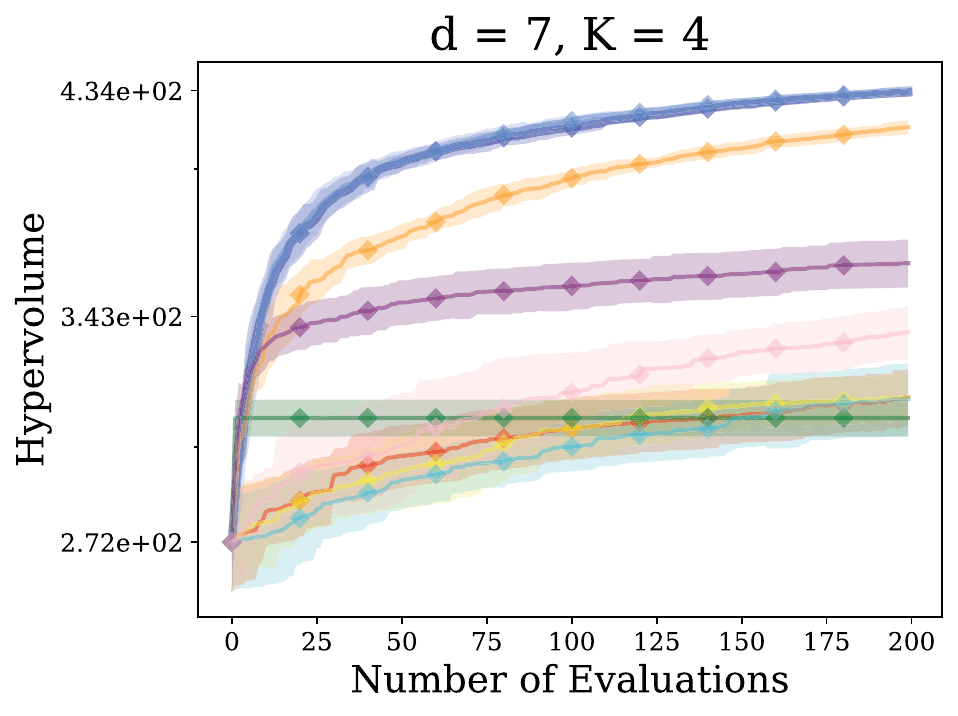}
        \caption{Car side-impact}
    \end{subfigure}%
    \begin{subfigure}{.25\textwidth}
        \includegraphics[width=1\linewidth]{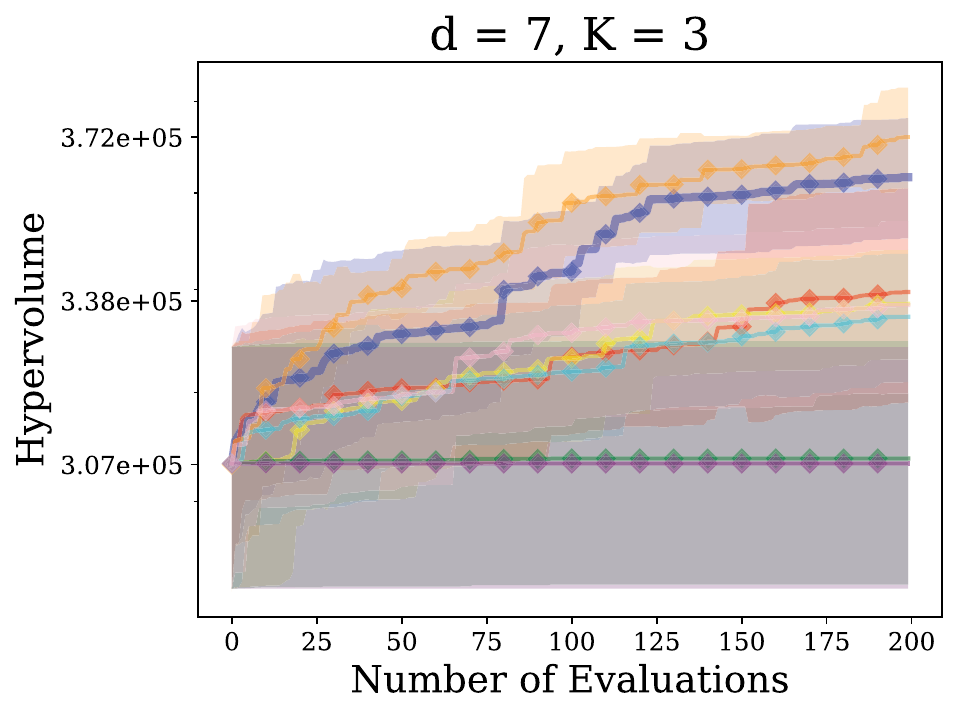}
        \caption{Penicillin production}
    \end{subfigure}%
    \begin{subfigure}{.25\textwidth}
        \includegraphics[width=1\linewidth]{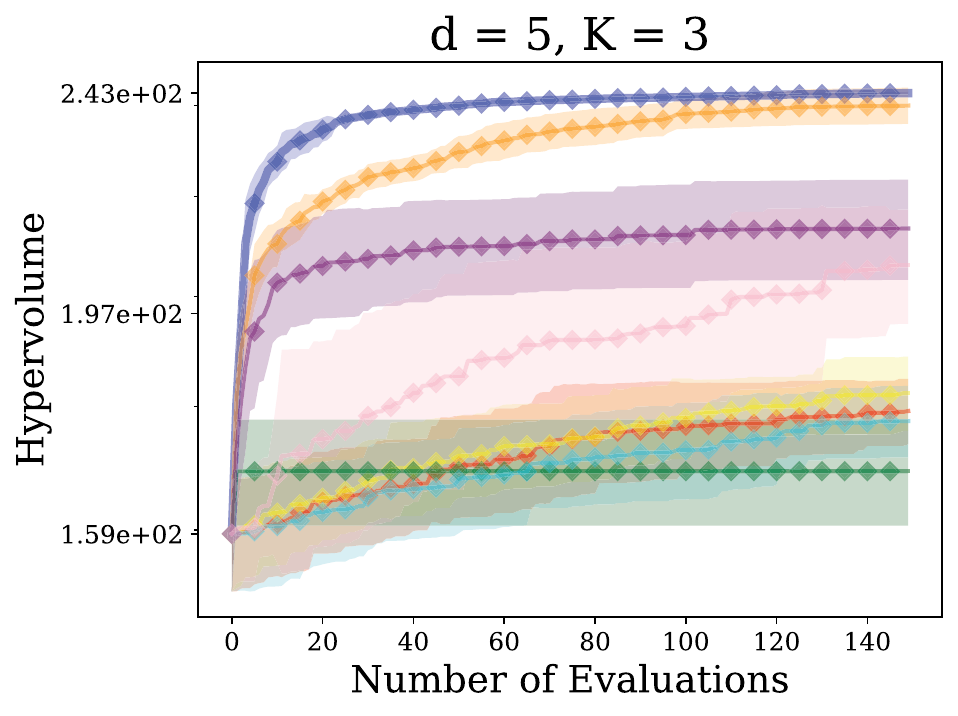}
        \caption{Vehicle design}
    \end{subfigure}%
    \begin{subfigure}{.25\textwidth}
        \includegraphics[width=1\linewidth]{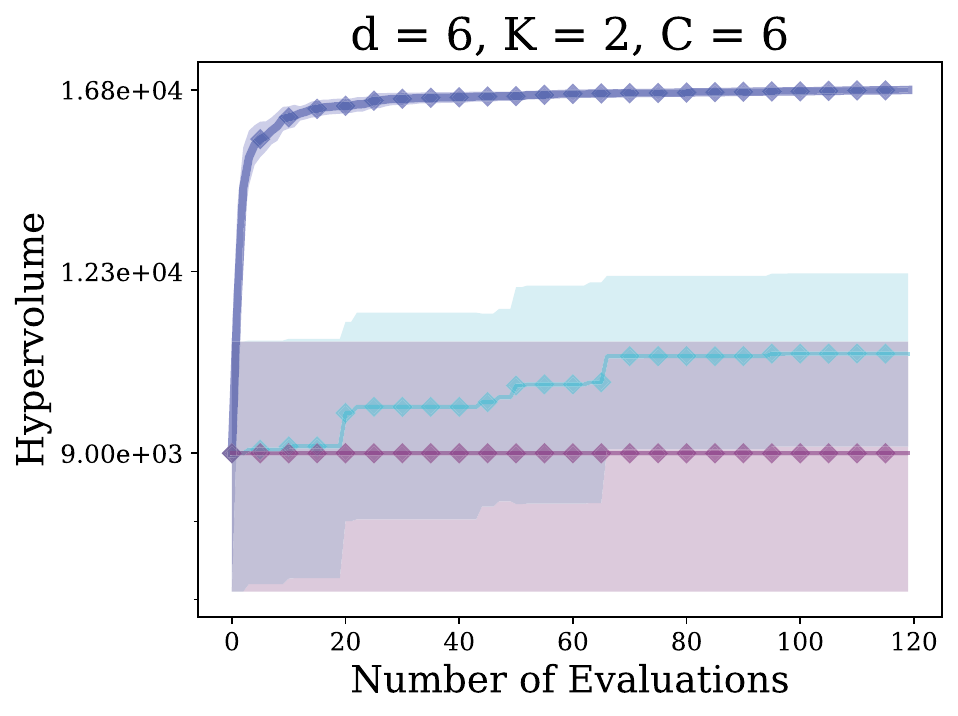}
        \caption{OSY (constrained)}
    \end{subfigure}\\
    \begin{subfigure}{1\textwidth}
        \includegraphics[width=1\linewidth]{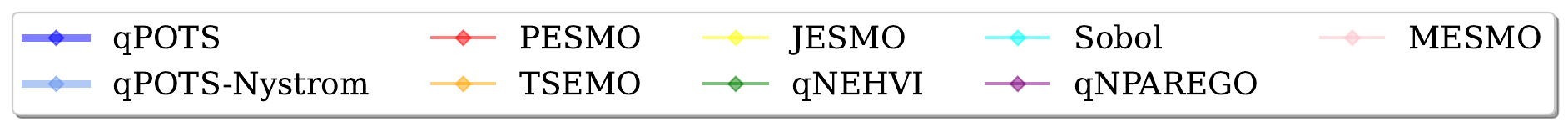}
    \end{subfigure}\\
    \caption{{\bf Sequential acquisition.} Hypervolume Vs. iterations for sequential ($q=1$) acquisition; plots show mean and $\pm 1$ standard deviation out of $10$ repetitions. \qpots~outperforms all competitors or is amongst the best.  Bottom right shows constrained case on the OSY problem~\cite{osyczka1995new}.
    Additional experiments are provided in the supplementary material.}
    \label{fig:hv_seq}
\end{figure*}

\begin{figure*}
    \centering
    \begin{subfigure}{.25\textwidth}
        \includegraphics[width=1\linewidth]{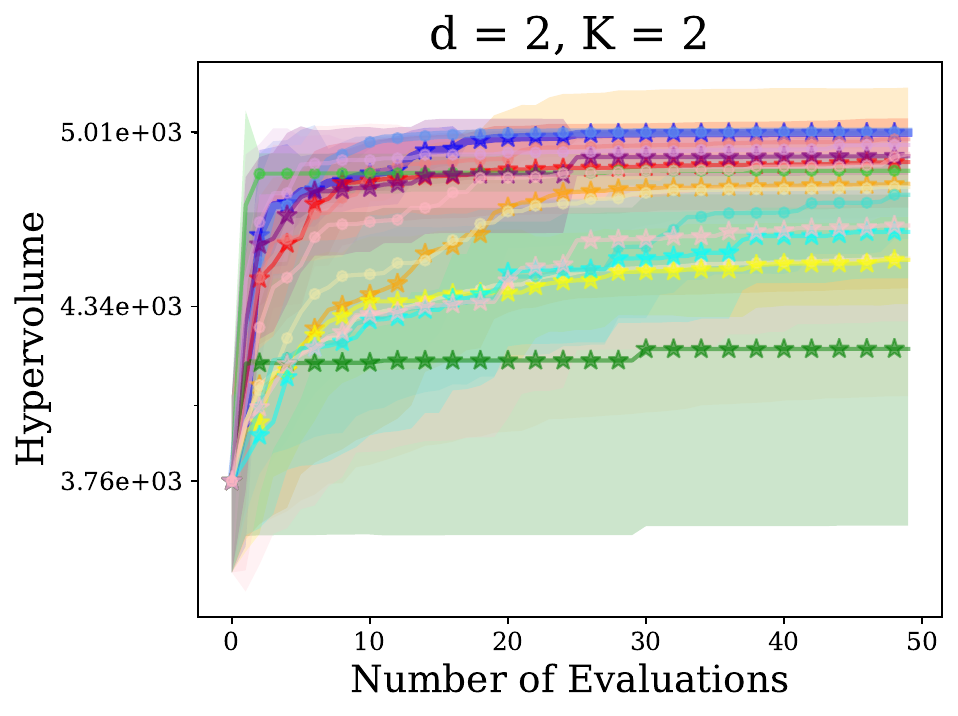}
        \caption{Branin-Currin}
    \end{subfigure}%
    \begin{subfigure}{.25\textwidth}
        \includegraphics[width=1\linewidth]{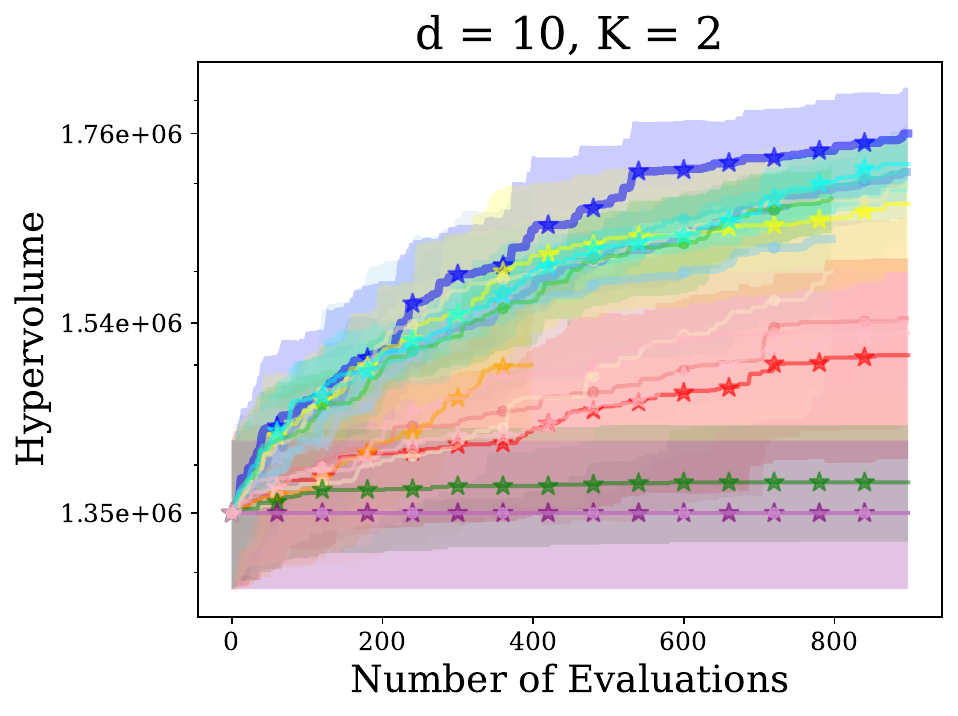}
        \caption{DTLZ3}
    \end{subfigure}%
    \begin{subfigure}{.25\textwidth}
        \includegraphics[width=1\linewidth]{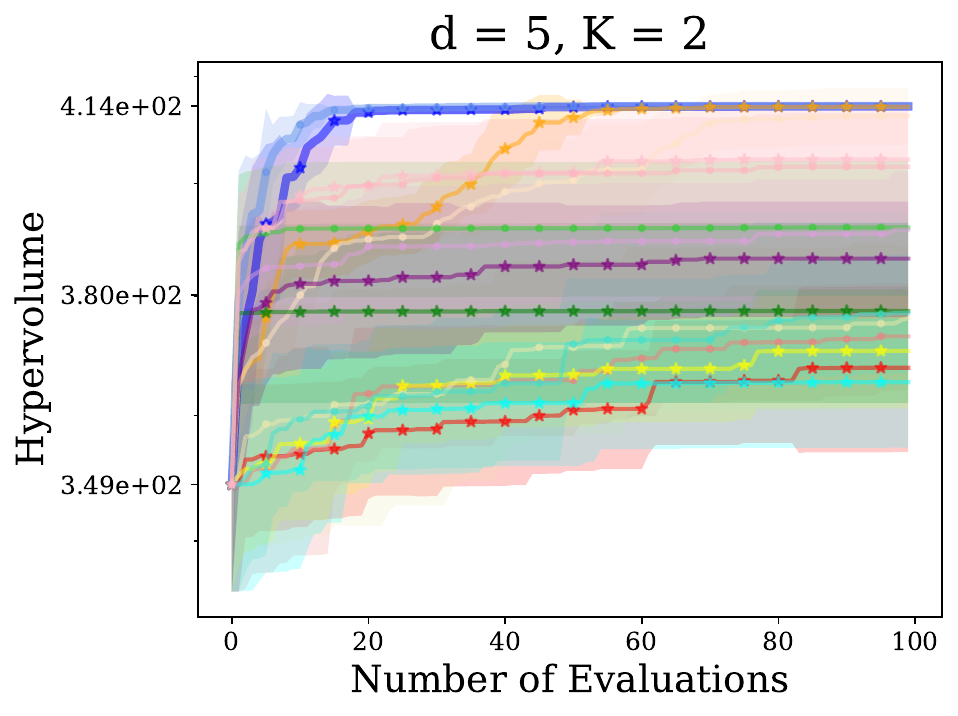}
        \caption{DTLZ7}
    \end{subfigure}%
    \begin{subfigure}{.25\textwidth}
        \includegraphics[width=1\linewidth]{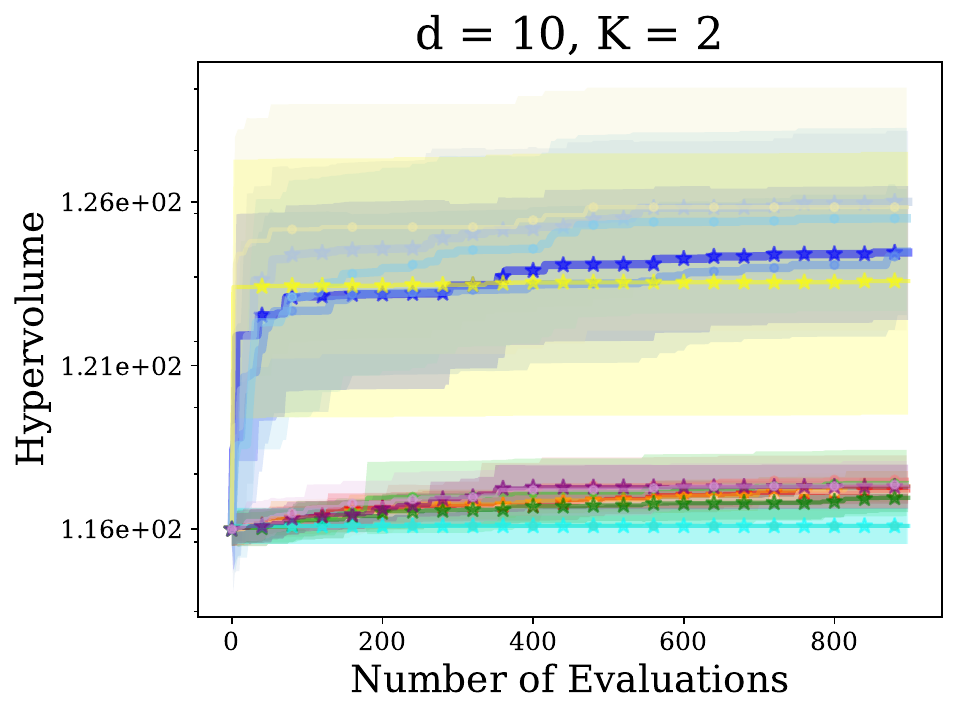}
        \caption{ZDT3}
    \end{subfigure}\\
    \begin{subfigure}{.25\textwidth}
        \includegraphics[width=1\linewidth]{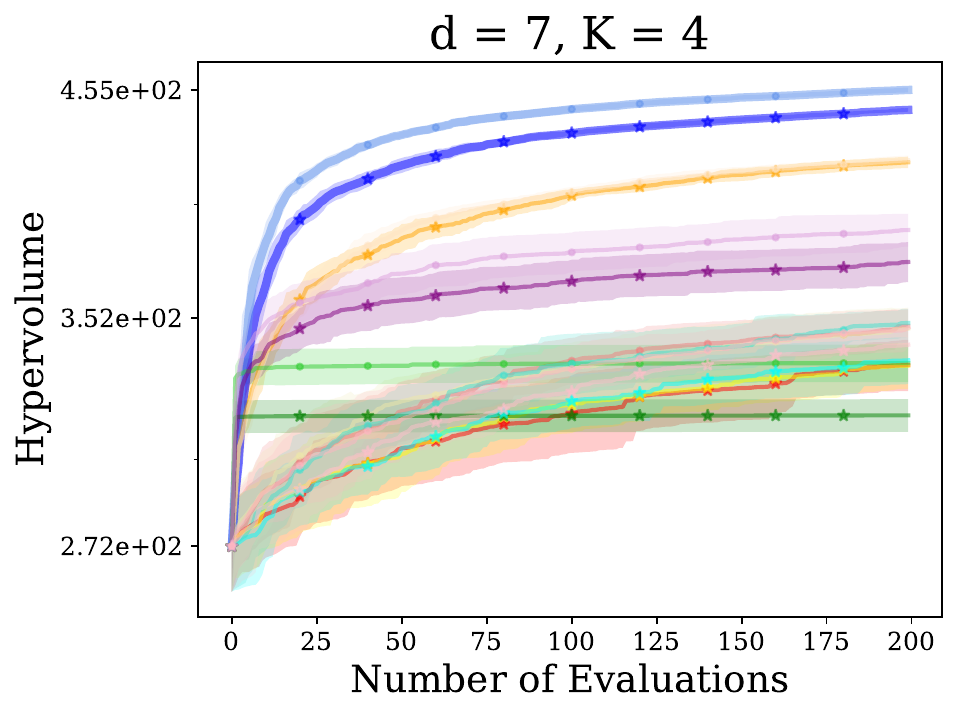}
        \caption{Car side-impact}
    \end{subfigure}%
    \begin{subfigure}{.25\textwidth}
        \includegraphics[width=1\linewidth]{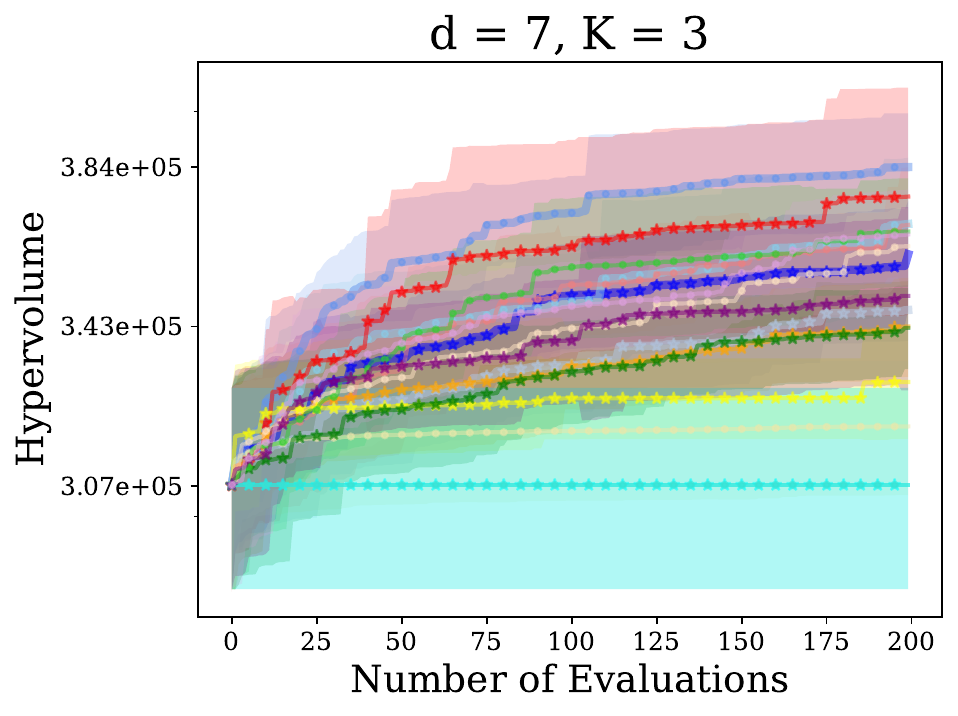}
        \caption{Penicillin production}
    \end{subfigure}%
    \begin{subfigure}{.25\textwidth}
        \includegraphics[width=1\linewidth]{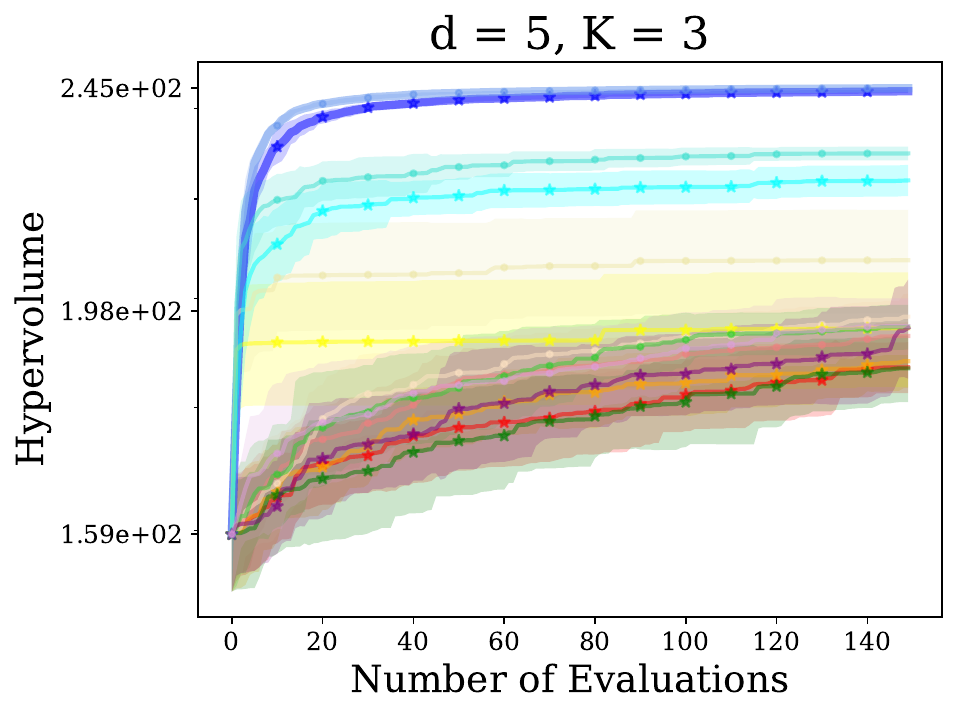}
        \caption{Vehicle design}
    \end{subfigure}%
    \begin{subfigure}{.25\textwidth}
        \includegraphics[width=1\linewidth]{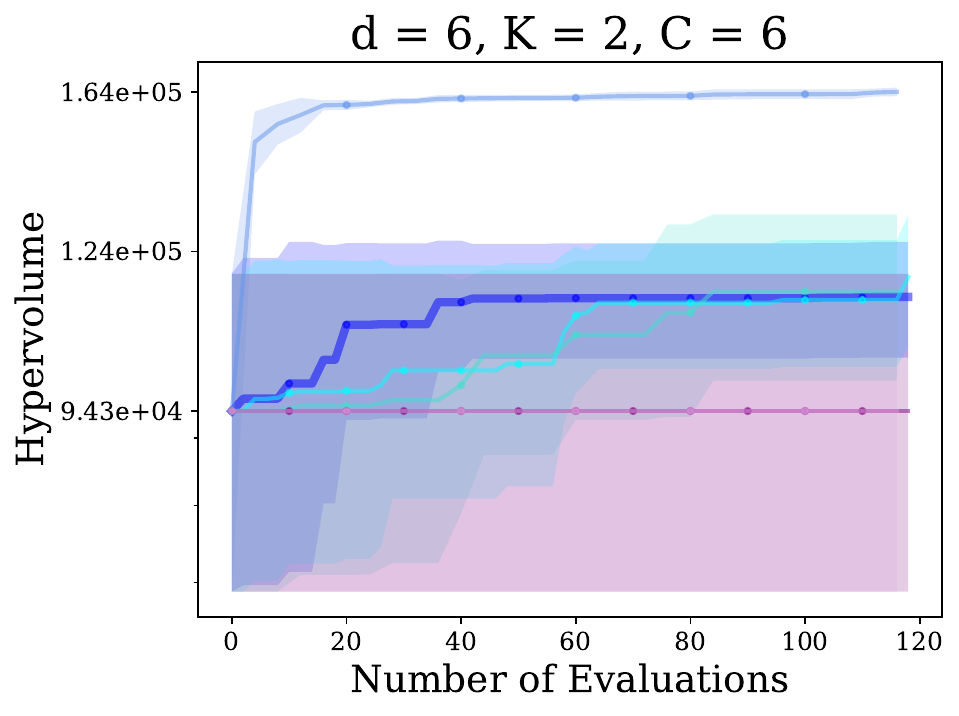}
        \caption{OSY (constrained)}
    \end{subfigure} \\
    \begin{subfigure}{1\textwidth}
        \includegraphics[width=1\linewidth]{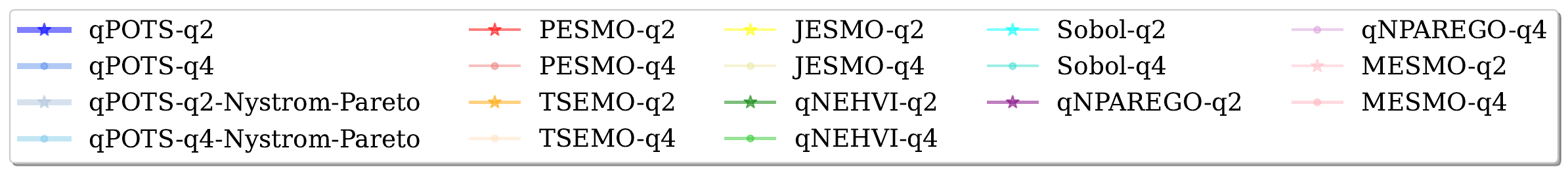}
    \end{subfigure}\\
    \caption{{\bf Batch acquisition.} Hypervolume Vs. iterations for batch ($q>1$) acquisition; plots show mean and $\pm 1$ standard deviation out of $10$ repetitions. \qpots~outperforms all competitors, but the benefit is more pronounced in the batch case. Additional experiments, including constrained problems, are shown in the supplementary material.}
    \label{fig:hv_batch}
\end{figure*}

\clearpage
\section*{Appendix}

\section{Proofs}
\label{sec:proofs}
We now provide theoretical guarantees on \qpots, with the following approach. We show that the absolute difference between $f_i$ and its GP posterior mean, $\forall \x\in\X$, is bounded for finite samples and is asymptotically consistent. Then, we show that the predicted Pareto frontier asymptotically converges to the true Pareto frontier. Our proofs depend on two key assumptions. First, we place regularity assumptions on $f_i,~\forall i$. Then, we assume that our inner (multiobjective) optimization solution, via evolutionary techniques, is solved exactly. We now proceed to present the key theoretical results.

\begin{assumption}[{\bf Exact knowledge of $X^*$}]
We assume that the inner multiobjective optimization problem $\argmax_{\x \in \X} \{Y_1(\x, \omega), \ldots, Y_K(\x, \omega)\},$ is solved exactly and the true $X^*$ is known. In practice, this assumption is reasonable and, to improve accuracy, can be fine-tuned with larger $\tilde{N}$ and $N_\T{gen}$.  
\end{assumption}

\begin{assumption}[{\bf Lipschitz}]
We assume that the sample paths drawn from the GP prior on $f$ are twice continuously differentiable. Such a GP is achieved by choosing the covariance kernel to be four times differentiable~\citep[Theorem 5]{ghosal2006posterior}, for  example, the squared-exponential or Matérn-class kernel. Then, the following holds (for $L \in \mbb{R}_+$): $\mP \left[ |f(\x) - f(\x')| > L \|\x - \x'\| \right] \leq e^{-L^2 / 2}.$
\end{assumption}

\subsection{Proof of Lemma 3.3}

\begin{lemma}
Let Assumption 3.2 hold. Let $\{\x_i\},~i=1,\ldots,n$ be a sequence of points selected via the \qpots~method according to Eq. (6) with $q=1$. Then, $\lim_{n \rightarrow \infty}~\sup_{\x \in \X} \sigma^2_n(\x) = 0$.
\end{lemma}
\begin{proof}[Proof]
First of all, note that \qpots~chooses new acquisitions according to the maximin distance criterion (Eq. (6)). Therefore, unless the posterior uncertainty is $0$ almost everywhere, the same acquisition is never chosen more than once.

W.l.o.g, let $k(\x, \x) = 1$. Recall that the posterior variance of the GP is given by
\[ \sigma_n^2(\x) = k(\x, \x) - \mbf{k}_n^\top \mbf{K}_n^{-1}\mbf{k}_n.\]
Note that as $n \rightarrow \infty$, then $\mbf{k}_n \rightarrow \mbf{K}^{:,i}_n$, where by $\mbf{K}^{:,i}_n$ we mean the $i$th column of $\mbf{K}_n$, $i \in [n]$. This is because, when the $f$ has been observed at infinitely many points $\{\x_i\}_{i=1}^\infty$ and recall that $X_n$ is the row stack of all of these points, then $\mbf{k}_n = k(\x, X_n)$ represents the covariance between a candidate point $\x$ and all previous observation sites $X_n$. When $n=\infty$, $\x$ is likely very close to one of the points $\x_i \in \{\x_i\}_{i=1}^\infty$ and hence $k(\x, X_n) \rightarrow k(\x_i, X_n),~i\in [n]$. And therefore $\mbf{k}_n^\top \mbf{K}_n^{-1} = \mbf{e}_i$, where $\mbf{e}_i$ is an $n$-vector with $1$ at the $i$th location and $0$ elsewhere. Then
\[ \lim_{n \rightarrow \infty} \sigma_n^2(\x) = 1 - \mbf{K}^{:,i~\top}_n \mbf{e}_i = 1 - k(\x_i, \x_i) = 0.\]
\end{proof}

Using the aforementioned assumptions and \Cref{lm:gpvariance}, we first show that the vector of posterior GPs $\mbf{Y}=[Y_1,\ldots,Y_K]^\top$ converges to the vector of true objectives $\mbf{f}=[f_1,\ldots,f_K]^\top$ with high probability.

\begin{theorem}[{\bf Convergence with high probability}]
Let $\delta \in (0, 1)$ and Assumption \ref{ass:lipschitz} hold with $L^2 = 2 \log \left(\f{2}{\delta}\right)$. Further, let $g$ be some affine function that operates on $f$. Let $\tilde{\X}$ denote a discretization of $\X$ with $|\tilde{\X}| < \infty$, where we evaluate the oracle $f$. Let $\epsilon \geq 0$, and $a_n = \sqrt{2\log(b_n |\tilde{\X}|)/\delta}$, where $\sum_{n=1}^\infty 1/b_n = 1$. Then the vector of posterior GP objectives converges to $\mbf{f}$ with high probability; that is, the following holds:

\[ \mP \left[\forall \x \in \X, ~ \lim_{n \rightarrow \infty}~ \|\mbf{f}(\x) - \mbf{Y}^{n-1}(\x)\|_\infty \leq \epsilon \right] > 1 - \delta. \]
\end{theorem}
\begin{proof}[Proof of Theorem 3.4]
 We first present a simple identity that bounds the probability that a Gaussian random variable is greater than some value. Given $r \sim \mcl{N}(0, 1)$ and $c > 0$, consider the following:
\begin{equation}
    \begin{split}
        \mP(r > c) =& \int_c^\infty \f{1}{\sqrt{2\pi}} e^{-r^2 / 2} dr \\
        = & e^{-c^2 /2}  \int_c^\infty \f{1}{\sqrt{2\pi}} e^{-(r - c)^2 / 2} e^{-c(r-c)} dr\\
        \leq & e^{-c^2 / 2} \int_c^\infty \f{1}{\sqrt{2\pi}} e^{-(r - c)^2 / 2} dr \leq \f{1}{2}e^{-c^2 / 2},
    \end{split}
\end{equation}
where the second line is due to the fact that $0 \leq e^{-c(r-c)} \leq 1$ for $r \geq c$ and $c>0$ and the last line from the fact that the integrand is the density of $\mcl{N}(c, 1)$. Using this result and the prior assumption that the unknown function $f$ has a GP prior, we state, with $c = a_n$ and a fixed $\x_i \in \tilde{X}$,
\begin{equation}
    \begin{split}
    \mP \left[ |f(\x_i) - \mu_{n-1}(\x_i)| > a_n \sigma_{n-1}(\x_{i}) \right] \leq & \f{1}{2} e^{-a_n^2 / 2} = \f{\delta}{2 b_n |\tilde{\mathcal{X}}|}. 
    \end{split}
\end{equation}
Applying the union bound $\forall i$, we get
\begin{equation}
    \begin{split}
    \mP \bigcup_{i=1}^{|\tilde{\mathcal{X}}|} &|f(\x_i) -\mu_{n-1}(\x_i)| \\
        &> a_n \sigma_{n-1}(\x_{i}) \\
        &\leq \sum_{i=1}^{|\tilde{\mathcal{X}}|}  \mP \left[ |f(\x_i) - \mu_{n-1}(\x_i)| > a_n \sigma_{n-1}(\x_{i}) \right]\\
        &\leq \sum_{i=1}^{|\tilde{\mathcal{X}}|} \f{\delta}{2 b_n |\tilde{\mathcal{X}}|} = \f{\delta}{2 b_n}.
    \end{split}
\end{equation}
Now we apply the union bound $\forall n$:
\begin{equation}
    \begin{split}
        & \mP \left[\forall \tilde{\mathcal{X}} \in \tilde{\mathcal{X}},~ \bigcup_{n=1}^{\infty} |f(\tilde{\mathcal{X}}) - \mu_{n-1}(\tilde{\mathcal{X}})| > a_n \sigma_{n-1}(\tilde{\mathcal{X}}) \right] \\
        &\leq \sum_{n=1}^{\infty} \f{\delta}{2 b_n} = \f{\delta}{2}.    \\  
    \end{split}
    \label{e:bound_for_discreteX}
\end{equation}
We would like to extend the proof to a general $\x \in \X$. For some $\x \in \X$, let $\tilde{\mathcal{X}}$ denote a point $\in \tilde{\mathcal{X}}$ that is nearest (in Euclidean distance) to $\x$. Then, per Assumption~3.2 and by setting $L^2 = 2 \log \left(\f{2}{\delta}\right)$, the following statement holds:
\begin{equation}
    \mP \left[ |f(\x) - f(\tilde{\x})| > \sqrt{2 \log \left(\f{2}{\delta}\right)} \|\x - \tilde{\x}\| \right] \leq \f{\delta}{2}.
    \label{e:lipshictz_bound}
\end{equation}
Combining \eqref{e:bound_for_discreteX} and \eqref{e:lipshictz_bound}, the following statement holds:
\begin{equation}
    \begin{split}
        & \mP \left[ \forall \x \in \X, \forall n,~ |f(\x) - \mu_{n-1}(\tilde{\x})| \leq \right. \\
        & \left. a_n \sigma_{n-1}(\x) + \sqrt{2 \log \left(\f{2}{\delta}\right)} \|\x - \tilde{\x}\| \right] > 1 - \delta.
    \end{split}
\end{equation}
From Lemma 3.2, the $\lim_{n \rightarrow \infty} \sup_{\x \in \X} \sigma_{n-1}(\x) = 0$, and thus
\[ \mP \left[\forall \x \in \X, \lim_{n\rightarrow \infty}|f(\x) - \mu_{n-1}(\tilde{\x})| \leq \epsilon \right] > 1 - \delta, \]
where $\epsilon = \sqrt{2 \log \left(\f{2}{\delta}\right)} \|\x - \tilde{\x}\|$. Circling back to the reparametrization trick, the above statement can be restated as
\[
\begin{split}
\mP \left[\forall \x \in \X, \lim_{n\rightarrow \infty}|f(\x) - Y^{n-1}(\tilde{\x}) + \sigma_{n-1}(\x)z| \leq \epsilon \right]&\\ > 1 - \delta \\
\mP \left[\forall \x \in \X, \lim_{n\rightarrow \infty}|f(\x) - Y^{n-1}(\tilde{\x})| \leq \epsilon \right]& \\
> 1 - \delta,
\end{split}
\]
where $Y^{n-1}$ is the posterior GP after the $(n-1)$th round, $z$ is the standard normal random variable, and the second line follows from Lemma 3.3.
By extension, the above statement holds for all objectives $f_i,~\forall i \in [K]$. Therefore, the above statement can be generalized as 
\[ \mP \left[\forall \x \in \X, \lim_{n\rightarrow \infty}\|\mbf{f}(\x) - \mbf{Y}^{n-1}(\tilde{\x})\|_\infty \leq \epsilon \right] > 1 - \delta. \]
\end{proof}

Due to \Cref{thm:ass_consistency}, the Pareto frontier computed with the GP posteriors converges to true Pareto frontier with high probability. Note that \Cref{subfig:assymp} confirms this empirically.

\section{Exploration-exploitation tradeoff}
In \Cref{fig:bc_demo}, we demonstrate the exploration-exploitation tradeoff in \qpots. The black crosses represent the true Pareto frontier/set and the blue square represents the \qpots~acquisition. Notice that the acquisition is far away from the brown training points (top middle), demonstrating exploration, and closer to the training points (top left and right), demonstrating exploitation. The right plots (top and bottom) demonstrate the \qpots~is asymptotically consistent by converging to the true solution.
\begin{figure*}[ht!]
    \centering
    \begin{subfigure}{0.33\textwidth}
        \includegraphics[width=1\linewidth]{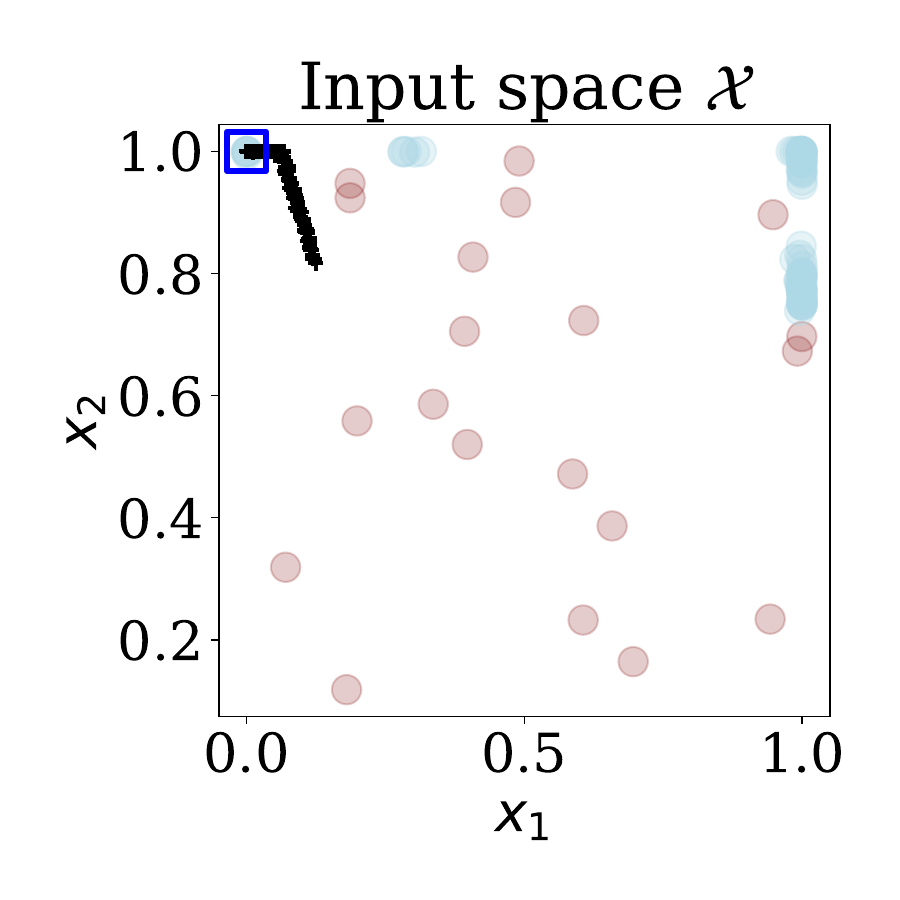}
        \caption{$n=20$}
    \end{subfigure}%
    \begin{subfigure}{0.33\textwidth}
        \includegraphics[width=1\linewidth]{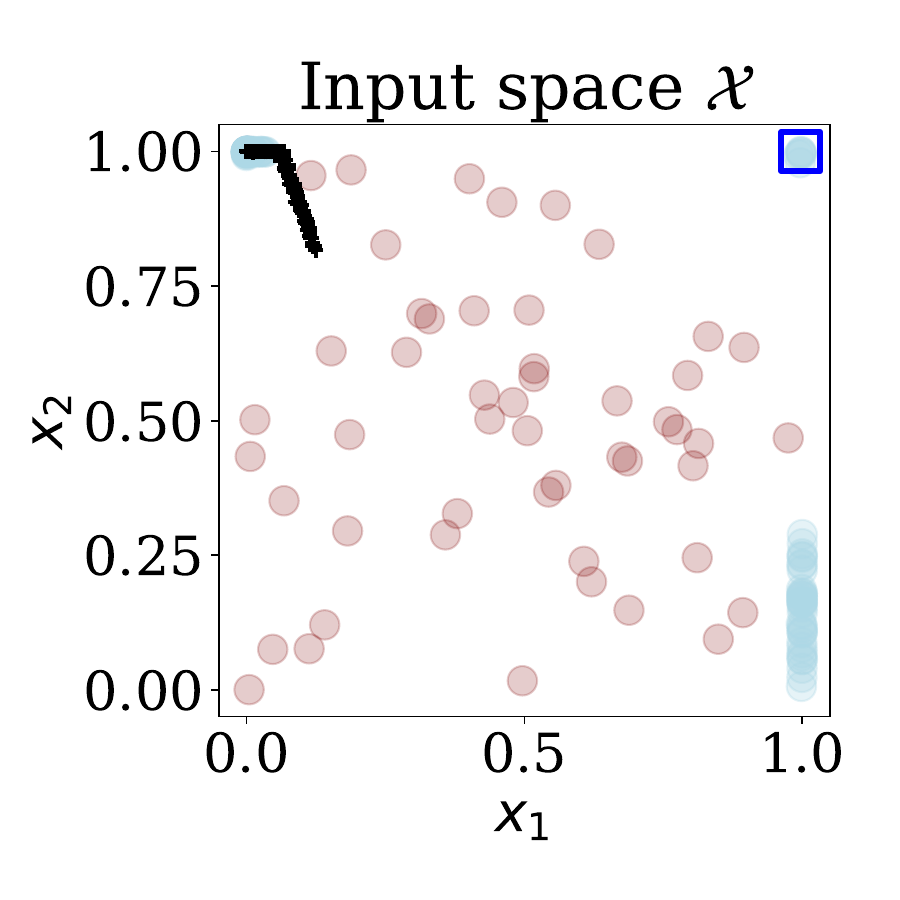}
        \caption{$n=50$}
    \end{subfigure}%
    \begin{subfigure}{0.33\textwidth}
        \includegraphics[width=1\linewidth]{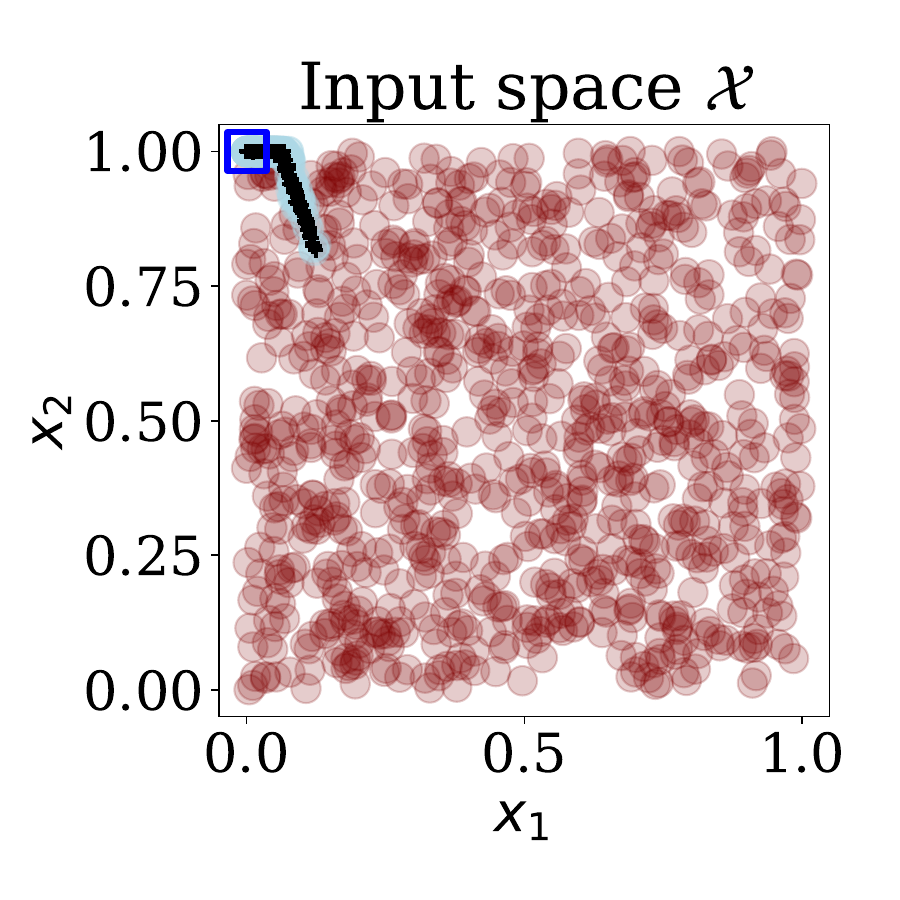}
        \caption{$n=900$}
    \end{subfigure}\\
    \begin{subfigure}{0.33\textwidth}
        \includegraphics[width=1\linewidth]{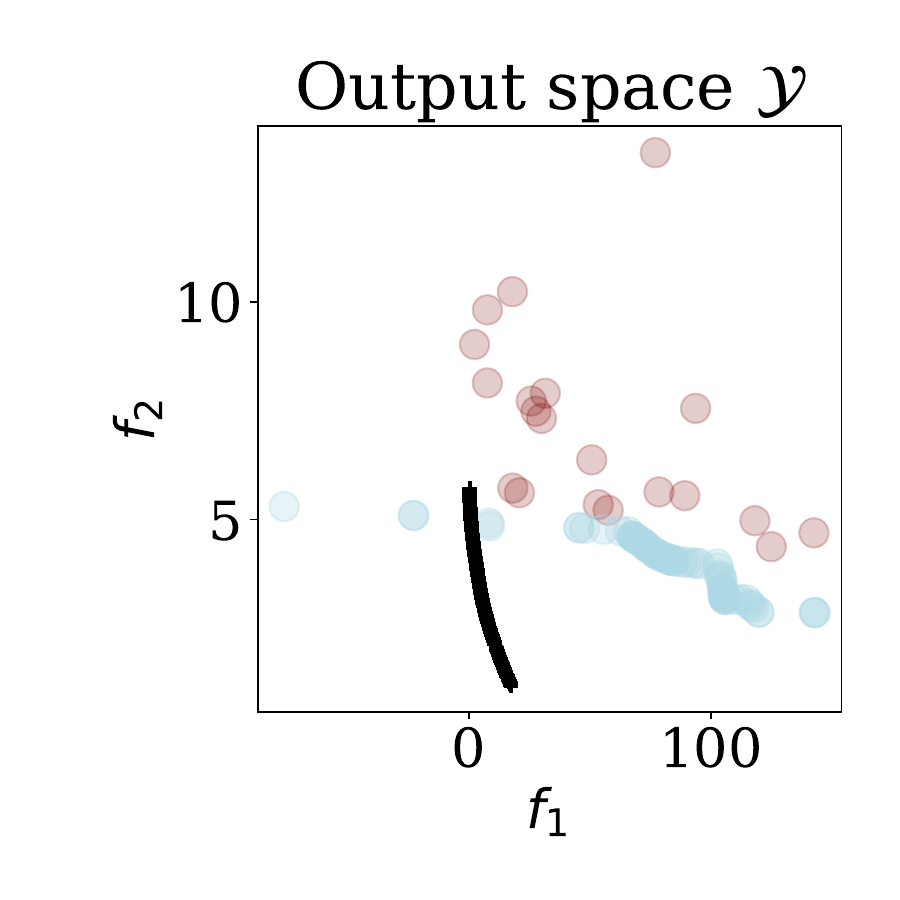}
        \caption{$n=20$}
    \end{subfigure}%
    \begin{subfigure}{0.33\textwidth}
        \includegraphics[width=1\linewidth]{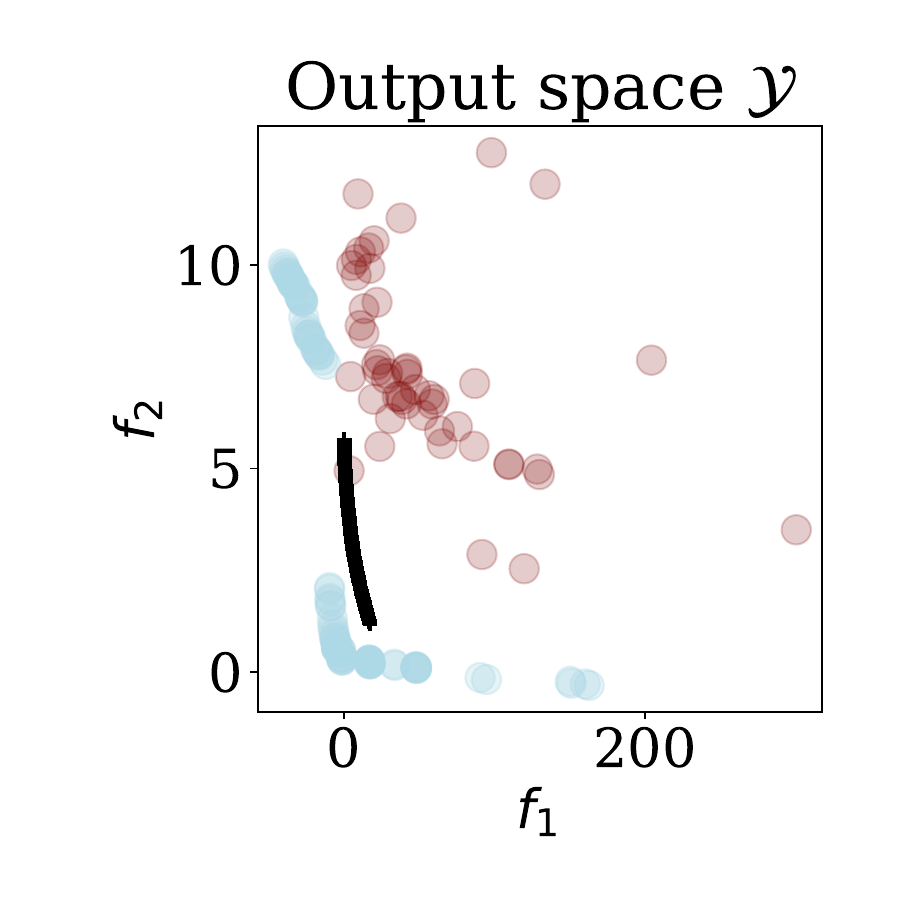}
        \caption{$n=50$}
    \end{subfigure}%
    \begin{subfigure}{0.33\textwidth}
        \includegraphics[width=1\linewidth]{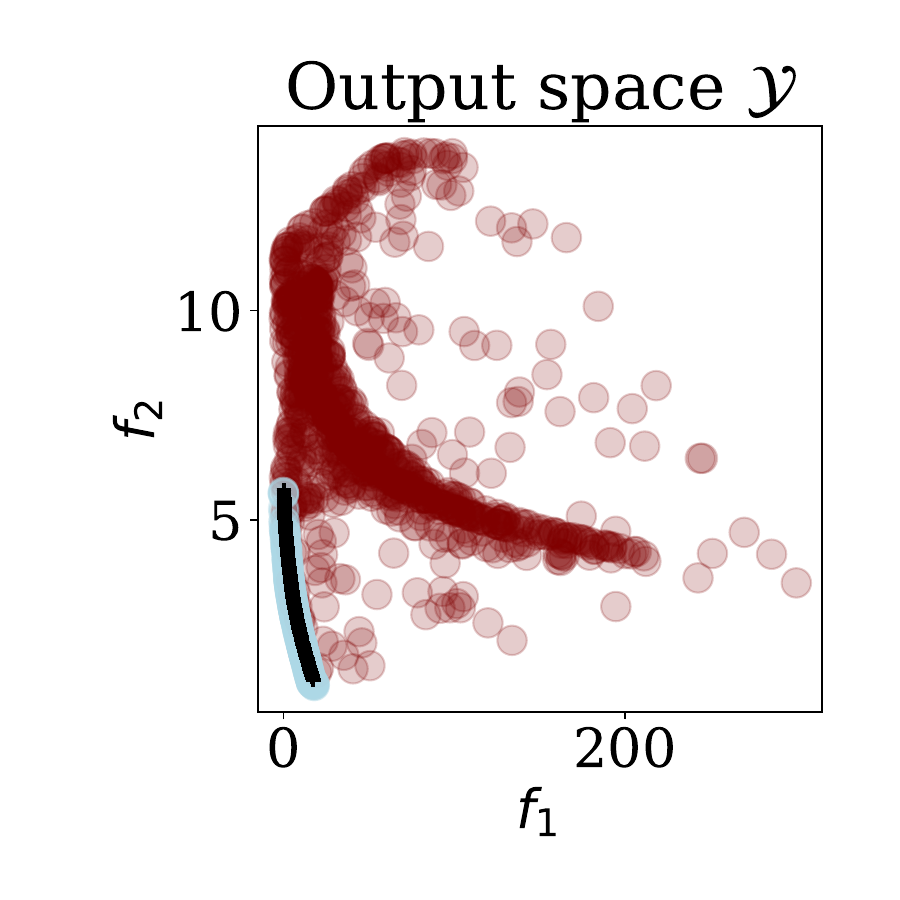}
        \caption{$n=900$}
        \label{subfig:assymp}
    \end{subfigure}\\ 
    \caption{Demonstration on the Branin-Currin ($d=2, K=2$) test function. Maroon: training points, blue circles: GP posterior samples, blue square: \qpots~acquisition, black: true Pareto frontier/set. The input space plots show how \qpots~explores and exploits the space (based on distance to maroon). The output space plots show how \qpots~asymptotically converges to the solution.}
    \label{fig:bc_demo}
\end{figure*}

\section{Additional experiments}
\begin{figure*}[htb!]
    \centering
    \begin{subfigure}{0.5\textwidth}
        \includegraphics[width=1\linewidth]{newest_figs/osy.pdf}
        \caption{OSY}
    \end{subfigure}%
    \begin{subfigure}{0.5\textwidth}
        \includegraphics[width=1\linewidth]{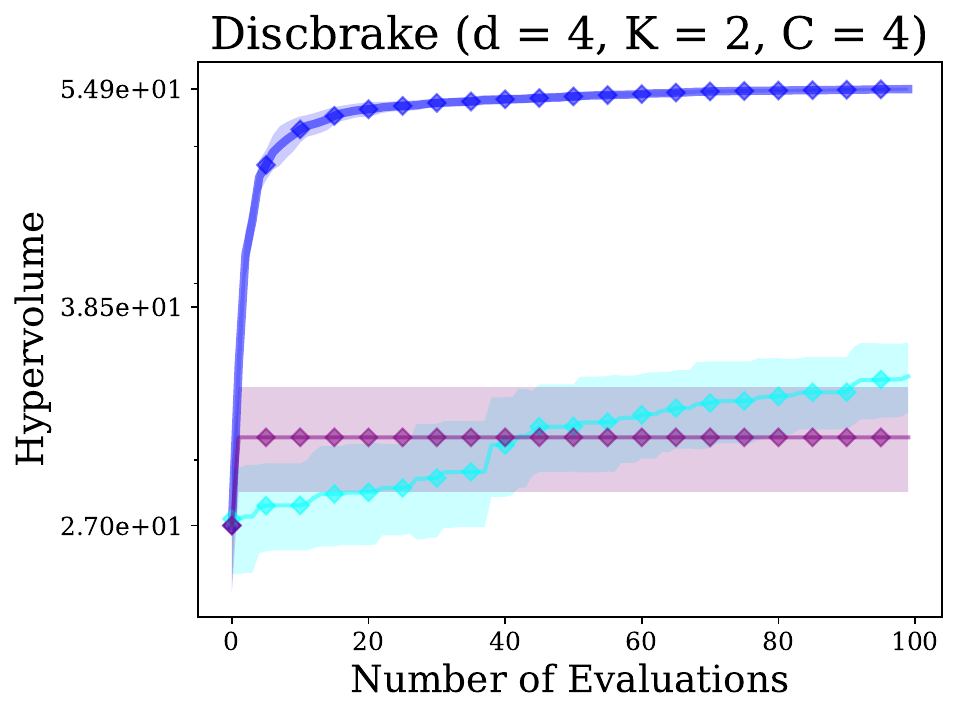}
        \caption{Discbrake}
    \end{subfigure}\\
    \begin{subfigure}{1\textwidth}
        \includegraphics[width=1\linewidth]{newest_figs/legend_seq.pdf}
    \end{subfigure}
    \begin{subfigure}{0.5\textwidth}
        \includegraphics[width=1\linewidth]{newest_figs/osy_batch.pdf}
        \caption{OSY}
    \end{subfigure}%
    \begin{subfigure}{0.5\textwidth}
        \includegraphics[width=1\linewidth]{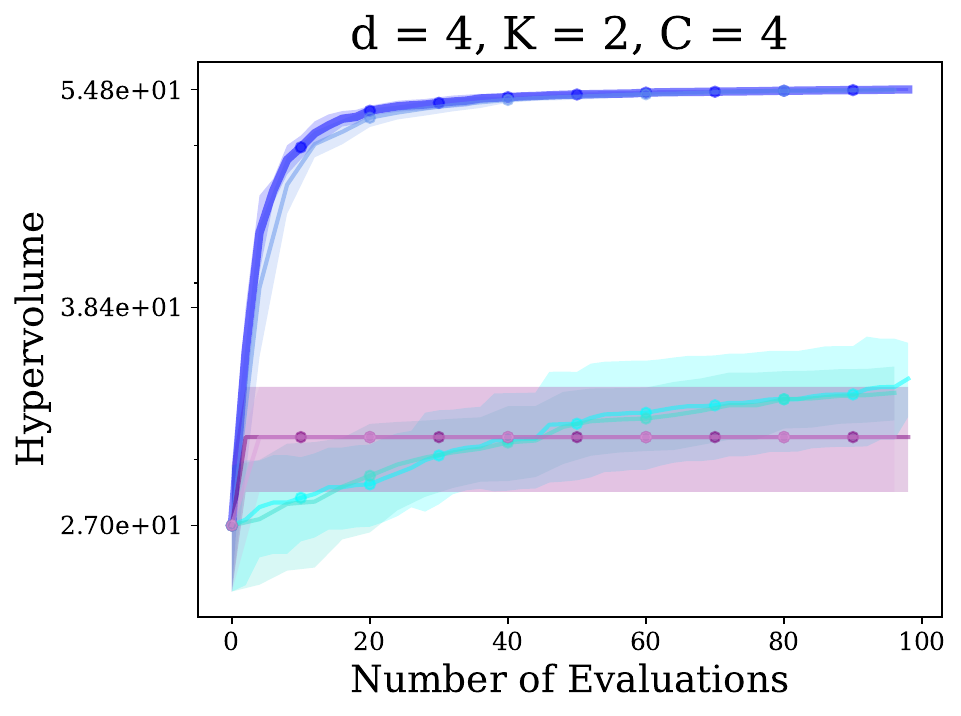}
        \caption{Discbrake}
    \end{subfigure}\\
        \begin{subfigure}{1\textwidth}
        \includegraphics[width=1\linewidth]{newest_figs/legend_batch.pdf}
    \end{subfigure}
    \caption{Additional {\bf constrained real-world experiments}; top row: sequential, bottom row: batch acquisitions. Note that entropy based acquisition functions (PESMO, MESMO, JESMO) and TSEMO do not support constraints. We were not able to get any results with $qNEHVI$ due to memory overflow issues. Therefore, we have compared against qNParEGO and Sobol. Note that we ran qPOTS-Nystrom-Pareto only for $d\geq 10$ cases as it is not necessary for low dimensional cases.}
    \label{fig:cons}
\end{figure*}

We show the performance on two constrained problems: the OSY problem~\cite{osyczka1995new} with $2$ objectives and $6$ constraints, and the discbrake~\cite{tanabe2020easy} problem with $2$ objectives and $4$ constraints. Note that, TSEMO as well as the entropy based acquisition functions don't support constraints. We had trouble running constrained problems with qNEHVI---the code took extraordinarily long time to execute and we had memory overflow issues. We perceive this as a computational bottleneck in qNEHVI. On the other hand, \qpots~requires no special modifications for handling constraints and the computational overhead is only due to fitting additional GPs for the constraints. We consider this one of the biggest strengths of \qpots.

We also show additional synthetic experiments in \Cref{fig:synthetic_batch} which includes high-dimensional experiments up to $d=20$. Again, \qpots~is either the best or amongst the best in all cases.

\begin{figure*}[htb!]
    \centering
    \begin{subfigure}{0.33\textwidth}
        \includegraphics[width=1\linewidth]{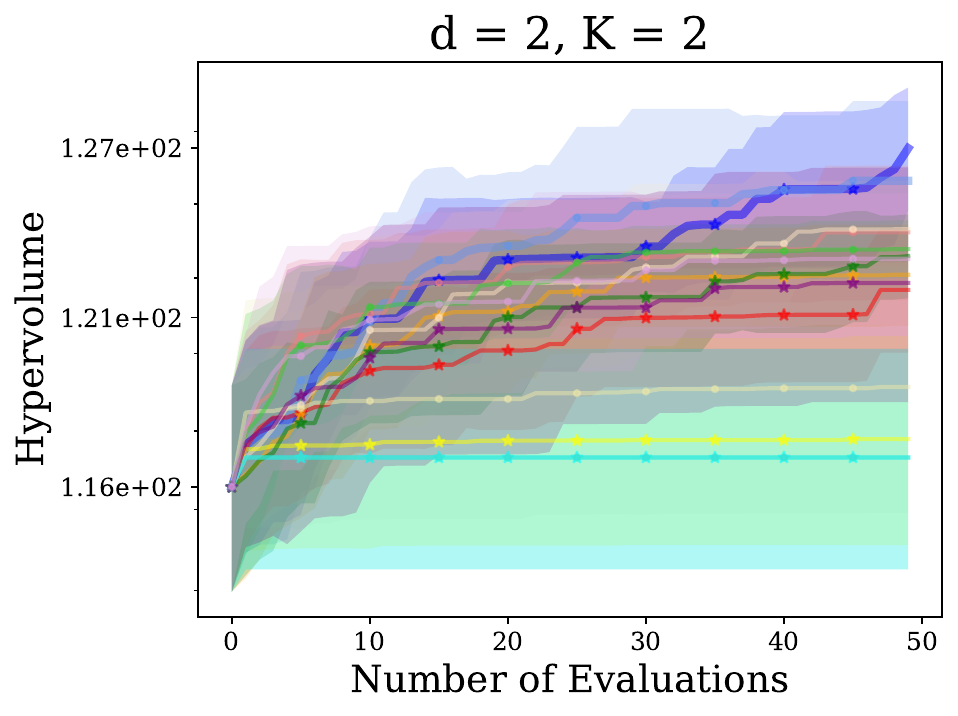}
        \caption{ZDT3}
    \end{subfigure}%
    \begin{subfigure}{0.33\textwidth}
        \includegraphics[width=1\linewidth]{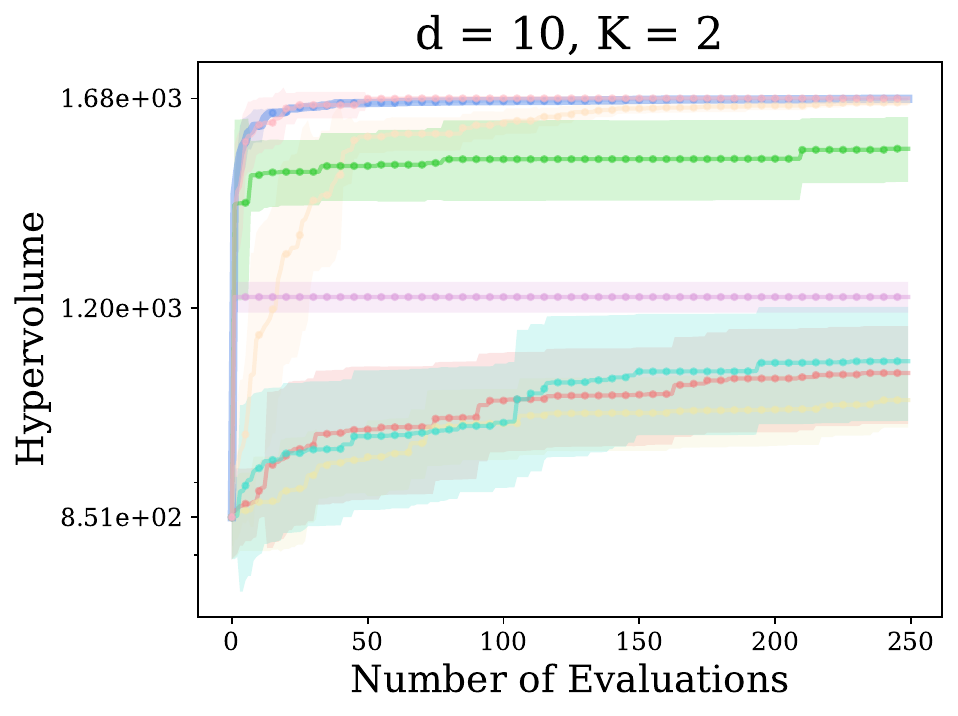}
        \caption{DH3($10$d)}
    \end{subfigure}%
    \begin{subfigure}{0.33\textwidth}
        \includegraphics[width=1\linewidth]{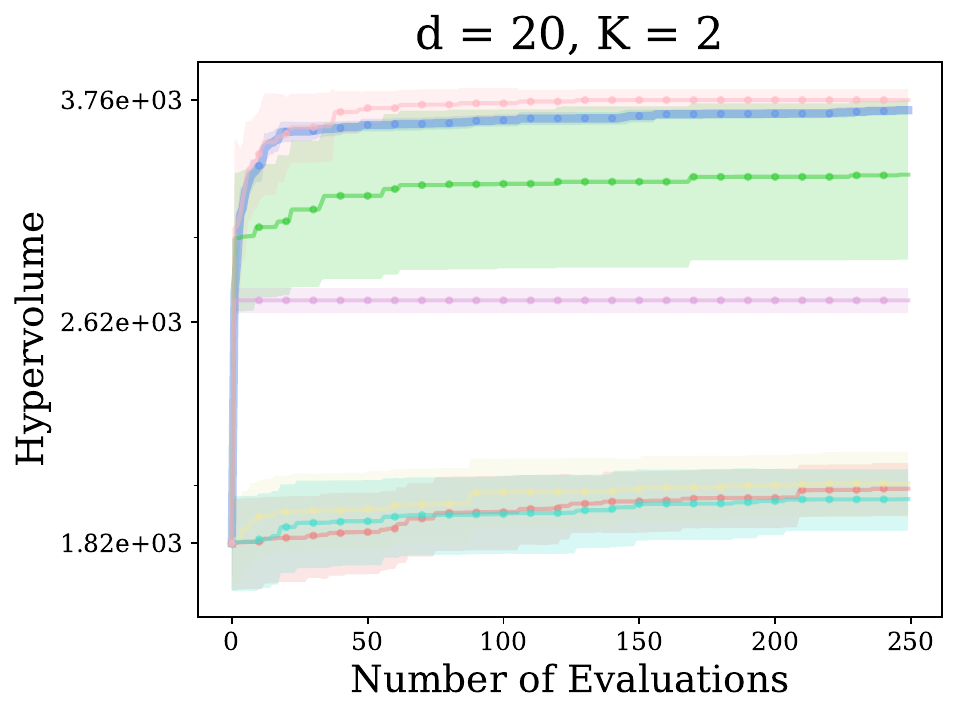}
        \caption{DH3($20$d)}
    \end{subfigure}\\
    \begin{subfigure}{1\textwidth}
        \includegraphics[width=1\linewidth]{newest_figs/legend_batch.pdf}
    \end{subfigure}\\
    \caption{Additional {\bf synthetic batch experiments}. \qpots~is either the best or amongst the best.}
    \label{fig:synthetic_batch}
\end{figure*}

\begin{figure*}[htb!]
    \centering
    \begin{subfigure}{0.33\textwidth}
        \includegraphics[width=1\linewidth]{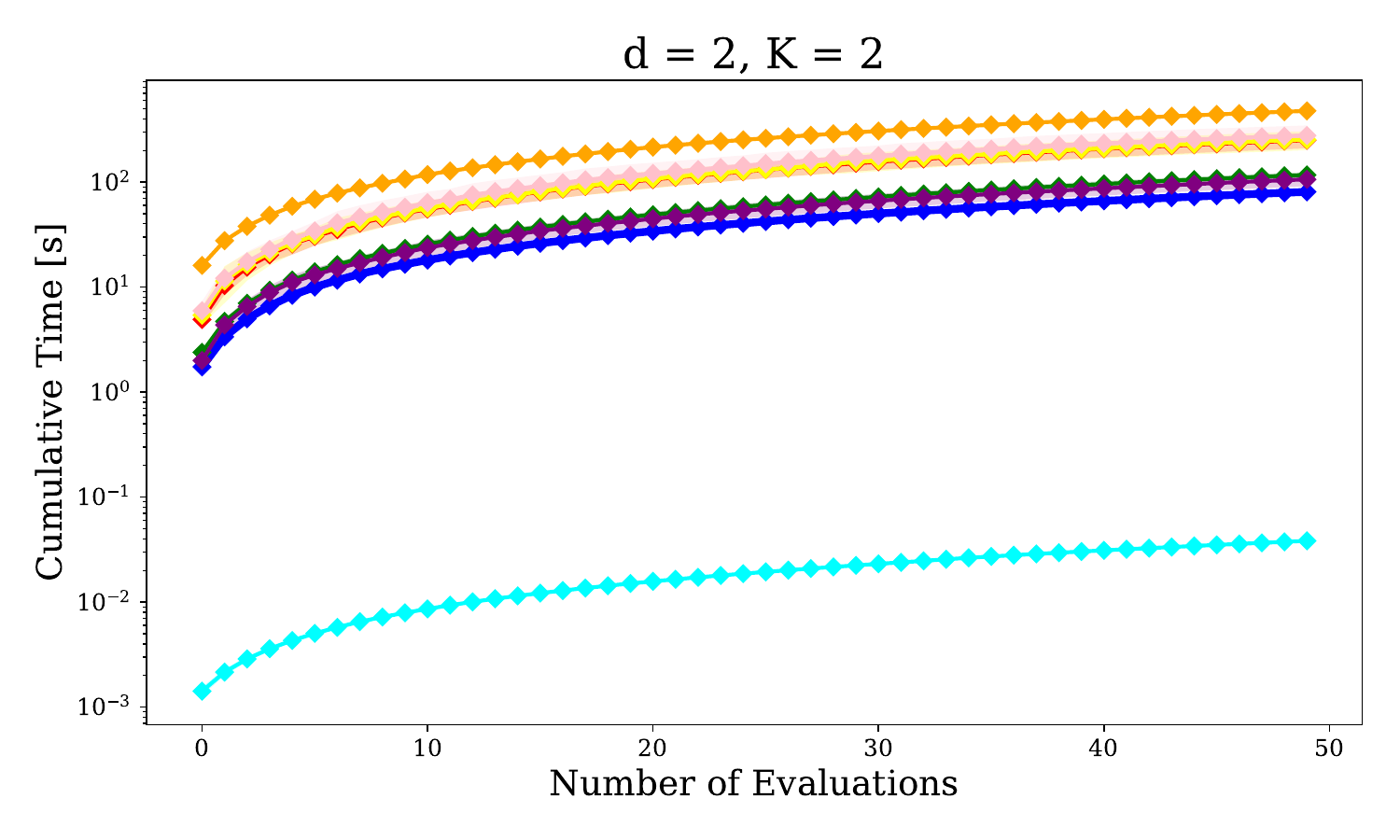}
        \caption{Branin-Currin}
    \end{subfigure}%
    \begin{subfigure}{0.33\textwidth}
        \includegraphics[width=1\linewidth]{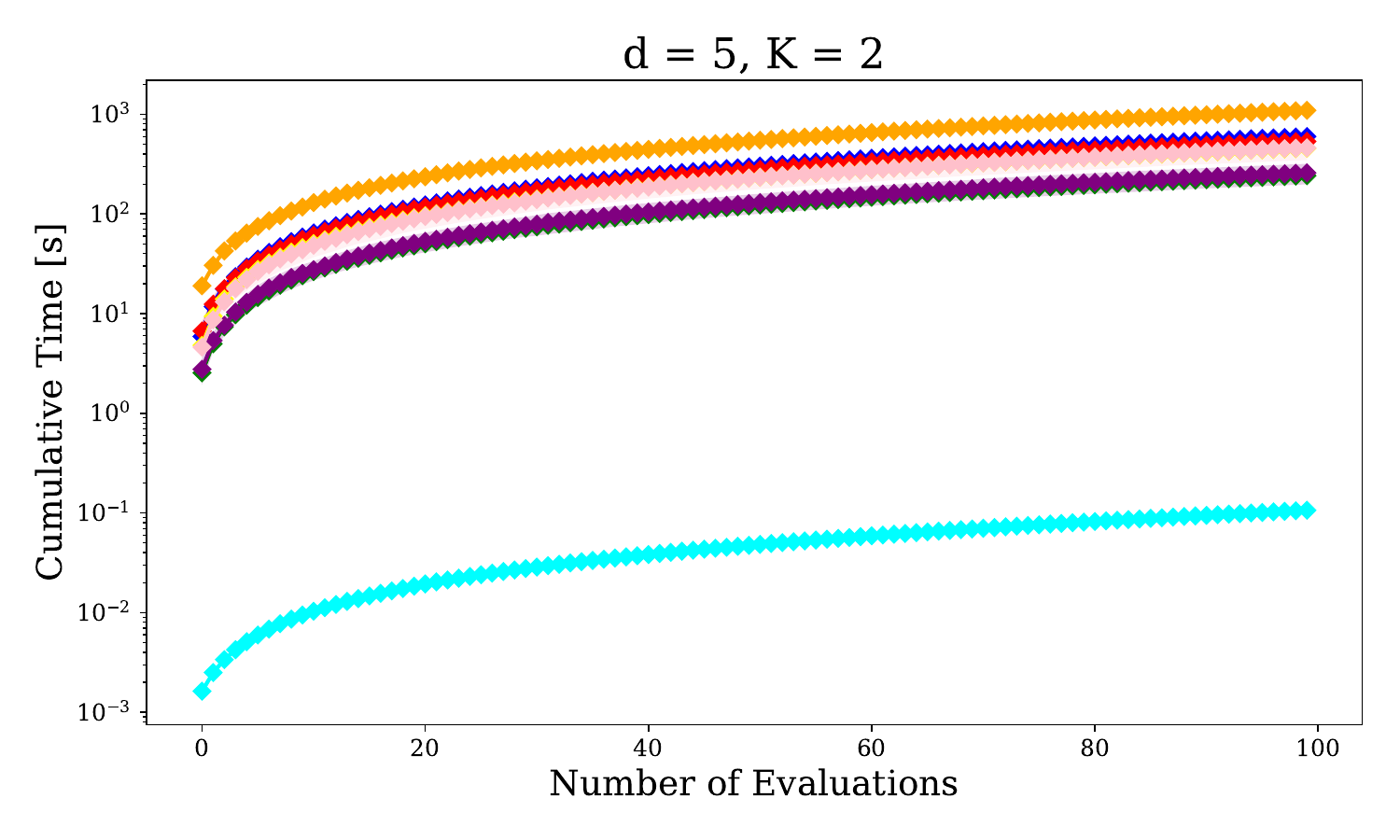}
        \caption{DTLZ7}
    \end{subfigure}%
    \begin{subfigure}{0.33\textwidth}
        \includegraphics[width=1\linewidth]{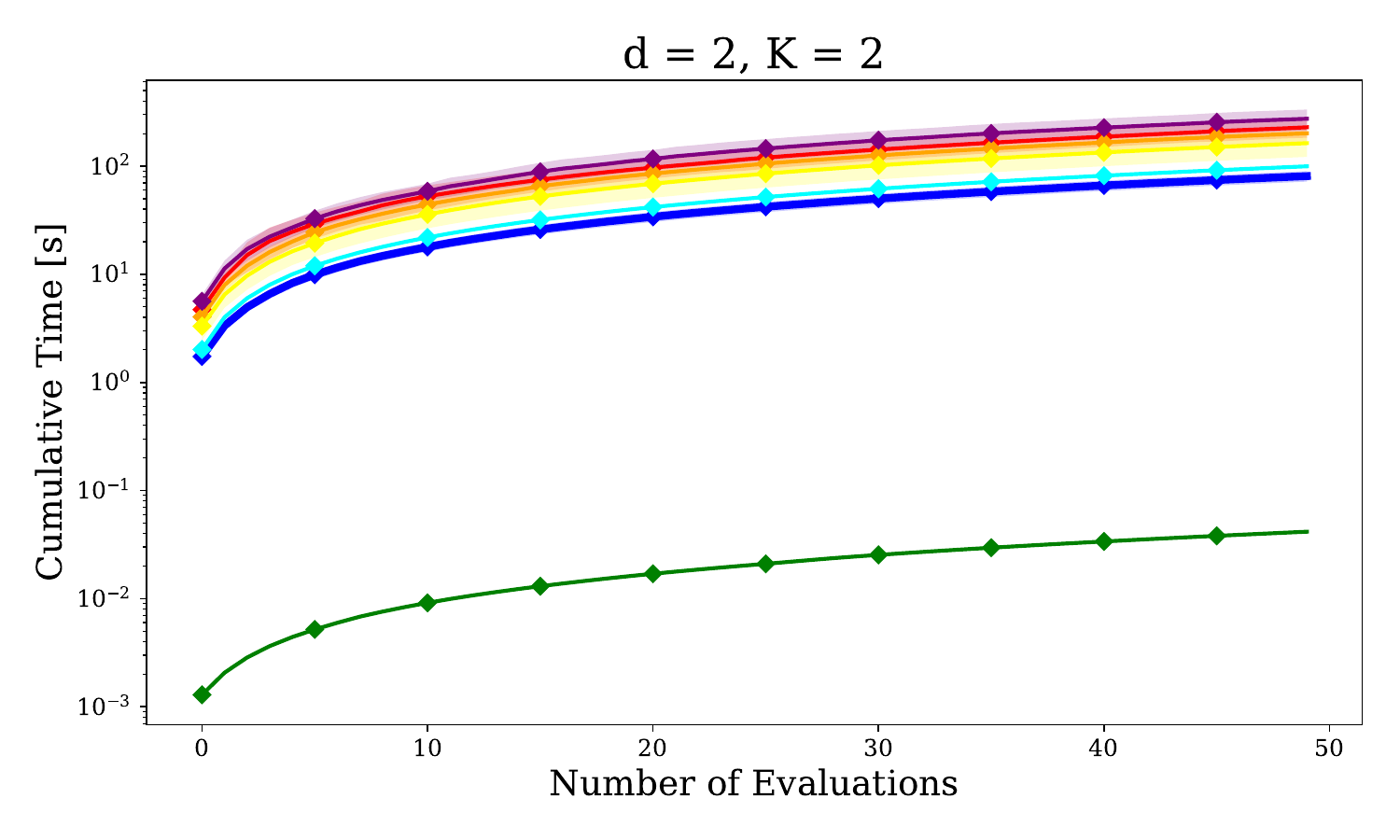}
        \caption{ZDT3}
    \end{subfigure}\\
    \begin{subfigure}{0.33\textwidth}
        \includegraphics[width=1\linewidth]{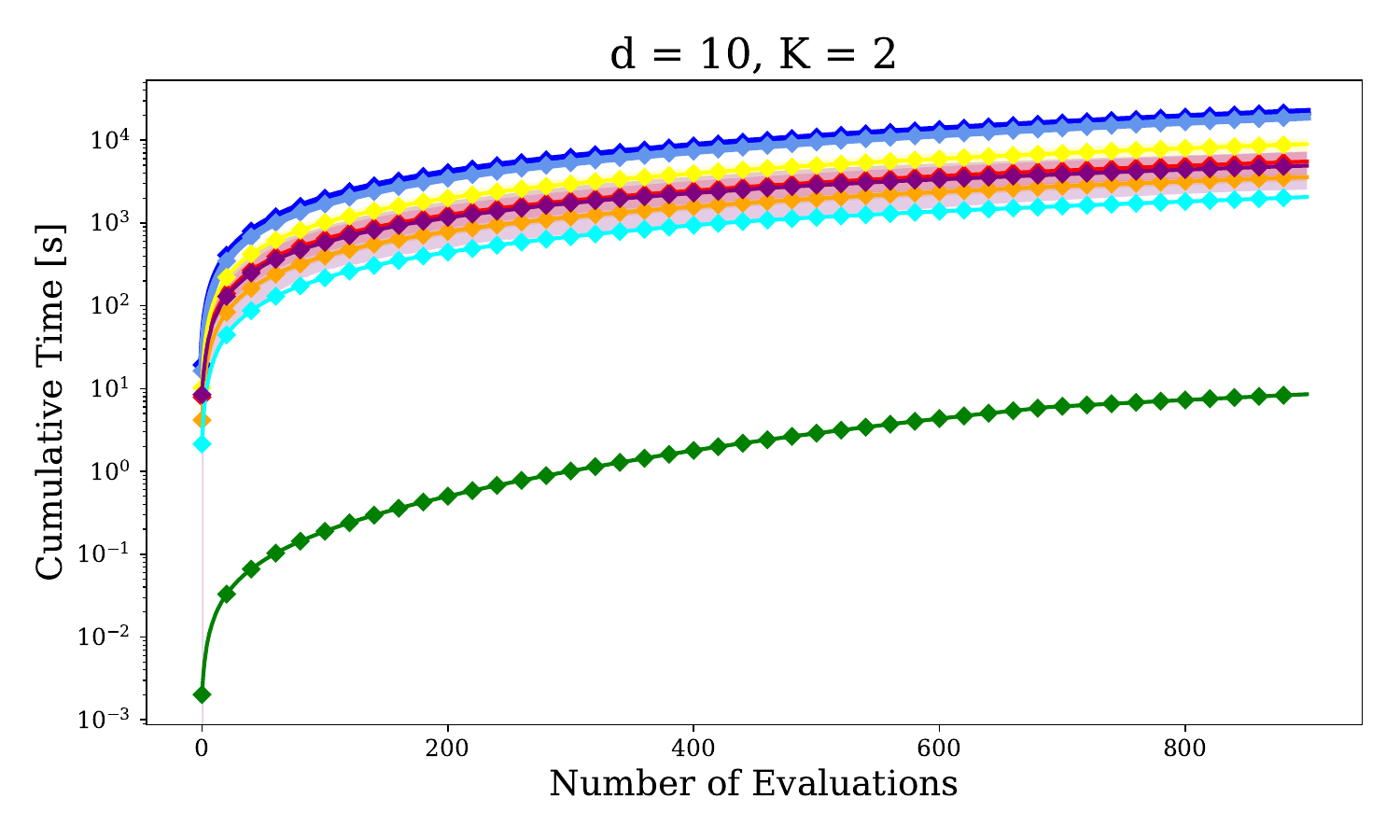}
        \caption{ZDT3}
    \end{subfigure}%
    \begin{subfigure}{0.33\textwidth}
        \includegraphics[width=1\linewidth]{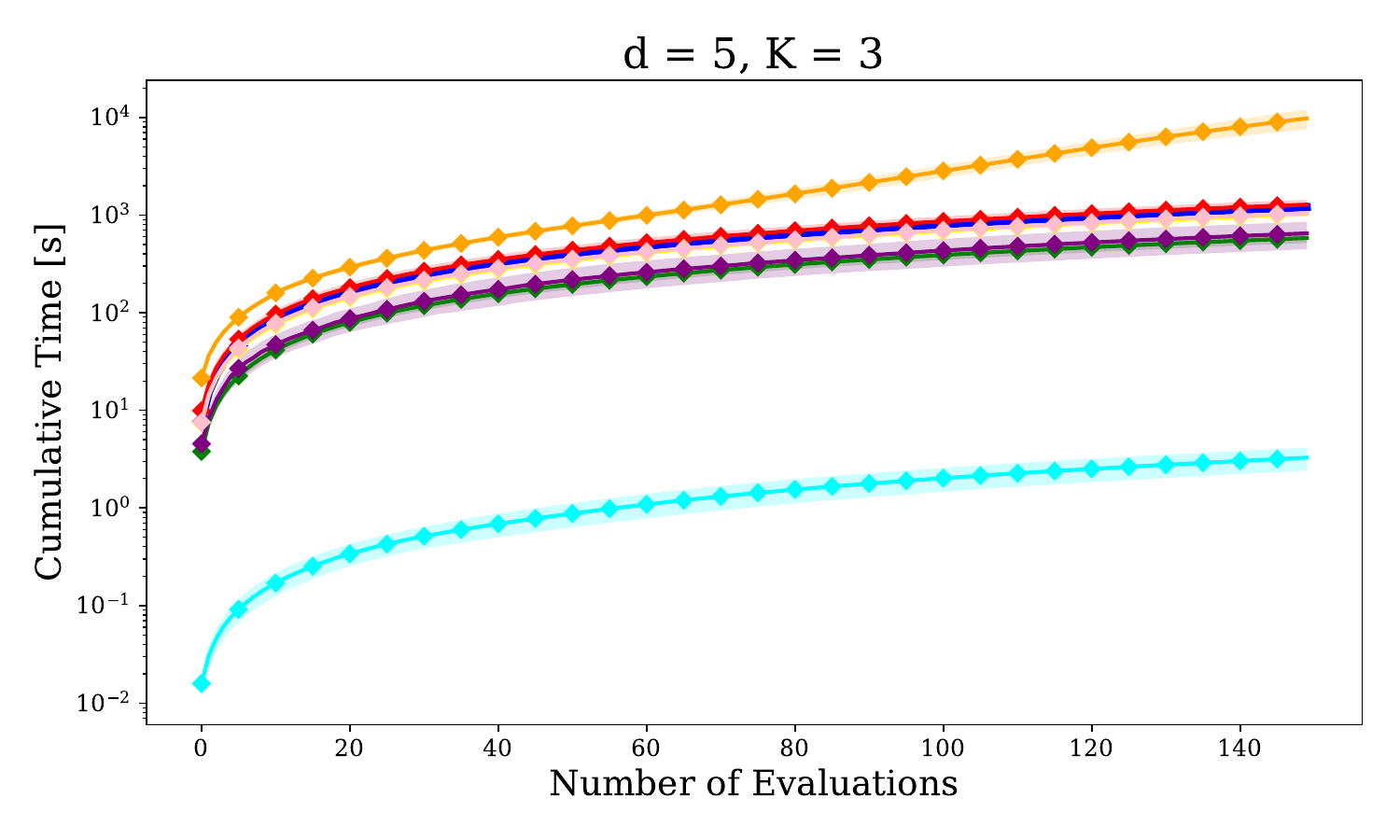}
        \caption{Vehicle design}
    \end{subfigure}%
    \begin{subfigure}{0.33\textwidth}
        \includegraphics[width=1\linewidth]{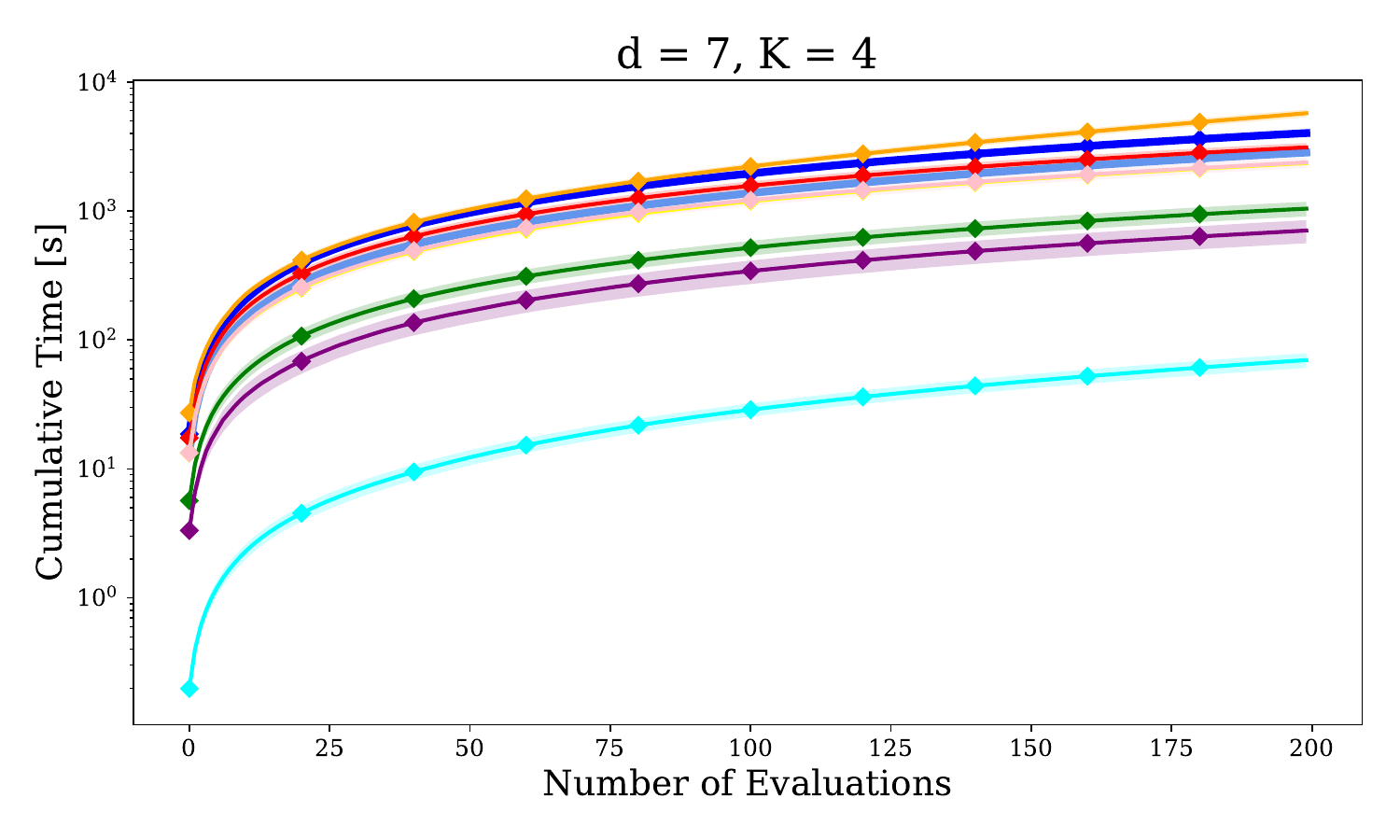}
        \caption{Car side-impact}
    \end{subfigure}\\     
    \begin{subfigure}{0.33\textwidth}
        \includegraphics[width=1\linewidth]{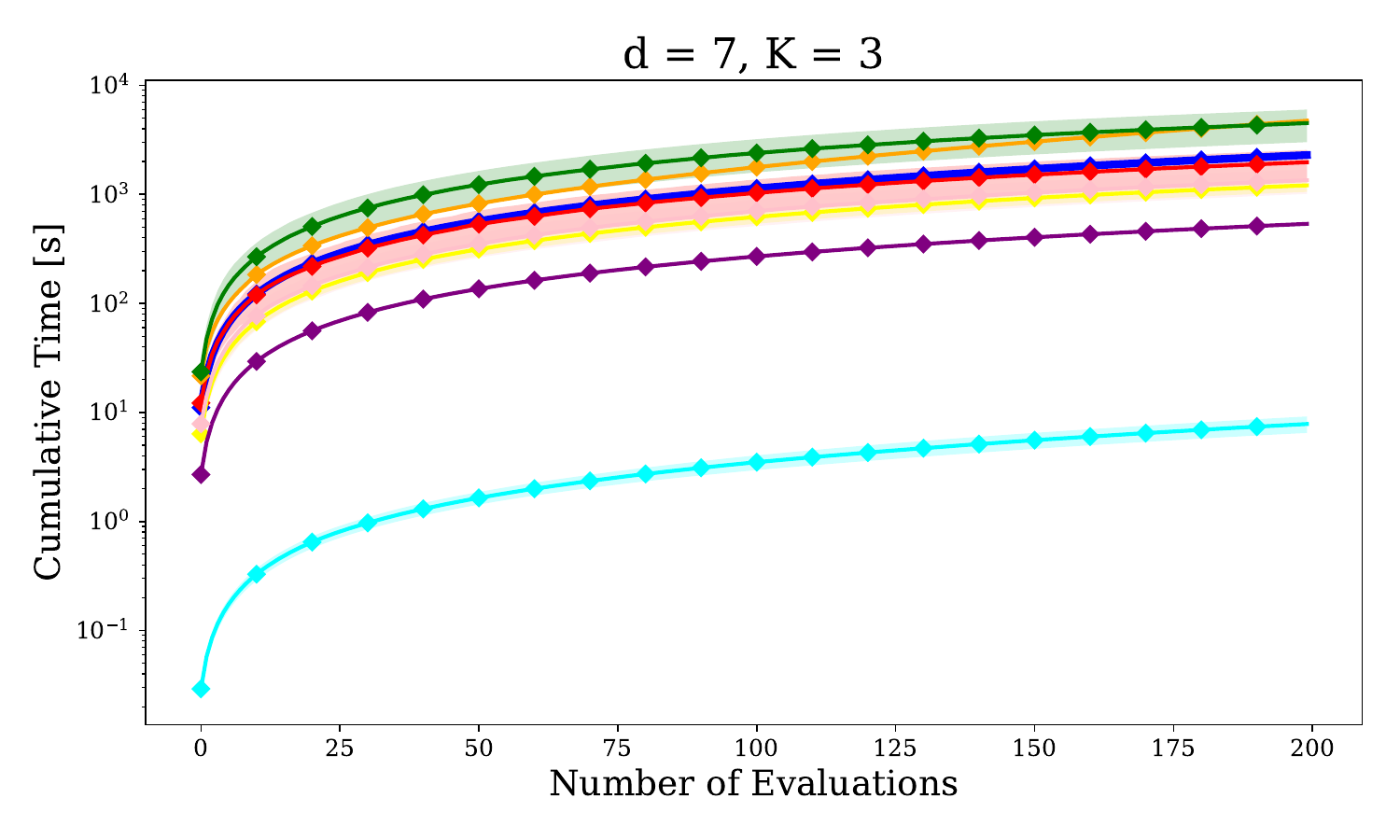}
        \caption{Penicillin production}
    \end{subfigure}%
        \begin{subfigure}{0.33\textwidth}
        \includegraphics[width=1\linewidth]{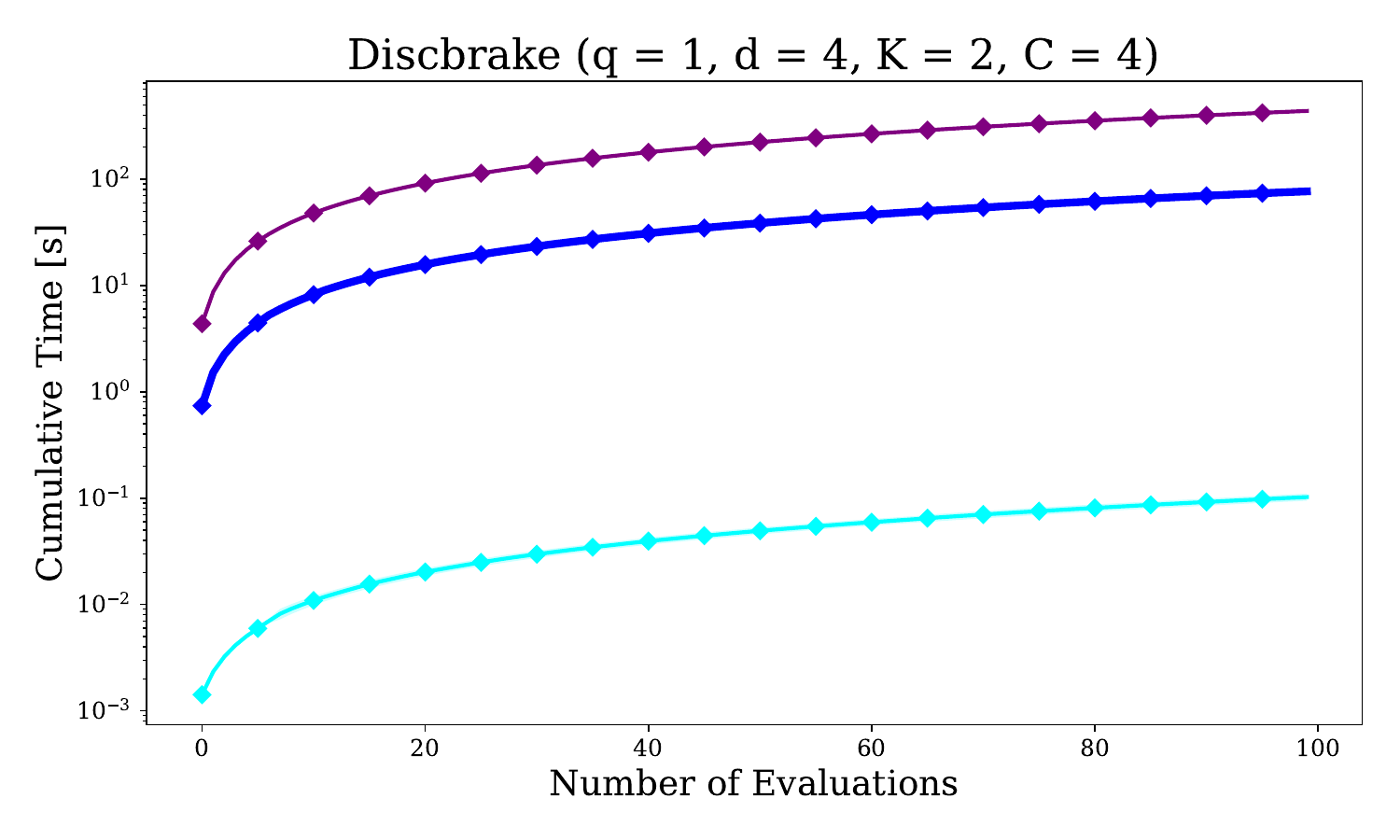}
        \caption{Discbrake}
    \end{subfigure}%
    \begin{subfigure}{0.33\textwidth}
        \includegraphics[width=1\linewidth]{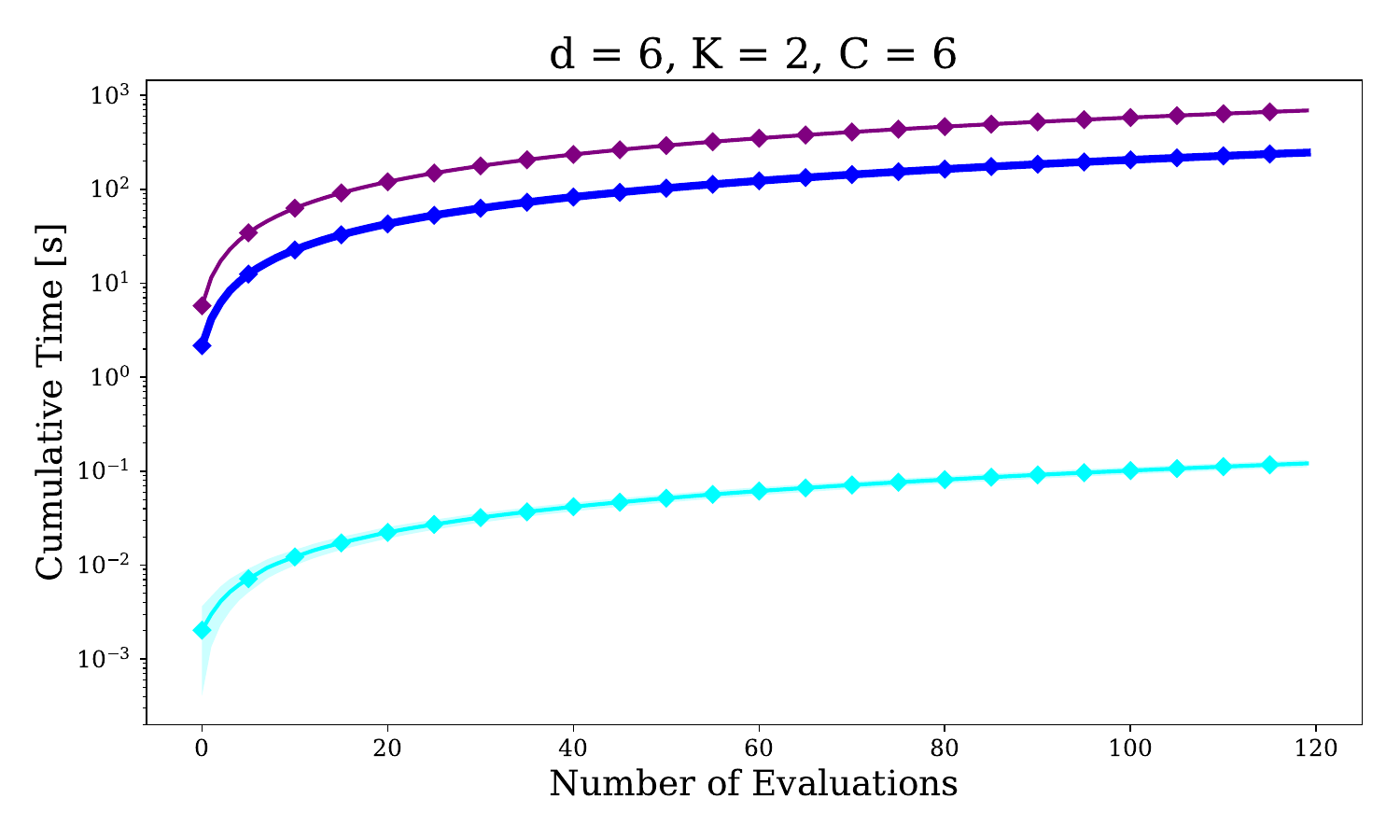}
        \caption{OSY}
    \end{subfigure}\\  
    \begin{subfigure}{1\textwidth}
        \includegraphics[width=1\linewidth]{newest_figs/legend_seq.pdf}        
    \end{subfigure}
    \caption{Cumulative times for sequential acquisition ($q=1$).}
    \label{fig:seq_times}
\end{figure*}

\begin{figure*}[htb!]
    \centering
    \begin{subfigure}{0.33\textwidth}
        \includegraphics[width=1\linewidth]{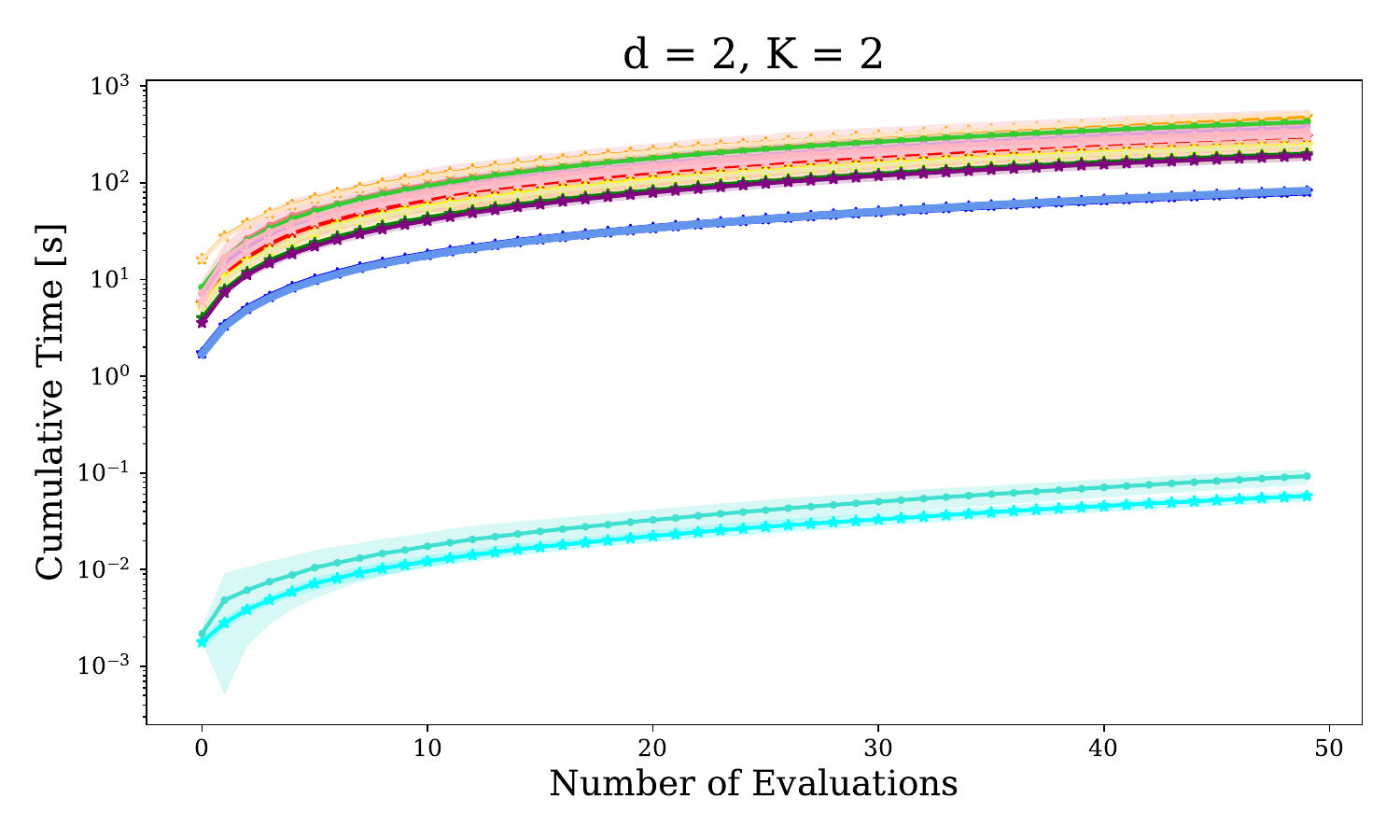}
        \caption{Branin-Currin}
    \end{subfigure}%
    \begin{subfigure}{0.33\textwidth}
        \includegraphics[width=1\linewidth]{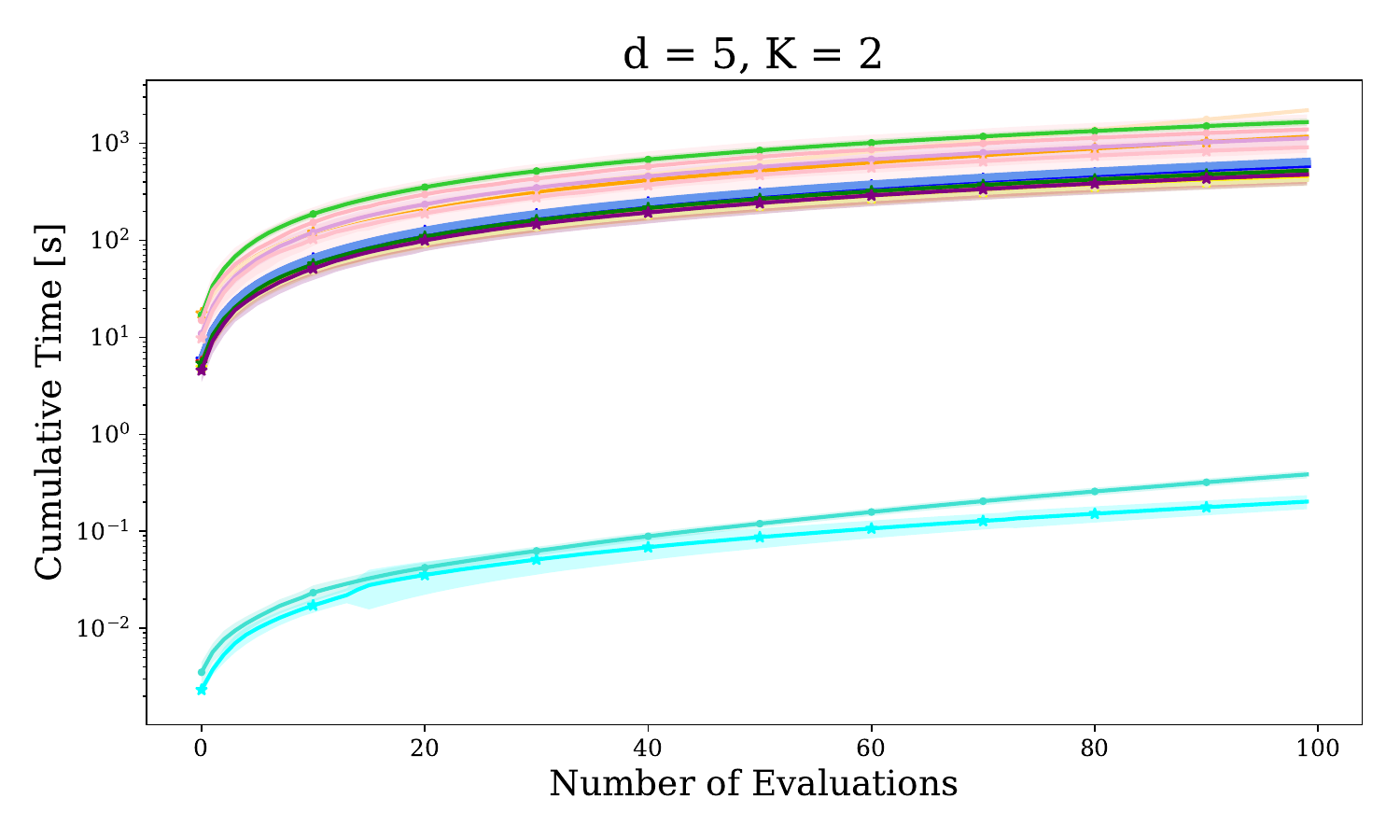}
        \caption{DTLZ7}
    \end{subfigure}%
    \begin{subfigure}{0.33\textwidth}
        \includegraphics[width=1\linewidth]{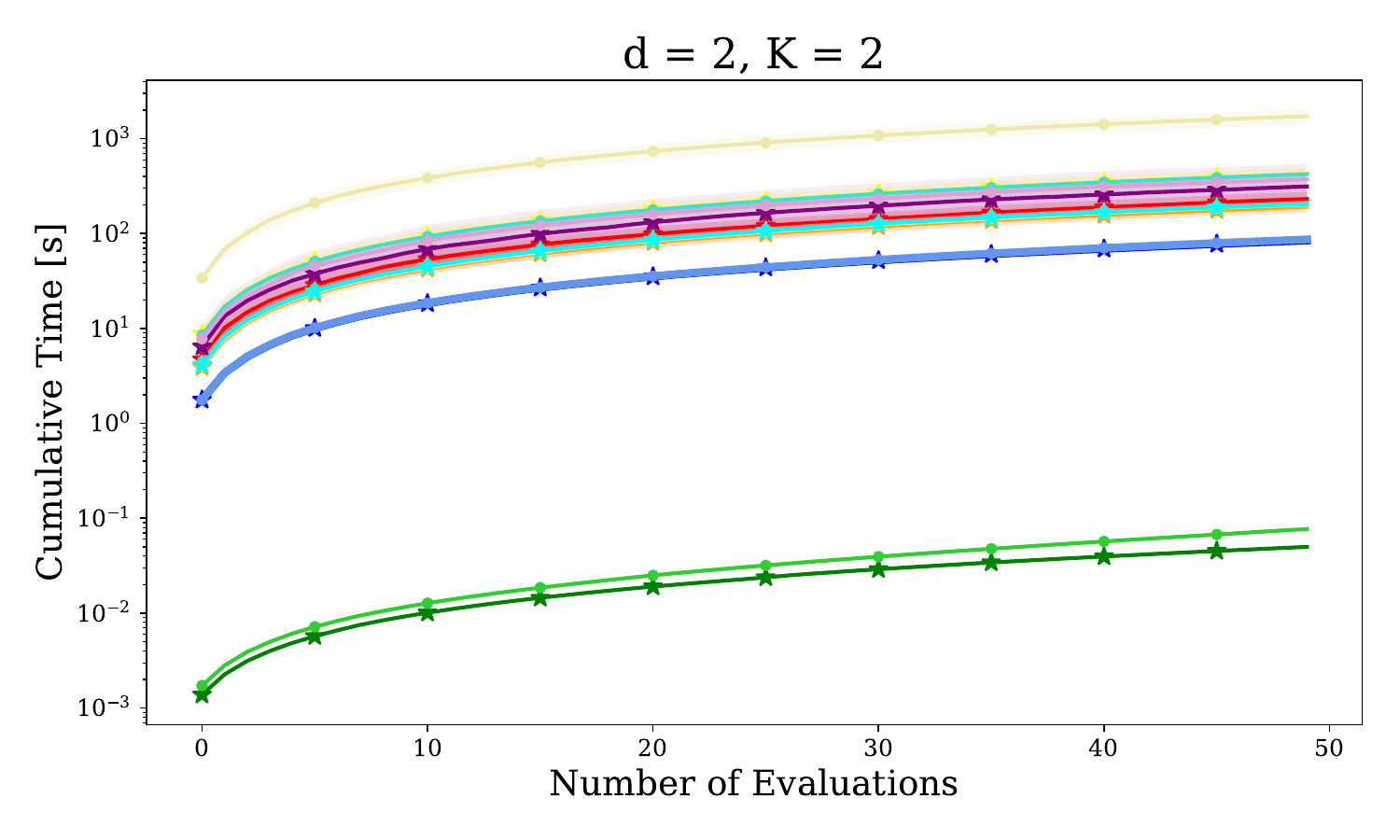}
        \caption{ZDT3}
    \end{subfigure}\\
    \begin{subfigure}{0.33\textwidth}
        \includegraphics[width=1\linewidth]{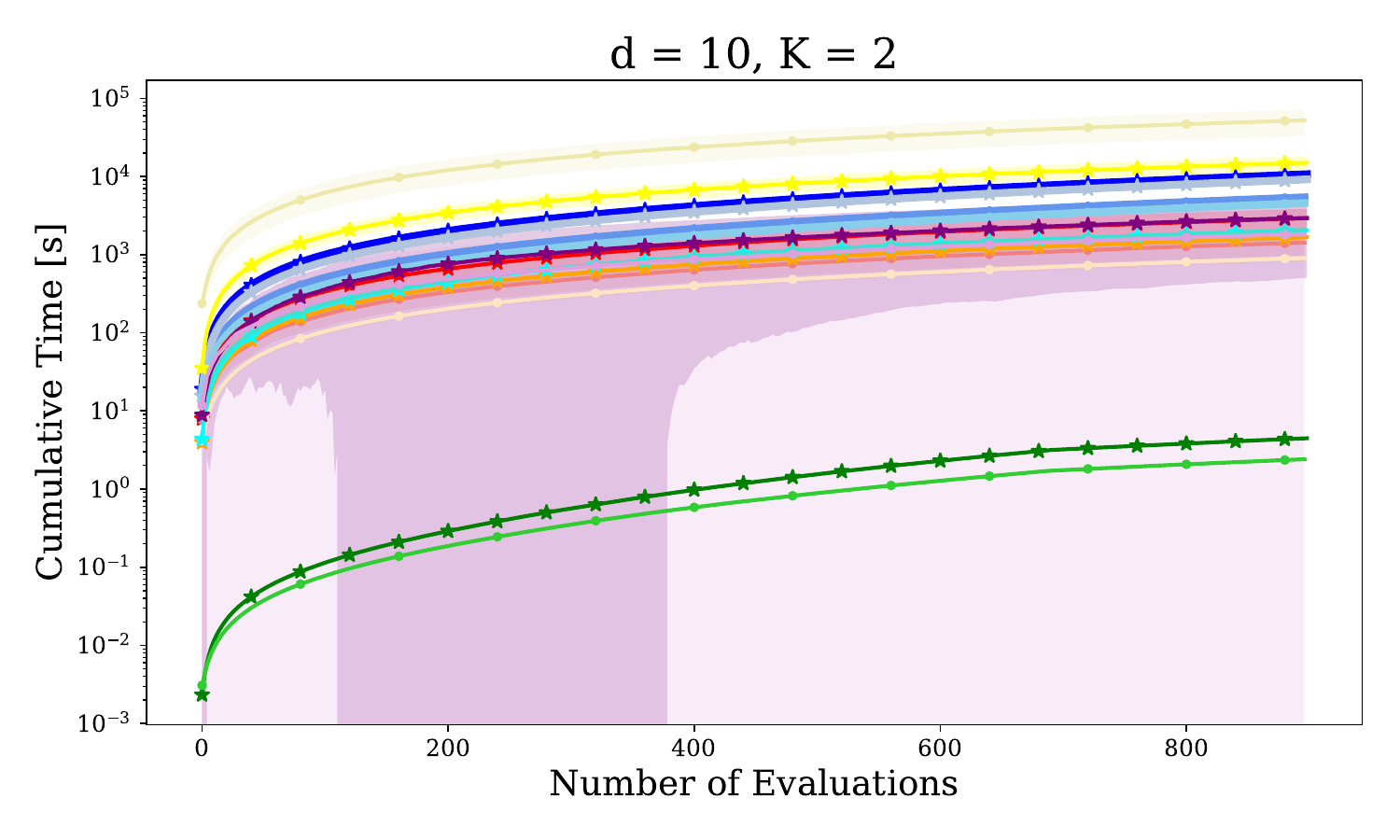}
        \caption{ZDT3}
    \end{subfigure}%
    \begin{subfigure}{0.33\textwidth}
        \includegraphics[width=1\linewidth]{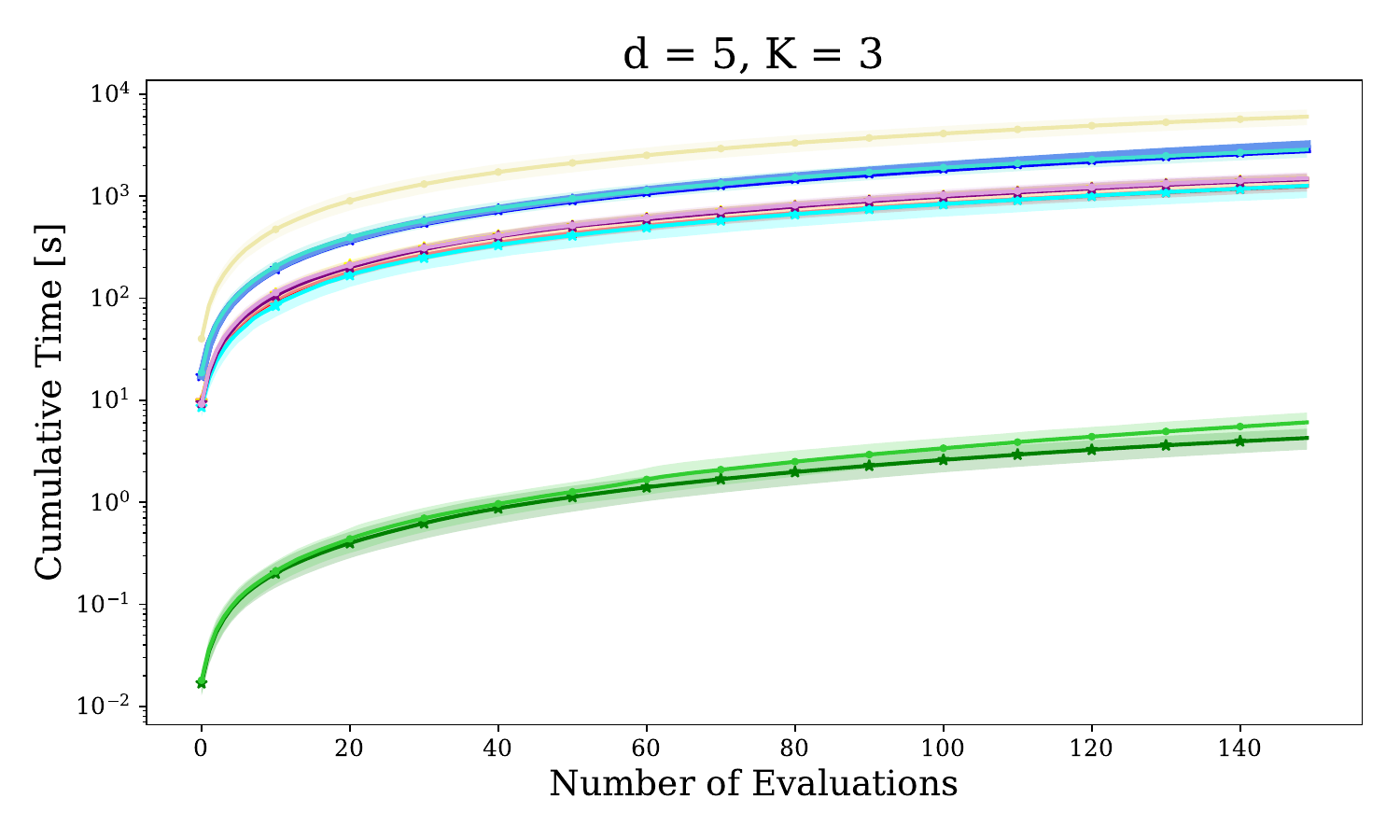}
        \caption{Vehicle design}
    \end{subfigure}%
    \begin{subfigure}{0.33\textwidth}
        \includegraphics[width=1\linewidth]{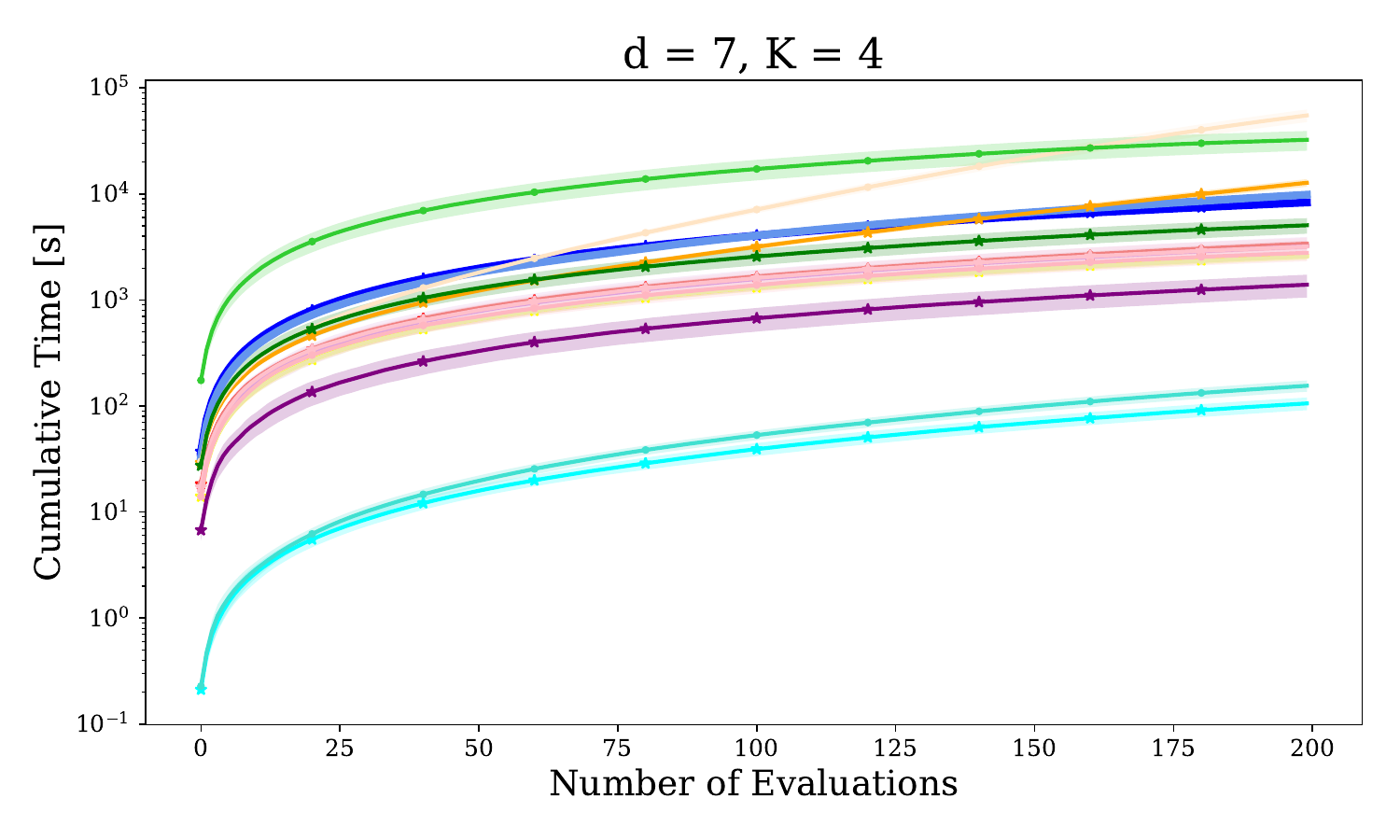}
        \caption{Car side-impact}
    \end{subfigure}\\     
    \begin{subfigure}{0.33\textwidth}
        \includegraphics[width=1\linewidth]{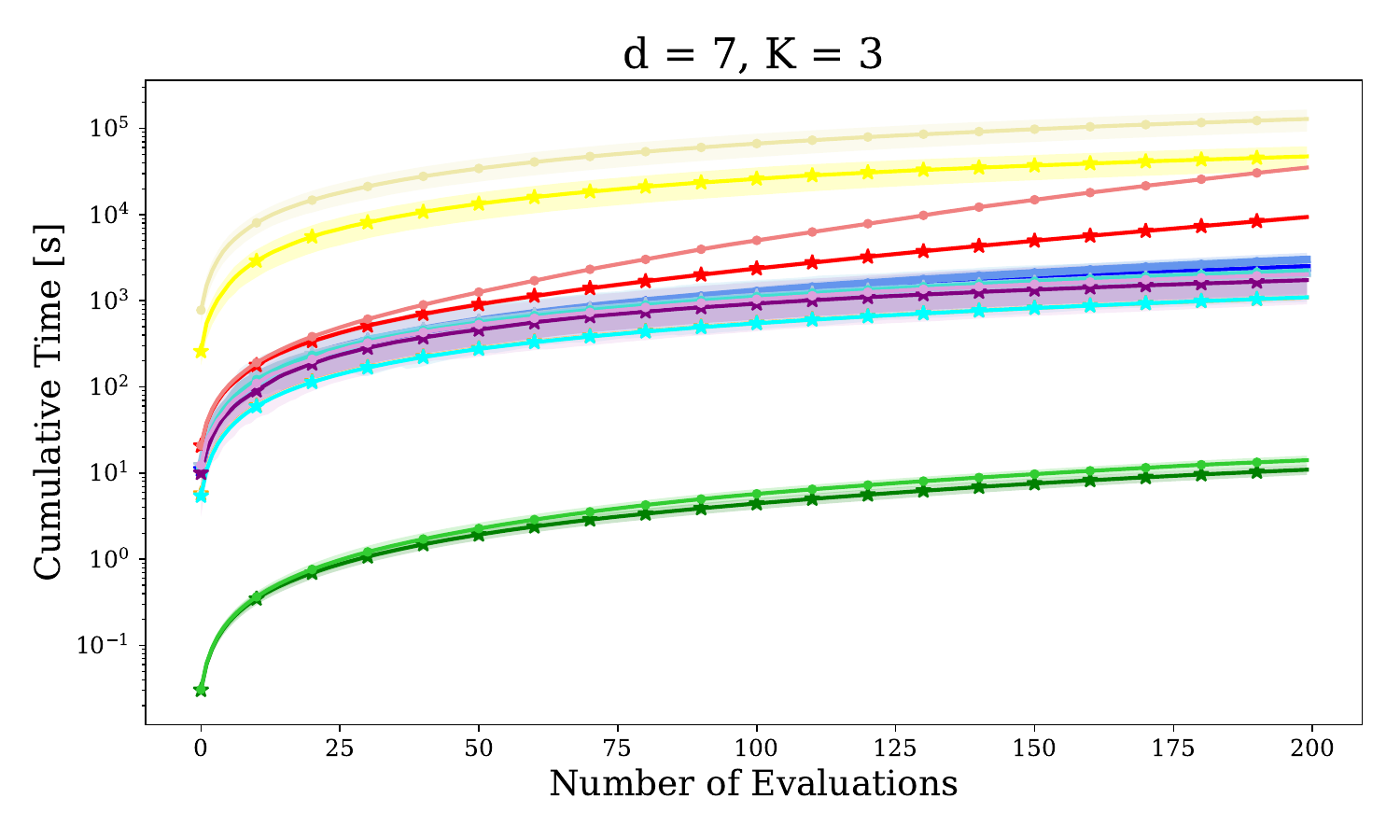}
        \caption{Penicillin production}
    \end{subfigure}%
        \begin{subfigure}{0.33\textwidth}
        \includegraphics[width=1\linewidth]{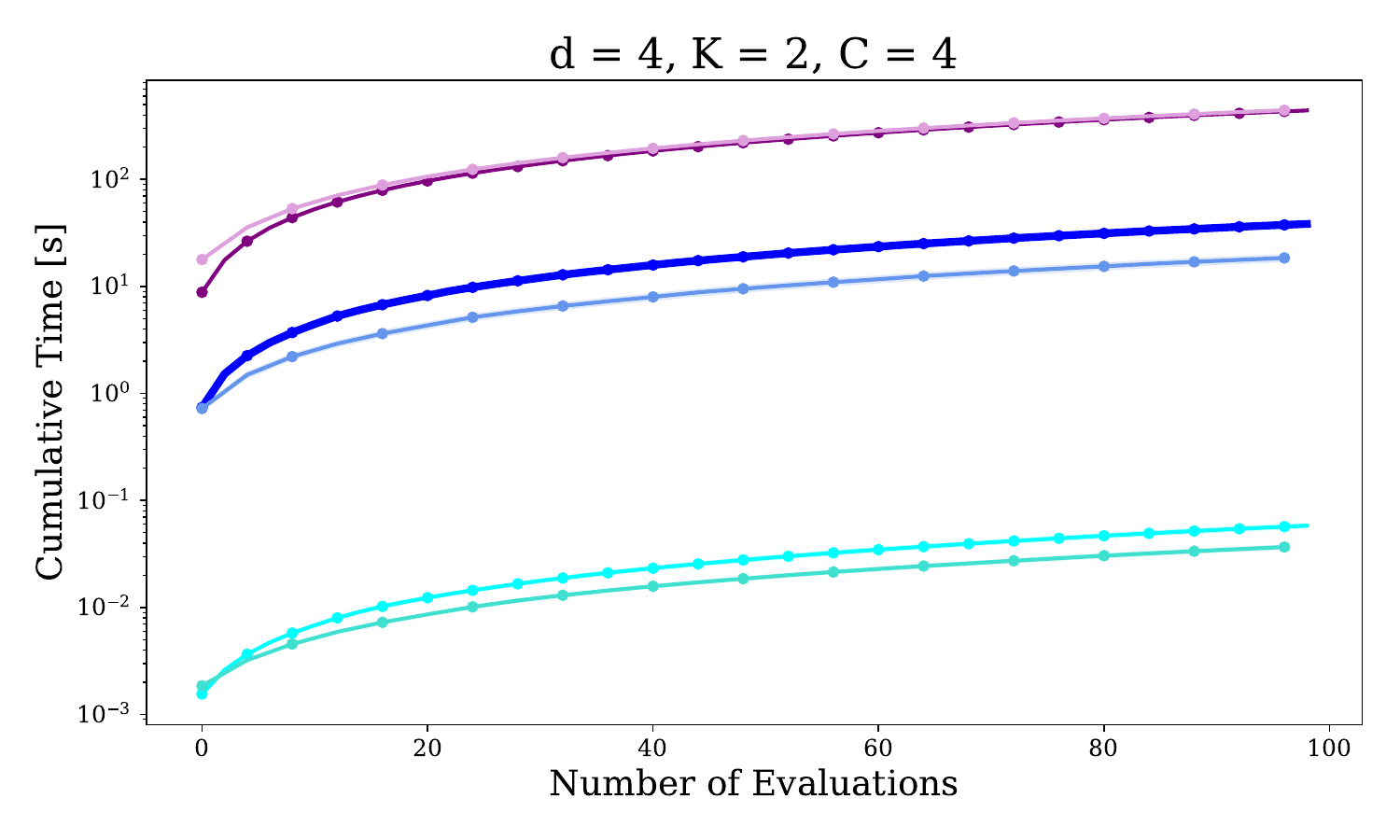}
        \caption{Discbrake}
    \end{subfigure}%
    \begin{subfigure}{0.33\textwidth}
        \includegraphics[width=1\linewidth]{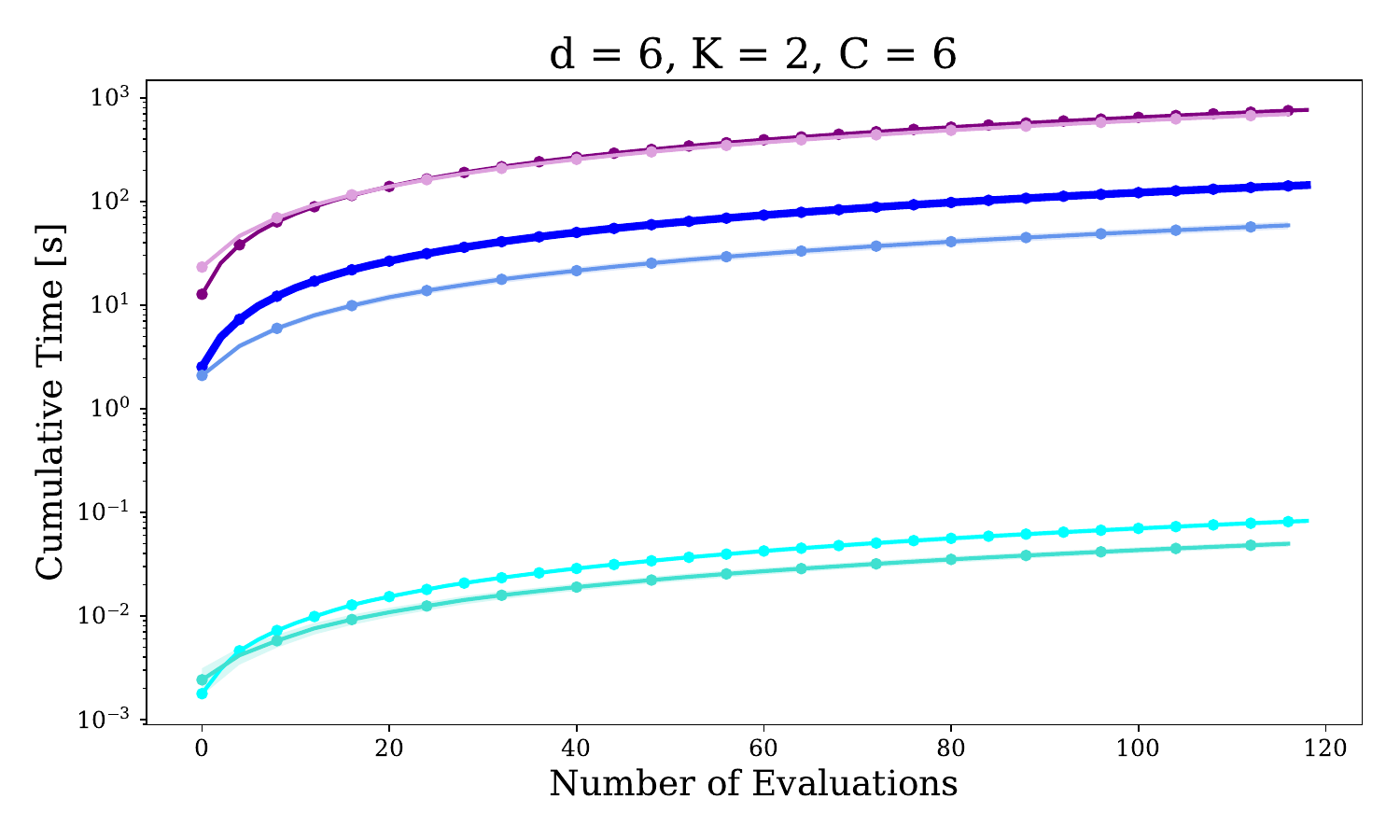}
        \caption{OSY}
    \end{subfigure}\\  
    \begin{subfigure}{0.33\textwidth}
        \includegraphics[width=1\linewidth]{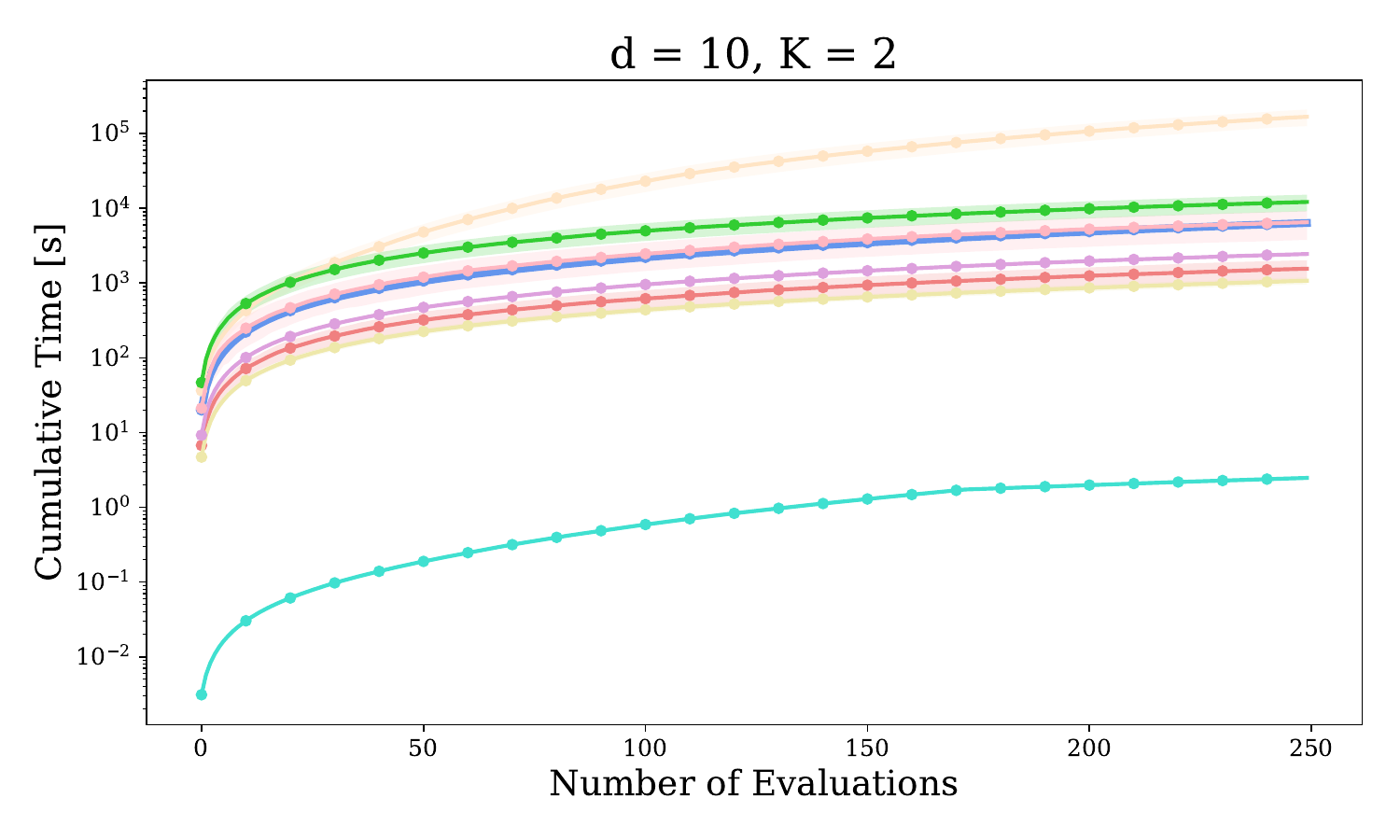}
        \caption{DH3}
    \end{subfigure}%
    \begin{subfigure}{0.33\textwidth}
        \includegraphics[width=1\linewidth]{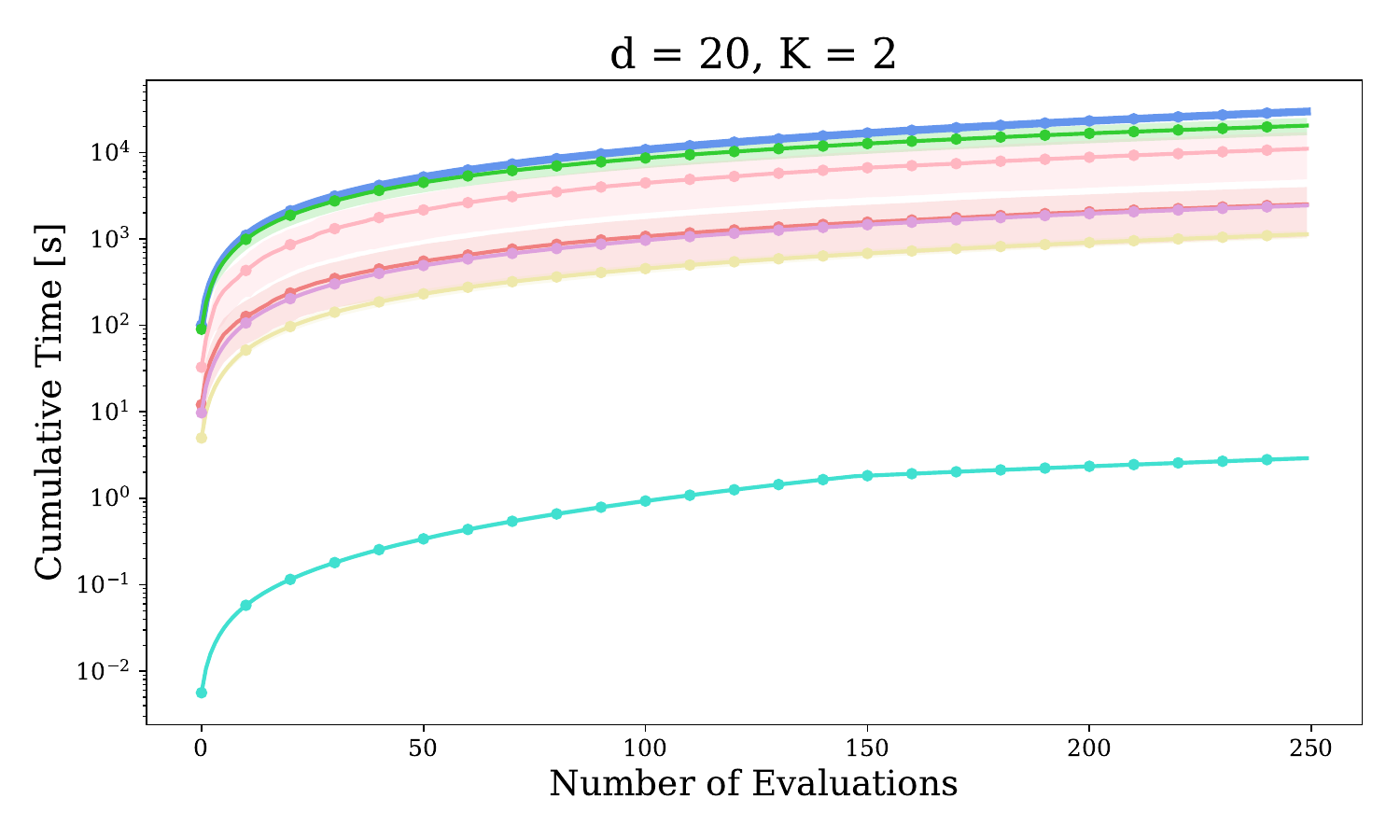}
        \caption{DH3}
    \end{subfigure}\\
    \begin{subfigure}{1\textwidth}
        \includegraphics[width=1\linewidth]{newest_figs/legend_batch.pdf}
    \end{subfigure}
    \caption{Cumulative times for batch acquisition ($q>1$).}
    \label{fig:batch-times}
\end{figure*}

Finally, we show the computational times for all algorithms, including \qpots, each of the experiments. \Cref{fig:seq_times} and \Cref{fig:batch-times} show them for sequential and batch acquisitions, respectively. First, notice that \qpots~is much cheaper than the rest, with the exception of Sobol, for low dimensional problems. In higher dimensions, and with more objectives and constraints, \qpots~is no worse than the rest. The benefit is even more prominent in the batch acquisition setting where the cost of \qpots~remains almost the same as the sequential setting but other acquisition policies become more expensive.

\section{Choice of NSGA-II parameters}
Assumption \ref{ass:ass1} necessitates exact knowledge of $X^*$. We demonstrate that this is reasonable provided that the number of generations and the population size in NSGA-II are sufficiently large. \Cref{fig:ngen} and \Cref{fig:npop} show the impact of systematically increasing both parameters on a constrained bi-objective problem~\footnote{\url{https://pymoo.org/getting_started/part_1.html}} where the true Pareto frontier is analytically known. Notice that as both parameters are increased, the true Pareto frontier is resolved.

\begin{figure*}[htb!]
    \centering
    \begin{subfigure}{0.5\textwidth}
        \includegraphics[width=1\linewidth]{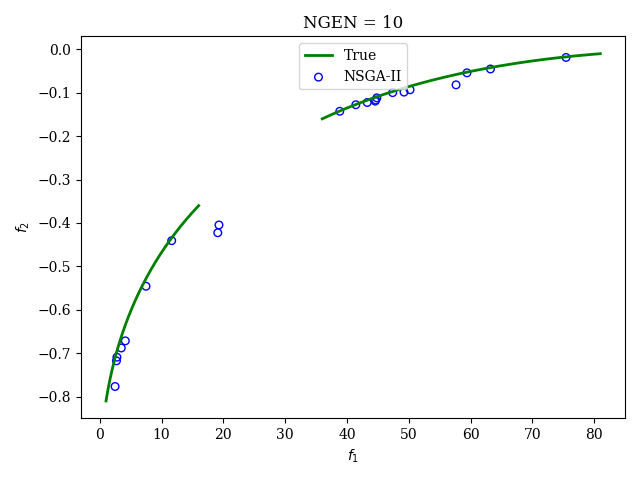}
        \caption{ngen=10}
    \end{subfigure}%
    \begin{subfigure}{0.5\textwidth}
        \includegraphics[width=1\linewidth]{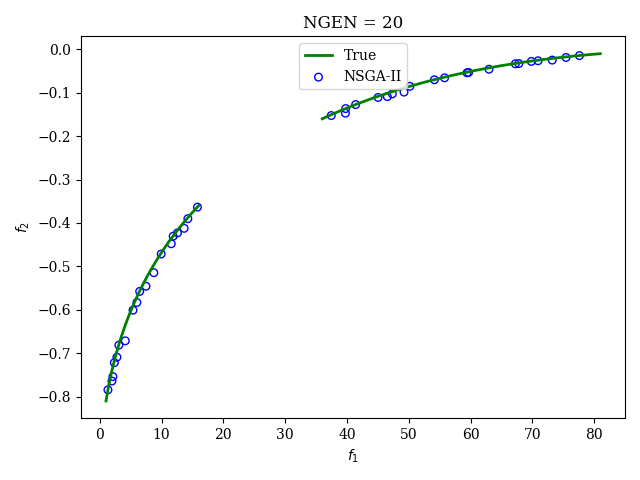}
        \caption{ngen=20}
    \end{subfigure}\\
    \begin{subfigure}{0.5\textwidth}
        \includegraphics[width=1\linewidth]{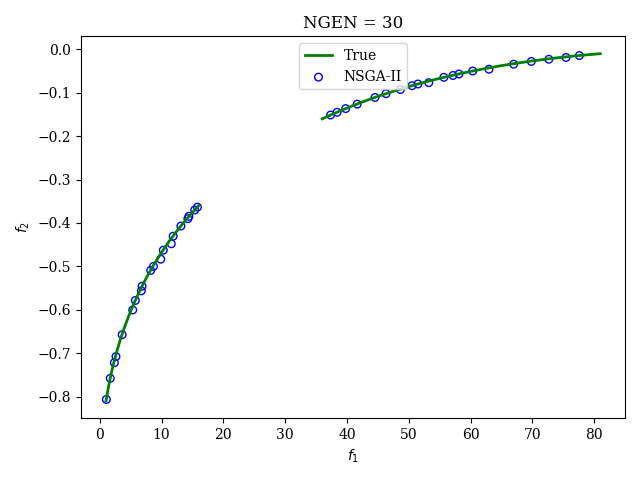}
        \caption{ngen=30}
    \end{subfigure}%
    \begin{subfigure}{0.5\textwidth}
        \includegraphics[width=1\linewidth]{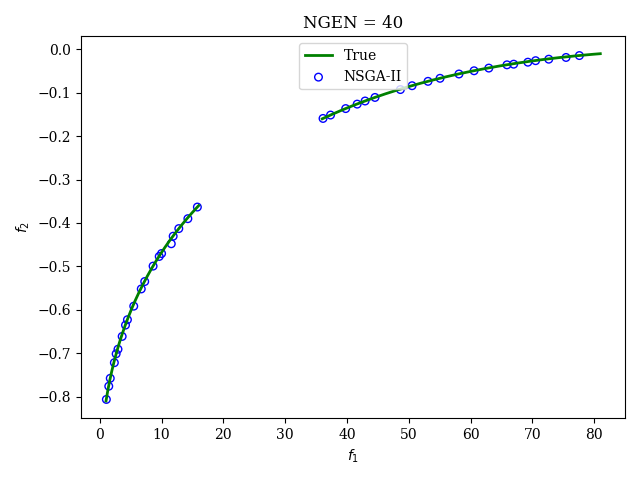}
        \caption{ngen=40}
    \end{subfigure}\\
    \begin{subfigure}{0.5\textwidth}
        \includegraphics[width=1\linewidth]{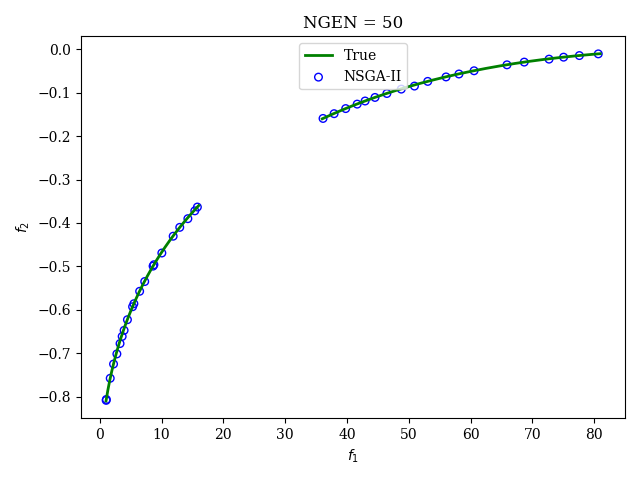}
        \caption{ngen=50}
    \end{subfigure}%
    \begin{subfigure}{0.5\textwidth}
        \includegraphics[width=1\linewidth]{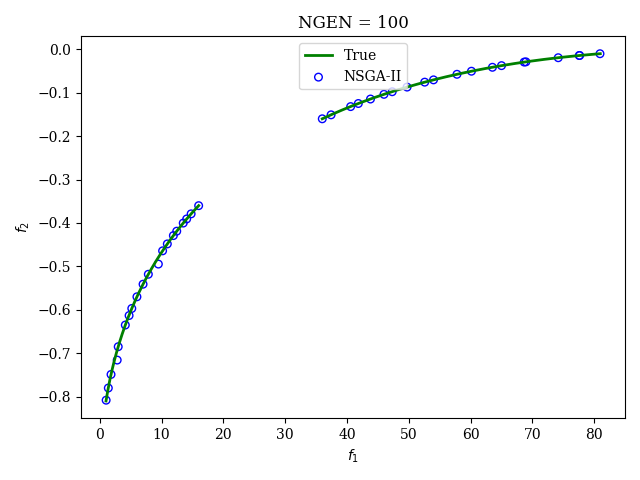}
        \caption{ngen=100}
    \end{subfigure}
    \caption{Impact of increasing number of generations at fixed population size for NSGA-II.}
    \label{fig:ngen}
\end{figure*}

\begin{figure*}[htb!]
    \centering
    \begin{subfigure}{0.5\textwidth}
        \includegraphics[width=1\linewidth]{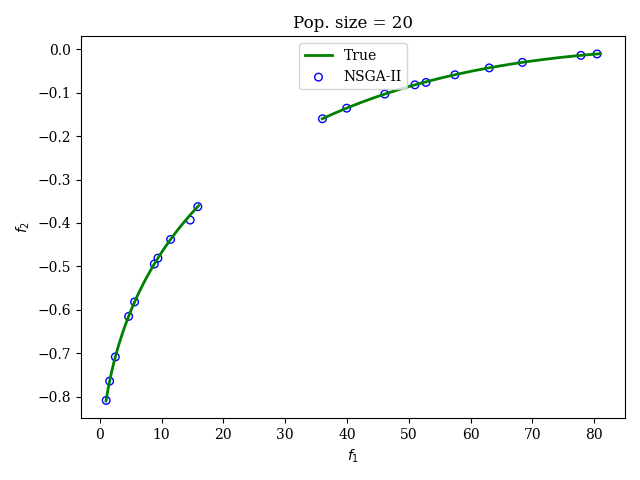}
        \caption{npop=20}
    \end{subfigure}%
    \begin{subfigure}{0.5\textwidth}
        \includegraphics[width=1\linewidth]{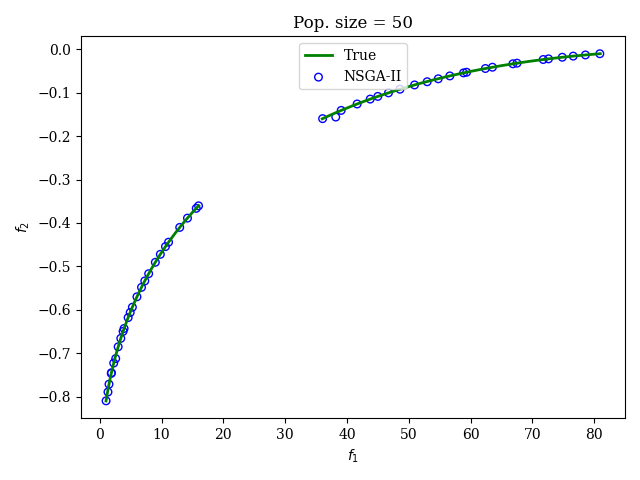}
        \caption{npop=50}
    \end{subfigure}\\
    \begin{subfigure}{0.5\textwidth}
        \includegraphics[width=1\linewidth]{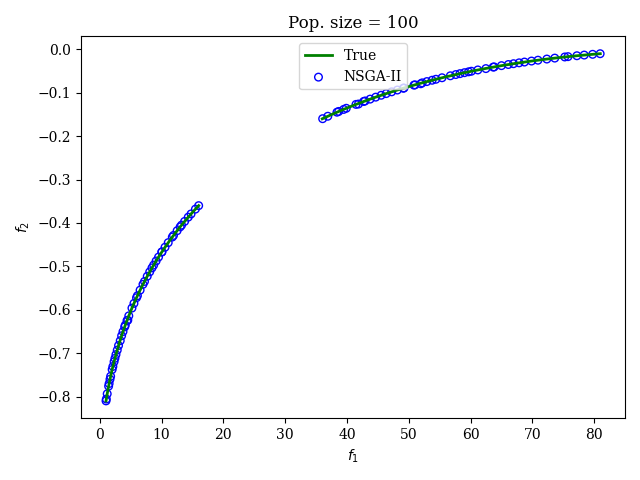}
        \caption{npop=100}
    \end{subfigure}%
    \begin{subfigure}{0.5\textwidth}
        \includegraphics[width=1\linewidth]{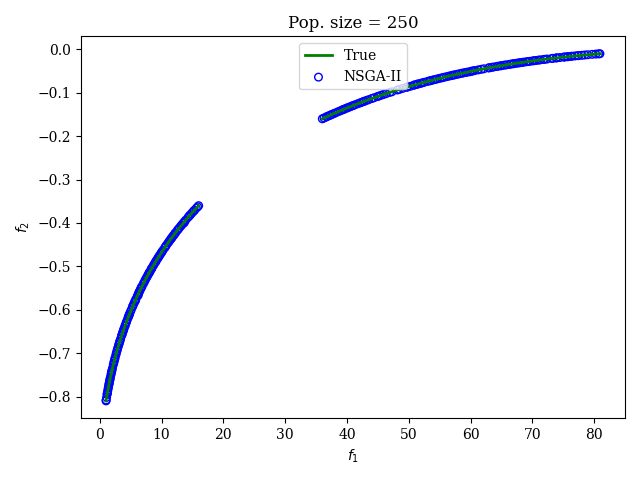}
        \caption{npop=250}
    \end{subfigure}\\
    \begin{subfigure}{0.5\textwidth}
        \includegraphics[width=1\linewidth]{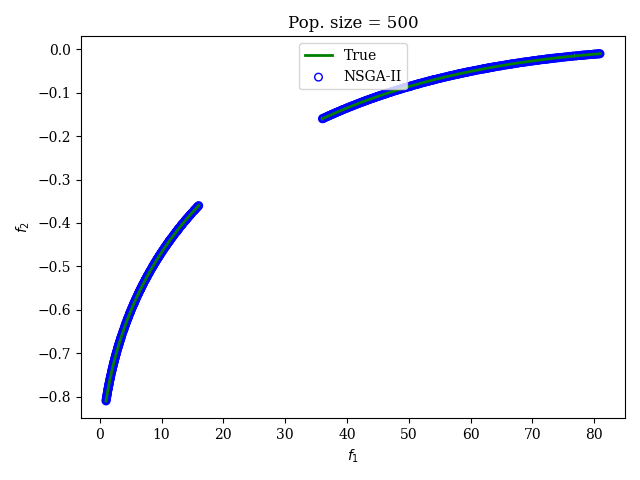}
        \caption{npop=500}
    \end{subfigure}%
    \caption{Impact of increasing population size at fixed number of generations for NSGA-II.}
    \label{fig:npop}
\end{figure*}
\clearpage
 \bibliographystyle{plainnat}
\bibliography{references}

\begin{thebibliography}{59}
\providecommand{\natexlab}[1]{#1}
\providecommand{\url}[1]{\texttt{#1}}
\expandafter\ifx\csname urlstyle\endcsname\relax
  \providecommand{\doi}[1]{doi: #1}\else
  \providecommand{\doi}{doi: \begingroup \urlstyle{rm}\Url}\fi

\bibitem[Balandat et~al.(2019)Balandat, Karrer, Jiang, Daulton, Letham, Wilson,
  and Bakshy]{balandat2019botorch}
Maximilian Balandat, Brian Karrer, Daniel~R Jiang, Samuel Daulton, Benjamin
  Letham, Andrew~Gordon Wilson, and Eytan Bakshy.
\newblock {BoTorch: P}rogrammable {Bayesian} optimization in {PyTorch}.
\newblock \emph{arXiv:1910.06403}, 2019.

\bibitem[Belakaria et~al.(2019)Belakaria, Deshwal, and
  Doppa]{belakariaMaxvalueEntropySearch2019}
Syrine Belakaria, Aryan Deshwal, and Janardhan~Rao Doppa.
\newblock Max-value {{Entropy Search}} for {{Multi-Objective Bayesian
  Optimization}}.
\newblock In \emph{Advances in {{Neural Information Processing Systems}}},
  volume~32. {Curran Associates, Inc.}, 2019.

\bibitem[Blank and Deb(2020)]{blank2020pymoo}
Julian Blank and Kalyanmoy Deb.
\newblock Pymoo: Multi-objective optimization in python.
\newblock \emph{Ieee access}, 8:\penalty0 89497--89509, 2020.

\bibitem[Bonilla et~al.(2007)Bonilla, Chai, and Williams]{bonilla2007multi}
Edwin~V Bonilla, Kian Chai, and Christopher Williams.
\newblock Multi-task gaussian process prediction.
\newblock \emph{Advances in neural information processing systems}, 20, 2007.

\bibitem[Booth et~al.(2023)Booth, Renganathan, and Gramacy]{booth2023contour}
Annie~S Booth, S~Ashwin Renganathan, and Robert~B Gramacy.
\newblock Contour location for reliability in airfoil simulation experiments
  using deep gaussian processes.
\newblock \emph{arXiv preprint arXiv:2308.04420}, 2023.

\bibitem[Bradford et~al.(2018)Bradford, Schweidtmann, and
  Lapkin]{bradford2018efficient}
Eric Bradford, Artur~M Schweidtmann, and Alexei Lapkin.
\newblock Efficient multiobjective optimization employing gaussian processes,
  spectral sampling and a genetic algorithm.
\newblock \emph{Journal of global optimization}, 71\penalty0 (2):\penalty0
  407--438, 2018.

\bibitem[Chevalier et~al.(2014)Chevalier, Bect, Ginsbourger, Vazquez, Picheny,
  and Richet]{chevalier2014fast}
Cl{\'e}ment Chevalier, Julien Bect, David Ginsbourger, Emmanuel Vazquez, Victor
  Picheny, and Yann Richet.
\newblock Fast parallel kriging-based stepwise uncertainty reduction with
  application to the identification of an excursion set.
\newblock \emph{Technometrics}, 56\penalty0 (4):\penalty0 455--465, 2014.

\bibitem[Couckuyt et~al.(2014)Couckuyt, Deschrijver, and
  Dhaene]{couckuytFastCalculationMultiobjective2014}
Ivo Couckuyt, Dirk Deschrijver, and Tom Dhaene.
\newblock Fast calculation of multiobjective probability of improvement and
  expected improvement criteria for {{Pareto}} optimization.
\newblock \emph{Journal of Global Optimization}, 60\penalty0 (3):\penalty0
  575--594, November 2014.
\newblock ISSN 1573-2916.
\newblock \doi{10.1007/s10898-013-0118-2}.

\bibitem[Dalcin and Fang(2021)]{dalcin2021mpi4py}
Lisandro Dalcin and Yao-Lung~L Fang.
\newblock mpi4py: Status update after 12 years of development.
\newblock \emph{Computing in Science \& Engineering}, 23\penalty0 (4):\penalty0
  47--54, 2021.

\bibitem[Damianou and Lawrence(2013)]{damianou2013deep}
Andreas Damianou and Neil~D Lawrence.
\newblock Deep gaussian processes.
\newblock In \emph{Artificial intelligence and statistics}, pages 207--215.
  PMLR, 2013.

\bibitem[Daulton et~al.(2020{\natexlab{a}})Daulton, Balandat, and
  Bakshy]{daulton2020differentiable}
Samuel Daulton, Maximilian Balandat, and Eytan Bakshy.
\newblock Differentiable expected hypervolume improvement for parallel
  multi-objective bayesian optimization.
\newblock \emph{Advances in Neural Information Processing Systems},
  33:\penalty0 9851--9864, 2020{\natexlab{a}}.

\bibitem[Daulton et~al.(2020{\natexlab{b}})Daulton, Balandat, and
  Bakshy]{daultonDifferentiableExpectedHypervolume2020}
Samuel Daulton, Maximilian Balandat, and Eytan Bakshy.
\newblock Differentiable {{Expected Hypervolume Improvement}} for {{Parallel
  Multi-Objective Bayesian Optimization}}.
\newblock In \emph{Advances in {{Neural Information Processing Systems}}},
  volume~33, pages 9851--9864. {Curran Associates, Inc.}, 2020{\natexlab{b}}.

\bibitem[Daulton et~al.(2021)Daulton, Balandat, and
  Bakshy]{daultonParallelBayesianOptimization2021}
Samuel Daulton, Maximilian Balandat, and Eytan Bakshy.
\newblock Parallel {{Bayesian Optimization}} of {{Multiple Noisy Objectives}}
  with {{Expected Hypervolume Improvement}}.
\newblock In \emph{Advances in {{Neural Information Processing Systems}}},
  volume~34, pages 2187--2200. {Curran Associates, Inc.}, 2021.

\bibitem[Daulton et~al.(2022)Daulton, Eriksson, Balandat, and
  Bakshy]{daulton2022multi}
Samuel Daulton, David Eriksson, Maximilian Balandat, and Eytan Bakshy.
\newblock Multi-objective bayesian optimization over high-dimensional search
  spaces.
\newblock In \emph{Uncertainty in Artificial Intelligence}, pages 507--517.
  PMLR, 2022.

\bibitem[Deb et~al.(2000)Deb, Agrawal, Pratap, and Meyarivan]{deb2000fast}
Kalyanmoy Deb, Samir Agrawal, Amrit Pratap, and Tanaka Meyarivan.
\newblock A fast elitist non-dominated sorting genetic algorithm for
  multi-objective optimization: Nsga-ii.
\newblock In \emph{Parallel Problem Solving from Nature PPSN VI: 6th
  International Conference Paris, France, September 18--20, 2000 Proceedings
  6}, pages 849--858. Springer, 2000.

\bibitem[Deb et~al.(2002)Deb, Pratap, Agarwal, and Meyarivan]{deb2002fast}
Kalyanmoy Deb, Amrit Pratap, Sameer Agarwal, and TAMT Meyarivan.
\newblock A fast and elitist multiobjective genetic algorithm: Nsga-ii.
\newblock \emph{IEEE transactions on evolutionary computation}, 6\penalty0
  (2):\penalty0 182--197, 2002.

\bibitem[Emmerich(2008)]{emmerichComputationExpectedImprovement2008}
Michael Emmerich.
\newblock The computation of the expected improvement in dominated hypervolume
  of {{Pareto}} front approximations.
\newblock January 2008.

\bibitem[Frazier et~al.(2008)Frazier, Powell, and
  Dayanik]{frazier2008knowledge}
Peter~I Frazier, Warren~B Powell, and Savas Dayanik.
\newblock A knowledge-gradient policy for sequential information collection.
\newblock \emph{SIAM Journal on Control and Optimization}, 47\penalty0
  (5):\penalty0 2410--2439, 2008.
\newblock \doi{10.1137/070693424}.

\bibitem[Gardner et~al.(2018)Gardner, Pleiss, Weinberger, Bindel, and
  Wilson]{gardner2018gpytorch}
Jacob Gardner, Geoff Pleiss, Kilian~Q Weinberger, David Bindel, and Andrew~G
  Wilson.
\newblock {GPyTorch: B}lackbox matrix-matrix {Gaussian} process inference with
  {GPU} acceleration.
\newblock In \emph{Advances in Neural Information Processing Systems}, pages
  7576--7586, 2018.

\bibitem[Ghosal and Roy(2006)]{ghosal2006posterior}
Subhashis Ghosal and Anindya Roy.
\newblock Posterior consistency of {G}aussian process prior for nonparametric
  binary regression.
\newblock \emph{The Annals of Statistics}, 34\penalty0 (5):\penalty0
  2413--2429, 2006.

\bibitem[{Hern{\'a}ndez-Lobato} et~al.(){Hern{\'a}ndez-Lobato},
  {Hern{\'a}ndez-Lobato}, Shah, and
  Adams]{hernandez-lobatoPredictiveEntropySearch}
Daniel {Hern{\'a}ndez-Lobato}, Jos{\'e}~Miguel {Hern{\'a}ndez-Lobato}, Amar
  Shah, and Ryan~P Adams.
\newblock Predictive {{Entropy Search}} for {{Multi-objective Bayesian
  Optimization}}.

\bibitem[Hern{\'a}ndez-Lobato et~al.(2014)Hern{\'a}ndez-Lobato, Hoffman, and
  Ghahramani]{hernandez2014predictive}
Jos{\'e}~Miguel Hern{\'a}ndez-Lobato, Matthew~W. Hoffman, and Zoubin
  Ghahramani.
\newblock Predictive entropy search for efficient global optimization of
  black-box functions.
\newblock In \emph{Advances in Neural Information Processing Systems}, pages
  918--926, 2014.

\bibitem[Hvarfner et~al.(2022)Hvarfner, Hutter, and Nardi]{hvarfner2022joint}
Carl Hvarfner, Frank Hutter, and Luigi Nardi.
\newblock Joint entropy search for maximally-informed bayesian optimization.
\newblock \emph{Advances in Neural Information Processing Systems},
  35:\penalty0 11494--11506, 2022.

\bibitem[Jain and Deb(2013)]{jain2013evolutionary}
Himanshu Jain and Kalyanmoy Deb.
\newblock An evolutionary many-objective optimization algorithm using
  reference-point based nondominated sorting approach, part ii: Handling
  constraints and extending to an adaptive approach.
\newblock \emph{IEEE Transactions on evolutionary computation}, 18\penalty0
  (4):\penalty0 602--622, 2013.

\bibitem[Johnson et~al.(1990)Johnson, Moore, and Ylvisaker]{johnson1990minimax}
Mark~E. Johnson, Leslie~M. Moore, and Donald Ylvisaker.
\newblock Minimax and maximin distance designs.
\newblock \emph{Journal of Statistical Planning and Inference}, 26\penalty0
  (2):\penalty0 131--148, 1990.
\newblock \doi{10.1016/0378-3758(90)90122-b}.

\bibitem[Jones(2001)]{jones2001taxonomy}
Donald~R. Jones.
\newblock A taxonomy of global optimization methods based on response surfaces.
\newblock \emph{Journal of Global Optimization}, 21\penalty0 (4):\penalty0
  345--383, 2001.
\newblock \doi{10.1023/A:1012771025575}.

\bibitem[Jones et~al.(1998)Jones, Schonlau, and Welch]{jones1998efficient}
Donald~R. Jones, Matthias Schonlau, and William~J. Welch.
\newblock Efficient global optimization of expensive black-box functions.
\newblock \emph{Journal of Global Optimization}, 13\penalty0 (4):\penalty0
  455--492, 1998.
\newblock \doi{10.1023/A:1008306431147}.

\bibitem[Knowles(2006)]{knowles2006parego}
Joshua Knowles.
\newblock Parego: A hybrid algorithm with on-line landscape approximation for
  expensive multiobjective optimization problems.
\newblock \emph{IEEE Transactions on Evolutionary Computation}, 10\penalty0
  (1):\penalty0 50--66, 2006.

\bibitem[Konakovic~Lukovic et~al.(2020)Konakovic~Lukovic, Tian, and
  Matusik]{konakoviclukovicDiversityGuidedMultiObjectiveBayesian2020}
Mina Konakovic~Lukovic, Yunsheng Tian, and Wojciech Matusik.
\newblock Diversity-{{Guided Multi-Objective Bayesian Optimization With Batch
  Evaluations}}.
\newblock In \emph{Advances in {{Neural Information Processing Systems}}},
  volume~33, pages 17708--17720. {Curran Associates, Inc.}, 2020.

\bibitem[Liang and Lai(2021)]{liang2021scalable}
Qiaohao Liang and Lipeng Lai.
\newblock Scalable bayesian optimization accelerates process optimization of
  penicillin production.
\newblock In \emph{NeurIPS 2021 AI for Science Workshop}, 2021.

\bibitem[Liao et~al.(2008)Liao, Li, Yang, Zhang, and
  Li]{liao2008multiobjective}
Xingtao Liao, Qing Li, Xujing Yang, Weigang Zhang, and Wei Li.
\newblock Multiobjective optimization for crash safety design of vehicles using
  stepwise regression model.
\newblock \emph{Structural and multidisciplinary optimization}, 35:\penalty0
  561--569, 2008.

\bibitem[Mockus et~al.(1978)Mockus, Tie{\v{s}}is, and
  {\v{Z}}ilinskas]{mockus1978application}
Jonas Mockus, Vytautas Tie{\v{s}}is, and Antanas {\v{Z}}ilinskas.
\newblock The application of {B}ayesian methods for seeking the extremum.
\newblock \emph{Towards Global Optimization}, 2:\penalty0 117--129, 1978.

\bibitem[Osyczka and Kundu(1995)]{osyczka1995new}
Andrzej Osyczka and Sourav Kundu.
\newblock A new method to solve generalized multicriteria optimization problems
  using the simple genetic algorithm.
\newblock \emph{Structural optimization}, 10:\penalty0 94--99, 1995.

\bibitem[Paria et~al.(2020)Paria, Kandasamy, and P{\'o}czos]{paria2020flexible}
Biswajit Paria, Kirthevasan Kandasamy, and Barnab{\'a}s P{\'o}czos.
\newblock A flexible framework for multi-objective bayesian optimization using
  random scalarizations.
\newblock In \emph{Uncertainty in Artificial Intelligence}, pages 766--776.
  PMLR, 2020.

\bibitem[Picheny(2014)]{picheny2014stepwise}
Victor Picheny.
\newblock A stepwise uncertainty reduction approach to constrained global
  optimization.
\newblock In \emph{Artificial intelligence and statistics}, pages 787--795.
  PMLR, 2014.

\bibitem[Picheny(2015)]{picheny2015multiobjective}
Victor Picheny.
\newblock Multiobjective optimization using gaussian process emulators via
  stepwise uncertainty reduction.
\newblock \emph{Statistics and Computing}, 25\penalty0 (6):\penalty0
  1265--1280, 2015.

\bibitem[Rahimi and Recht(2007)]{rahimi2007random}
Ali Rahimi and Benjamin Recht.
\newblock Random features for large-scale kernel machines.
\newblock \emph{Advances in neural information processing systems}, 20, 2007.

\bibitem[Rajaram et~al.(2020)Rajaram, Puranik, Renganathan, Sung,
  Pinon-Fischer, Mavris, and Ramamurthy]{rajaram2020deep}
Dushhyanth Rajaram, Tejas~G Puranik, Ashwin Renganathan, Woong~Je Sung,
  Olivia~J Pinon-Fischer, Dimitri~N Mavris, and Arun Ramamurthy.
\newblock Deep gaussian process enabled surrogate models for aerodynamic flows.
\newblock In \emph{AIAA Scitech 2020 Forum}, page 1640, 2020.

\bibitem[Rasmussen and Williams(2006)]{rasmussen:williams:2006}
C.~E. Rasmussen and C.~K.~I. Williams.
\newblock \emph{Gaussian Processes for Machine Learning}.
\newblock MIT Press, 2006.
\newblock \doi{10.7551/mitpress/3206.001.0001}.

\bibitem[Russo and Van~Roy(2014)]{russo2014learning}
Daniel Russo and Benjamin Van~Roy.
\newblock Learning to optimize via posterior sampling.
\newblock \emph{Mathematics of Operations Research}, 39\penalty0 (4):\penalty0
  1221--1243, 2014.

\bibitem[Sauer et~al.(2023)Sauer, Gramacy, and Higdon]{sauer2023active}
Annie Sauer, Robert~B Gramacy, and David Higdon.
\newblock Active learning for deep gaussian process surrogates.
\newblock \emph{Technometrics}, 65\penalty0 (1):\penalty0 4--18, 2023.

\bibitem[Snoek et~al.(2014)Snoek, Swersky, Zemel, and Adams]{snoek2014input}
Jasper Snoek, Kevin Swersky, Rich Zemel, and Ryan Adams.
\newblock Input warping for bayesian optimization of non-stationary functions.
\newblock In \emph{International conference on machine learning}, pages
  1674--1682. PMLR, 2014.

\bibitem[Srinivas et~al.(2009)Srinivas, Krause, Kakade, and
  Seeger]{srinivas2009gaussian}
Niranjan Srinivas, Andreas Krause, Sham~M. Kakade, and Matthias Seeger.
\newblock Gaussian process optimization in the bandit setting: No regret and
  experimental design.
\newblock In \emph{Proceedings of the International Conference on Machine
  Learning}, 2009.

\bibitem[Sun et~al.(2019)Sun, Gramacy, Haaland, Lu, and
  Hwang]{sun2019synthesizing}
Furong Sun, Robert~B Gramacy, Benjamin Haaland, Siyuan Lu, and Youngdeok Hwang.
\newblock Synthesizing simulation and field data of solar irradiance.
\newblock \emph{Statistical Analysis and Data Mining: The ASA Data Science
  Journal}, 12\penalty0 (4):\penalty0 311--324, 2019.

\bibitem[Suzuki et~al.(2020)Suzuki, Takeno, Tamura, Shitara, and
  Karasuyama]{suzuki2020multi}
Shinya Suzuki, Shion Takeno, Tomoyuki Tamura, Kazuki Shitara, and Masayuki
  Karasuyama.
\newblock Multi-objective bayesian optimization using pareto-frontier entropy.
\newblock In \emph{International Conference on Machine Learning}, pages
  9279--9288. PMLR, 2020.

\bibitem[Swersky et~al.(2013)Swersky, Snoek, and Adams]{swersky2013multi}
Kevin Swersky, Jasper Snoek, and Ryan~P Adams.
\newblock Multi-task bayesian optimization.
\newblock \emph{Advances in neural information processing systems}, 26, 2013.

\bibitem[Tanabe and Ishibuchi(2020)]{tanabe2020easy}
Ryoji Tanabe and Hisao Ishibuchi.
\newblock An easy-to-use real-world multi-objective optimization problem suite.
\newblock \emph{Applied Soft Computing}, 89:\penalty0 106078, 2020.

\bibitem[Thompson(1933)]{thompson1933likelihood}
William~R Thompson.
\newblock On the likelihood that one unknown probability exceeds another in
  view of the evidence of two samples.
\newblock \emph{Biometrika}, 25\penalty0 (3-4):\penalty0 285--294, 1933.

\bibitem[Tu et~al.(2022)Tu, Gandy, Kantas, and Shafei]{tu2022joint}
Ben Tu, Axel Gandy, Nikolas Kantas, and Behrang Shafei.
\newblock Joint entropy search for multi-objective bayesian optimization.
\newblock \emph{Advances in Neural Information Processing Systems},
  35:\penalty0 9922--9938, 2022.

\bibitem[Wang and Jegelka(2017)]{wang2017max}
Zi~Wang and Stefanie Jegelka.
\newblock Max-value entropy search for efficient {Bayesian} optimization.
\newblock In \emph{Proceedings of the 34th International Conference on Machine
  Learning}, volume~70, pages 3627--3635, 2017.

\bibitem[Wu and Frazier(2016)]{wu2016parallel}
Jian Wu and Peter Frazier.
\newblock The parallel knowledge gradient method for batch {Ba}yesian
  optimization.
\newblock In \emph{Advances in Neural Information Processing Systems}, pages
  3126--3134, 2016.

\bibitem[Wu and Frazier(2019)]{wu2019practical}
Jian Wu and Peter Frazier.
\newblock Practical two-step lookahead {Bayesian} optimization.
\newblock In \emph{Advances in Neural Information Processing Systems}, pages
  9810--9820, 2019.

\bibitem[Yang et~al.(2019{\natexlab{a}})Yang, Emmerich, Deutz, and
  B{\"a}ck]{yang2019efficient}
Kaifeng Yang, Michael Emmerich, Andr{\'e} Deutz, and Thomas B{\"a}ck.
\newblock Efficient computation of expected hypervolume improvement using box
  decomposition algorithms.
\newblock \emph{Journal of Global Optimization}, 75:\penalty0 3--34,
  2019{\natexlab{a}}.

\bibitem[Yang et~al.(2019{\natexlab{b}})Yang, Emmerich, Deutz, and
  B{\"a}ck]{yang2019multi}
Kaifeng Yang, Michael Emmerich, Andr{\'e} Deutz, and Thomas B{\"a}ck.
\newblock Multi-objective bayesian global optimization using expected
  hypervolume improvement gradient.
\newblock \emph{Swarm and evolutionary computation}, 44:\penalty0 945--956,
  2019{\natexlab{b}}.

\bibitem[Zhang et~al.(2009)Zhang, Liu, Tsang, and Virginas]{zhang2009expensive}
Qingfu Zhang, Wudong Liu, Edward Tsang, and Botond Virginas.
\newblock Expensive multiobjective optimization by moea/d with gaussian process
  model.
\newblock \emph{IEEE Transactions on Evolutionary Computation}, 14\penalty0
  (3):\penalty0 456--474, 2009.

\bibitem[Zhang and Golovin(2020)]{zhang2020random}
Richard Zhang and Daniel Golovin.
\newblock Random hypervolume scalarizations for provable multi-objective black
  box optimization.
\newblock In \emph{International Conference on Machine Learning}, pages
  11096--11105. PMLR, 2020.

\bibitem[Zhou et~al.(2012)Zhou, Zhang, and Zhang]{zhou2012multiobjective}
Aimin Zhou, Qingfu Zhang, and Guixu Zhang.
\newblock A multiobjective evolutionary algorithm based on decomposition and
  probability model.
\newblock In \emph{2012 IEEE Congress on Evolutionary Computation}, pages 1--8.
  IEEE, 2012.

\bibitem[Zitzler et~al.(2000)Zitzler, Deb, and Thiele]{zitzler2000comparison}
Eckart Zitzler, Kalyanmoy Deb, and Lothar Thiele.
\newblock Comparison of multiobjective evolutionary algorithms: Empirical
  results.
\newblock \emph{Evolutionary computation}, 8\penalty0 (2):\penalty0 173--195,
  2000.

\bibitem[Zuluaga et~al.(2013)Zuluaga, Sergent, Krause, and
  P{\"u}schel]{zuluaga2013active}
Marcela Zuluaga, Guillaume Sergent, Andreas Krause, and Markus P{\"u}schel.
\newblock Active learning for multi-objective optimization.
\newblock In \emph{International Conference on Machine Learning}, pages
  462--470. PMLR, 2013.

\end{thebibliography}
\end{document}